\documentclass[10pt]{amsart}
\usepackage{amsfonts,amsmath,amssymb,color,amscd,amsthm}
\usepackage[T1]{fontenc} 

\usepackage[francais,english]{babel} 
\usepackage[all]{xy} 
\usepackage{tikz-cd}
\usepackage{pgf,tikz} 
\usepackage{pgfplots}
\usepgfplotslibrary{patchplots}
 
\pgfplotsset{
        compat=1.11,
        My Line Style/.style={
            smooth,
            thick,
            samples=400,
        },
    }

\usepackage{stmaryrd}
\usepackage[normalem]{ulem} 
\newcommand\redsout{\bgroup\markoverwith{\textcolor{red}{\rule[1.5ex]{2pt}{0.4pt}}}\ULon}

\usepackage[shortlabels]{enumitem}
\setlist[1]{wide}
\setlist[2]{leftmargin=15mm}
\setlist[enumerate]{label=\rm{(\arabic*)}}
\setlist[enumerate,2]{label=\rm({\it\roman*}), }
\setlist[itemize]{label=\raisebox{0.25ex}{\tiny$\bullet$}}

\usepackage[backref, colorlinks, linktocpage, citecolor = blue, linkcolor = blue]{hyperref}

\newcommand\NN{{\mathbb N}}

\newcommand\RR{{\mathbb R}}
\renewcommand\AA{{\mathbb A}}
\newcommand\ZZ{{\mathbb Z}}
\newcommand\CC{{\mathbb C}}

\renewcommand\r{{\mathbb R}}
\newcommand\Mat{\mathrm{Mat}}

\newcommand\tr{\hbox to 1mm  {${}^t \!  $} }

\renewcommand\k{\mathbf{ \textbf k}}

\DeclareMathOperator{\Aut}{Aut}

\DeclareMathOperator{\Aff}{Aff}
\DeclareMathOperator{\TEnd}{TEnd}
\DeclareMathOperator{\TAut}{TAut}
\DeclareMathOperator{\EEnd}{EEnd}
\DeclareMathOperator{\EAut}{EAut}

\DeclareMathOperator{\GL}{GL}

\DeclareMathOperator{\End}{End}

\DeclareMathOperator{\id}{id}

\DeclareMathOperator{\Spec}{Spec}

\DeclareMathOperator{\Sym}{Sym}

\newtheorem{lemma}{Lemma}[subsection]
\newtheorem{theorem}[lemma]{Theorem}

\newtheorem{corollary}[lemma]{Corollary}
\newtheorem{proposition}[lemma]{Proposition}
\newtheorem{question}[lemma]{Question}

\newtheorem*{corollary*}{Corollary}
\newtheorem*{corollary**}{Corollary}

\newtheorem{maintheorem}{Theorem}
\newtheorem{mainproposition}{Proposition}

\theoremstyle{definition}
\newtheorem{definition}[lemma]{Definition}
\newtheorem{notation}[lemma]{Notation}

\theoremstyle{remark}
\newtheorem{remark}[lemma]{Remark}
\newtheorem{example}[lemma]{Example}

\newcommand{\Z}{\mathbb{Z}}

\newcommand{\A}{\mathbb{A}}

\makeatletter
\ifcsname phantomsection\endcsname
    \newcommand*{\qrr@gobblenexttocentry}[5]{}
\else
    \newcommand*{\qrr@gobblenexttocentry}[4]{}
\fi
\newcommand*{\addsubsection}{
    \addtocontents{toc}{\protect\qrr@gobblenexttocentry}
    \subsection}
\makeatother

\newcommand{\set}[2]{\left\{\,#1 \ | \ #2\,\right\}}
\newcommand{\Bigset}[2]{\left\{\,#1 \ \Big| \ #2\,\right\}}

\subjclass[2010]{14R10, 37F10}

\author{J\'er\'emy Blanc}
\address{J\'er\'emy Blanc, Universit\"{a}t Basel, Departement Mathematik und Informatik, Spiegelgasse $1$, CH-$4051$ Basel, Switzerland}
\email{jeremy.blanc@unibas.ch}

\author{Immanuel van Santen}
\address{Immanuel van Santen, Universit\"{a}t Basel, Departement Mathematik und Informatik, Spiegelgasse $1$, CH-$4051$ Basel, Switzerland}
\email{immanuel.van.santen@math.ch}

\title[Affine-triangular automorphisms]{Dynamical degrees 
	of affine-triangular automorphisms of affine spaces}
\begin{document}
\maketitle
\begin{abstract}
	We study the possible dynamical degrees of automorphisms of the affine space $\mathbb{A}^n$. In dimension $n=3$, we determine all dynamical degrees arising from the composition of an affine automorphism with a triangular one. This generalises the easier case of shift-like automorphisms which can be studied in any dimension. We also prove that each weak Perron number is the dynamical degree of an affine-triangular
	automorphism of the affine space $\mathbb{A}^n$ for some $n$, and we
	give the best possible $n$ for quadratic integers, 
	which is either $3$ or $4$.
\end{abstract}

\tableofcontents

\section{Introduction}
\subsection{Dynamical degrees of polynomial endomorphisms}
	In this text, we work over an arbitrary field $\k$. For each $n\ge 1$, recall that an endomorphism $f\in \End(\A^n)$ of $\A^n=\A^n_\k$ is given by
		\[f\colon (x_1,\ldots,x_n)\mapsto (f_1(x_1,\ldots,x_n),\ldots ,f_n(x_1,\ldots,x_n))\]
		where $f_1,\ldots,f_n\in \k[x_1,\ldots,x_n]$. To simplify the notation, we often write $f=(f_1,\ldots,f_n)$ and thus identify $\End(\A^n)$ with $(\k[x_1,\ldots,x_n])^n$.

		 The degree of an  endomorphism $f=(f_1,\ldots,f_n)$, denoted by $\deg(f)$, is defined to be $\deg(f)=\max(\deg(f_1),\ldots,\deg(f_n))$. The set $\End(\A^n)$ of endomorphisms of $\A^n$ is a monoid, for the composition law, and the subset of invertible elements is the group $\Aut(\A^n)$ of automorphisms of $\A^n$.
		
		The dynamics of endomorphisms of $\A^n$, specially in the case of the ground field $\k=\mathbb{C}$, was studied intensively in the last decades, see for instance \cite{FornaessWu,Sibony1999,Mae2000,BonifantFornaess,Maequadratic,Maegoldenratio,Guedj02,GuedjSibony,Guedj04,Ueda,FavreJonsson,JonssonWulcan,XieMordellLang,DeLe2018Dynamics-of-a-fami}.
		For each dominant endomorphism $f\in \End(\A^n)$, the \emph{$($first$)$ dynamical degree} is defined as the real number
		\[
			\lambda(f)=\lim\limits_{r \to \infty} \deg(f^r)^{\frac{1}{r}}\in \r_{\ge 1}
		\] 
		(the limit exists by Fekete's subadditivity 
	Lemma, see Lemma~\ref{lem.dynamical_degree_exists}). If $f\in \End(\A^1)$ or $f\in \Aut(\A^2)$, then $\lambda(f)$ is an integer, but in higher dimensions, it can be quite complicated to understand the possible dynamical degrees. In \cite[Corollary 3]{DaFa2020Spectral-interpret}, the authors conjecture that $\lambda(f)$ is an algebraic integer of degree $\le n$, and of degree $\le n-1$ if $f\in \Aut(\A^n)$, a conjecture proven until now only for $n\le 2$.
	
	In this article, we study some particular family of automorphisms of $\A^n$, that we call \emph{affine-triangular}. These are compositions consisting of one
	affine automorphism and one
	triangular automorphism
	(see Definition~\ref{definition:triangularaffine}) below.
	Our two main results are Theorem~\ref{Thm:AffTriang3} and Theorem~\ref{Theorem:PerronAuto} below:
	
\begin{maintheorem}\label{Thm:AffTriang3}
For each field $\k$ and each integer $d\ge 2$, the set of dynamical degrees of all affine-triangular automorphisms of $\A^3$ of degree $\le d$ is equal to
\[
	\left.\left\{\frac{a+\sqrt{a^2+4bc}}{2}\right| (a,b,c)\in \NN^3, a+b\le d, c\le d\right\}\setminus \{0\}.
\]
Moreover, for $a,b,c\in \NN$ such that $\lambda=\frac{a+\sqrt{a^2+4bc}}{2}\not=0$, 
the dynamical degre $\lambda$ is achieved by either of the automorphisms
\[
	(x_3 + x_1^ax_2^b,x_2+x_1^c,x_1) \text{ and }(x_3+x_1^ax_2^{bc},x_1,x_2) \, .
\] 
\end{maintheorem}

Using Theorem~\ref{Thm:AffTriang3}, we prove in \cite[Theorem 2]{BlSa2019Automorphisms-of-t} that the set of dynamical degrees of all automorphisms of degree $3$ of $\A^3$ is equal to   
\[
	\left\{1 , \sqrt{2} , \ \frac{1+\sqrt{5}}{2} , \ \sqrt{3} , 2 , \ \frac{1 + \sqrt{13}}{2} , \
		1 + \sqrt{2} , \ \sqrt{6} , \ \frac{1 + \sqrt{17}}{2} , \ \frac{3+\sqrt{5}}{2} , \ 1+ \sqrt{3} , \ 3 \right\} \, .
\]
Note that $\frac{3+\sqrt{5}}{2}$ is the only number that does not belong to the list in Theorem~\ref{Thm:AffTriang3} and thus it is the dynamical degree of an automorphism of degree $3$
of $\AA^3$ that is not conjugate to an affine-triangular automorphism of any degree.		
	
For the next theorem, we recall the definition of $($weak$)$-Perron numbers (see Theorem~\ref{LindTheorem}  for some equivalent characterisations).
	\begin{definition}	
	\label{def.Perron_number}
	A \emph{Perron number} (respectively \emph{weak Perron number}) is a real number 
	$\lambda \geq 1$ that is an algebraic integer such that all other Galois conjugates $\mu\in \CC$ satisfy $\lvert \mu\rvert<\lambda$ (respectively $\lvert \mu\rvert\le \lambda$). 
	\end{definition}
	\begin{maintheorem}\label{Theorem:PerronAuto}
	Each weak-Perron number $\lambda$ is the dynamical degree of an affine-triangular automorphism of $\A^{n}$ for some integer $n$. Moreover:
	\begin{enumerate}[leftmargin=*]
	\item \label{Theorem:PerronAuto_1} If $\lambda>1$ is an integer, the least $n$ possible is $2$.
	\item \label{Theorem:PerronAuto_2} If $\lambda$ is a quadratic integer and its conjugate is negative, the least possible $n$ is $3$.
	\item \label{Theorem:PerronAuto_3}
	If $\lambda$ is a quadratic integer and its conjugate is positive, the least possible $n$ is $4$.
	\end{enumerate}
	\end{maintheorem}
	Note that Statement \ref{Theorem:PerronAuto_1} in Theorem~\ref{Theorem:PerronAuto} is well-known, 
	as $\{\lambda(f)\mid f\in \Aut(\A^2)\}=\Z_{\ge 1}$. We 
	include it to emphasise the relation between the degree of the weak-Perron numbers and 
	the possible $n$.
	In view of the above theorems and of the techniques developped in this text, it is natural to ask the following 
	\begin{question}
	Is every dynamical degree of any element of $\End(\A^n)$ $($respectively $\Aut(\A^n))$ equal to a 
	weak Perron number of degree $\le n$ $($respectively of degree $\le n-1)$?
	\end{question}
	As already mentioned above, a positive answer to this question, where ``weak Perron number'' is replaced by ''algebraic integer'', was conjectured in the recent preprint
	\cite[Corollary 3]{DaFa2020Spectral-interpret} 
	(that appeared after we asked the above question in a first version of this text).
	In \cite{DaFa2020Spectral-interpret}, it is also proven that the dynamical degree of any element
	in $\Aut(\AA^3)$ is an algebraic number of degree at most six.
	More generally they prove that the dynamical degree of any element
	of $\End(\AA^n)$ is an algebraic number of degree at most $n$
	in case the square of the first dynamical degree is bigger than the second dynamical degree of $f$
	\cite[Theorem 2]{DaFa2020Spectral-interpret}.

Theorem~\ref{Thm:AffTriang3} shows in particular that the dynamical degree of every affine-triangular automorphism of $\A^3$ is equal to the dynamical degree of a shift-like automorphism. However, for each $d\ge 3$ the set of dynamical degrees of all affine-triangular automorphisms of $\A^3$ of degree $d$ strictly contains the set 
of dynamical degrees of all shift-like automorphisms of $\A^3$ of degree $d$. Indeed, the latter set of
dynamical degrees consists of the numbers $(a + \sqrt{a^2 + 4d-4a})/2$ where $0 \leq a \leq d$
and does not contain $(1 + \sqrt{1 + 4d})/2$ , which is the dynamical degree of the affine-triangular 
automorphism $(x_3 + x_1 x_2,x_2+x_1^d,x_1)$, see Corollary~\ref{Cor:ToThm1}.
\[
\begin{array}{|c|c|c|}
\hline
 & \text{dynamical degrees of shift-like} &\text{dynamical degrees of affine-triangular} \\
d& \text{automorphisms of }\A^3\text{ of degree }& \text{automorphisms of }\A^3\text{ of degree }d\\
&d \text{ not appearing in degree }< d& \text{ not appearing in degree }< d \\
\hline
1 & \{1\} & \{1\}\\
2 & \{\sqrt{2},\frac{1+\sqrt{5}}{2},2\} & \{\sqrt{2},\frac{1+\sqrt{5}}{2},2\}\\
3 & \{\sqrt{3},1+\sqrt{2},3\} & \{\sqrt{3},\frac{1+\sqrt{13}}{2},1+\sqrt{2},\sqrt{6},\frac{1+\sqrt{17}}{2},1+\sqrt{3},3\} \\ 
4 & \{\frac{1+\sqrt {13}}{2},1+\sqrt{3},\frac{3+\sqrt {13}}{2},4\}
& \begin{array}{c}\{2\sqrt {2},1+\sqrt{5},\frac{3+\sqrt {13}}{2},\frac{1+\sqrt {33}}{2},2
\sqrt {3},\frac{1+\sqrt {37}}{2},\\
\frac{3+\sqrt {17}}{2},1+\sqrt {7},\frac{3+\sqrt {21}}{2},4\}\end{array}\\
\hline
\end{array}
\]

Note that $2\sqrt{2}$ and $\sqrt{3}$ appear as dynamical degrees of affine-triangular automorphisms in degree 
$4$ and $3$, respectively (and not smaller), even if $2\sqrt{2}<3$ and $\sqrt{3} < 2$. 
Similarly, for each prime $p$, the number $\sqrt{p}$ is the dynamical degree of a 
shift-like automorphism of degree $p$, but it is not the dynamical degree of an affine-triangular automorphism of degree $< p$.

\bigskip
	
		\subsection{Dynamical degrees of affine-triangular automorphisms in higher dimensions}
In dimension $n\ge 4$, we are not able to compute all dynamical degrees of all affine-triangular automorphisms, but can get some large families. The case of shift-like automorphisms is covered by our method, and we retrieve a proof of the result of Mattias Jonsson (Proposition~\ref{Prop:Jonsson}), but we can also study wider classes. We give the dynamical degrees of all permutation-elementary automorphisms (a family that strictly includes the shift-like automorphisms) in $\S\ref{Sec:PermEl}$ (especially Proposition~\ref{Prop:PermutationElementary}) and also give the dynamical degrees of other affine-triangular automorphisms. In particular, we show that in any dimension $n\ge 4$, there are affine-triangular  automorphisms of $\A^n$ whose dynamical degrees are not those of a shift-like automorphisms or more generally of a permutation-elementary automorphisms, contrary to the case of dimension $n\le 3$. The reason is that
dynamical degrees of shift-like automorphisms are special kinds of weak Perron numbers. Indeed, they are positive real numbers that are roots of a monic integral polynomial where all 
	coefficients (except the first one) are non-positive. These numbers are called \emph{Handelman numbers} in \cite{Bassino} (see especially \cite[Lemma~10]{Bassino}) and they have 
	no other positive real Galois conjugates (Lemma~\ref{Lemm:Handelman}). 
	This implies that Handelman numbers are weak Perron numbers (see Corollary~\ref{Cor:Handelman_implies_weak_Perron}).
	Theorem~\ref{Thm:AffTriang3} implies that the dynamical degree of an affine-triangular automorphism of $\A^3$ is a Handelman number (and the same holds for all automorphisms of $\A^1$ and $\A^2$), but for any $n\ge 4$, there are affine-triangular automorphisms of $\A^n$ whose dynamical degrees are not Handelman numbers. This follows in particular from Theorem~\ref{Theorem:PerronAuto}, applied to any weak Perron quadratic integer with a positive conjugate, for instance to $(3+\sqrt{5})/2$. We can also apply Theorem~\ref{Theorem:PerronAuto} to 
	weak Perron numbers of arbitrary large degree.
	
	\subsection{Results in the literature on dynamical degrees of endomorphisms of $\AA^n$}
	\label{SubSec:Previous}
	
	Let us recall what is known on the dynamical degrees of elements of $\End(\A^n)$. 
	\begin{enumerate}
	\item
	The case where $n=1$ is obvious: in this case we have $\lambda(f)=\deg(f)$, so each dynamical degree is an integer, which is moreover equal to $1$ in the case of automorphisms.
	\item
	When $n=2$, the case of automorphisms follows from the Jung-van der Kulk Theorem \cite{Ju1942Uber-ganze-biratio,Ku1953On-polynomial-ring}: every dynamical degree is an integer, 
	as $\deg(f^r)=\deg(f)^r$ for each $r$, when $f$ is taken to be cyclically reduced 
	(this is explained in Corollary~\ref{Cor:Jung2} below, or 
	in~\cite[Proposition~3]{Furter99}).
	The set of all dynamical degrees of quadratic endomorphisms of $\A^2_\CC$ is equal to $\{1,\sqrt{2}, (1+\sqrt{5})/2,2\}$ by \cite[Theorem 2.1]{Guedj04}. 
	Moreover, the dynamical degree of every element of 
	$\End(\A^2_\mathbb{C})$ 
	is a quadratic integer, by \cite[Theorem A']{FavreJonsson07}. 
	\item
	The case of dimension $n\ge 3$ is open in general: there is for the moment no hope 
	of classifying all dynamical degrees, even when studying only automorphisms.
	
	The set of dynamical degrees of all automorphisms of $\A^3_\CC$ of degree $2$ is equal to $\{1,\sqrt{2},(1+\sqrt{5})/2,2\}$ by \cite[Theorem~3.1]{Maequadratic} 
	(and the same holds over any field \cite[Theorem 2]{BlSa2019Automorphisms-of-t}).
	\end{enumerate}
	
	\bigskip
	
	Apart from the above classification results, two natural families are also known: the \emph{monomial endomorphisms} and the \emph{shift-like automorphisms}.
	\begin{enumerate}[$(A)$]
	\item
	A \emph{monomial endomorphism} of $\A^n$ is an endomorphism of the form $f=(f_1,\ldots,f_n)$, where each $f_i$ is a monomial. When we write $f_i=\alpha_i x_1^{m_{i,1}}\cdots x_n^{m_{i,n}}$ with $\alpha_i\in \k^*$ and $m_{i,1},\ldots,m_{i,n}\in \NN$ and assume that $f$ is dominant, then
	the dynamical degree of $f$ is the spectral radius of the corresponding matrix $M=(m_{i,j})_{i,j=1}^n\in \Mat_n(\NN)$. This classical result is proven again in Corollary~\ref{Coro:MonomialEndo} below. The numbers arising this way are the weak Perron numbers 
	(see Theorem~\ref{LindTheorem}).
	\item
	 For each $n\ge 1$, a \emph{shift-like automorphism} of $\A^{n+1}$ is an automorphism of the form $(x_{n+1}+p(x_1,\ldots,x_{n}),x_1,\ldots,x_{n})$ for some polynomial $p\in \k[x_1,\ldots,x_n]$.  
	 These are  particular examples of affine-triangular automorphisms. The dynamics of such automorphisms have been studied in various texts (see for instance \cite{BedfordPambuccian,Mae2000,Maegoldenratio,Ueda,SayaniKaushal}). The dynamical degrees of shift-like automorphisms are known, by a result of Mattias Jonsson (see Proposition~\ref{Prop:Jonsson} below).
	For a proof of this result, together with a generalisation, see $\S\ref{Sec:PermEl}$.
	
	\end{enumerate}

\subsection{Description of the techniques associated to degrees}\label{Sec:Algo}
	In the rest of this introduction, we describe the main technique that we introduce in order 
	to compute dynamical degrees of endomorphisms of $\AA^n$. 
	This is related to degree functions (or monomial valuations), and may be applied to endomorphisms of $\A^n$, not only affine-triangular automorphisms. We also give an outline of the 
	whole article.

\begin{definition}\label{Defi:degreeMu}
For each $\mu=(\mu_1,\ldots,\mu_n)\in (\mathbb{R}_{\ge 0})^n\setminus \{0\}$, we define a \emph{degree function}
$\deg_\mu\colon \k[x_1,\ldots,x_n]\to \r_{\ge 0}\cup \{-\infty\}$
  by $\deg_\mu(0)=-\infty$ and
  \[\deg_\mu(\sum_{(a_1,\ldots,a_n)\in \mathbb{N}^n} \underbrace{c_{(a_1,\ldots,a_n)}}_{\in \k}\cdot  x_1^{a_1}x_2^{a_2}\cdots x_n^{a_n})=\max\left.\left\{\sum\limits_{i=1}^n a_i \mu_i\right| c_{(a_1,\ldots,a_n)}\not=0\right\}.\]

 We say that a polynomial $p\in\k[x_1,\ldots,x_n]$ is \emph{$\mu$-homogeneous of degree $\theta\in \r$} if $p$ is a finite sum of monomials $p_i$ with $\deg_\mu(p_i)=\theta$ for each $i$ $($where the zero polynomial is $\mu$-homogeneous of degree $\theta$ for each $\theta)$.
 
 We can then write every element $q\in\k[x_1,\ldots,x_n]\setminus \{0\}$ uniquely as  
 \[
 	q=\sum\limits_{\theta\in \r_{\ge 0}} q_\theta \, ,
 \]
 where each $q_\theta\in \k[x_1,\ldots,x_n]$ is $\mu$-homogeneous of degree $\theta$ (and only finitely many $q_{\theta}$ are non-zero). We then say
 that $q_\theta$ is the \emph{$\mu$-homogeneous part of $q$ of degree $\theta$}. 
 The \emph{$\mu$-leading part} of $q$ is the $\mu$-homogeneous part of $q$ of degree $\deg_\mu(q)$.
\end{definition}

\begin{remark}\label{RemVal}
	Note that if $\mu \in (\RR_{\geq 0})^n \setminus \{0\}$, then 
	\[\k(x_1, \ldots, x_n) \to \RR \cup \{\infty\} \, , \quad f/g \mapsto \deg_{\mu}(g)-\deg_{\mu}(f)\]
	is a valuation in the sense of \cite[p.75]{Ma1989Commutative-ring-t} where  $\k(x_1, \ldots, x_n)$ denotes the field of rational functions in $x_1, \ldots, x_n$ over $\k$. 
	Such valuations are often called ``monomial valuations'' in the literature.
\end{remark}

\begin{definition}\label{Defi:degmuendo}
  Let $\mu=(\mu_1,\ldots,\mu_n)\in (\mathbb{R}_{\ge 0})^n\setminus \{0\}$. For each $f=(f_1,\ldots,f_n)\in \End(\A^n) \setminus \{0\}$ we denote the \emph{$\mu$-degree of $f$} by 
  \[
  	\deg_\mu(f)=\inf\{\theta\in \r_{\ge 0}\mid \deg_\mu(f_i)\le \theta \mu_i\text{ for each }i\in \{1,\ldots,n\}\}
  \]
  and we say that $\deg_\mu(f)=\infty$ if the above set is empty.
  
  We moreover say that $f$ is \emph{$\mu$-algebraically stable} if $\deg_\mu(f)<\infty$ and $\deg_\mu(f^r)=\deg_\mu(f)^r$ for each $r\ge 1$.
\end{definition}

\begin{remark}
If $\mu=(1,\ldots,1)$, then $\deg_\mu(f)=\deg(f)$ is the standard degree and the notion of being $\mu$-algebraically stable is the standard notion of ``algebraically stable'', studied  for instance in~\cite{GuedjSibony,Bisi,Blanc2016}. The fact of being algebraically stable can be interpreted geometrically by looking at the behaviour of the endomorphism at infinity: \cite[Corollary 2.16]{Blanc2016}.
\end{remark}

In order to compute the dynamical degree of an endomorphism $f \in \End(\AA^n)$,
the following endomorphism associated to $f$ will be of great importance for us:

\begin{definition}\label{Def:LeadingPart}
	Let $f=(f_1,\ldots,f_n)\in \End(\A^n)$ be a dominant endomorphism, let $\mu=(\mu_1,\ldots,\mu_n)\in (\mathbb{R}_{\ge 0})^n$ be such that $\deg_\mu(f)=\theta<\infty$. We define the \emph{$\mu$-leading part of $f$} to be the endomorphism $g=(g_1,\ldots,g_n)\in \End(\A^n)$, where $g_j\in \k[x_1,\ldots,x_n]$ is the $\mu$-homogeneous part of $f_j$ of degree $\theta \mu_j$ for each $j\in \{1,\ldots,n\}$.
\end{definition}

The degree functions are studied in \S\ref{Sec:IneqAndProofOfProp:Dynamicaldegreeestimat}. Basic properties are given in   $\S\ref{SubSec:DegreeFunctions}$, and the relation with $\mu$-homogeneous endomorphisms is given in $\S\ref{SubSec:HomEndo}$ (we explain in particular when $\deg_\mu(f)=\infty$ in Lemma~\ref{Lem:HomDecEnd}). In $\S\ref{SubSec:IneqDegrees}$, we explain how degree functions allow us to give an estimate on the dynamical degrees, and sometimes to compute it exactly. In particular, we prove the following result (at the end of $\S\ref{SubSec:IneqDegrees}$).
\begin{mainproposition}
\label{Prop:Dynamicaldegreeestimate}
Let $f=(f_1,\ldots,f_n)\in \End(\A^n)$ be a dominant endomorphism. For each 
$\mu=(\mu_1,\ldots,\mu_n)\in (\mathbb{R}_{> 0})^n$  the following hold:
\begin{enumerate}[leftmargin=*]
\item\label{thetafinite} $\theta: = \deg_\mu(f) < \infty$,
\item\label{lambdasmaller}
The dynamical degree of $f$ satisfies $1 \leq \lambda(f)\le \theta$.
\item\label{lambdaequal}
Let $g \in \End(\AA^n)$ be the $\mu$-leading part of $f$. If $\theta > 1$, then
\[
	\lambda(f)=\theta
	\Leftrightarrow f \text{ is $\mu$-algebraically stable}\Leftrightarrow 
	g^r\not=0 \text{ for each }r\ge 1.
\]
\end{enumerate}
\end{mainproposition}
\begin{remark}
Let $\mu=(1,\ldots,1)$. In this case, the $\mu$-degree is the classical degree and  Proposition~\ref{Prop:Dynamicaldegreeestimate}\ref{lambdasmaller} is the classical inequality $\lambda(f)\le \deg(f)$.
\end{remark}
\begin{remark}
Proposition~\ref{Prop:Dynamicaldegreeestimate} is false when we apply it 
to $\mu\in  (\mathbb{R}_{\ge 0})^n\setminus \{0\}$. For instance, if $f=(x_1,x_2^2)$, $\mu=(1,0)$,
then $\deg_\mu(f)=1$ but $1<\lambda(f)=2$.
\end{remark}

To apply Proposition~\ref{Prop:Dynamicaldegreeestimate} 
to compute the dynamical degree, we need to find some eigenvectors and eigenvalues. This is done here by looking at monomial maps associated to endomorphisms in $\End(\AA^n)$. These behave quite well with respect to degree functions (see Corollary~\ref{Coro:MonomialEndo}).

\begin{definition}\label{Defi:Matrixcontained}
Let $f=(f_1,\ldots,f_n)\in \End(\A^n)$ be an endomorphism such that $f_i\not=0$ for each $i$. We will say that a square matrix $M=(m_{i,j})_{i,j=1}^n\in \Mat_n(\NN)$ is \emph{contained in $f$} if for each $i\in \{1,\ldots,n\}$, the coefficient of the monomial $\prod_{j=1}^n x_j^{m_{i,j}}$ in $f_i$ is nonzero. The set of matrices 
that are contained in $f$ is then finite and non-empty.

The \emph{maximal eigenvalue of $f$} is defined to be  
\[
	\theta=\max \set{ \lvert \xi\rvert \in \r}{\xi \text{ is an eigenvalue of a matrix that is contained in $f$}} \, .
\] 
An element $\mu=(\mu_1,\ldots,\mu_n)\in (\mathbb{R}_{\ge  0})^n \setminus \{ 0 \}$ 
is a \emph{maximal eigenvector of $f$} if $\deg_\mu(f_i) = \theta \mu_i$ for each $i\in \{1,\ldots,n\}$. 
In particular, we then get $\deg_{\mu}(f) = \theta < \infty$.
\end{definition}

It often happens that we cannot apply Proposition~\ref{Prop:Dynamicaldegreeestimate} to compute the dynamical degree, but that we can do it by allowing $\mu$ to have some coordinates, but not all, to be equal to zero. In fact, the following generalization of Proposition~\ref{Prop:Dynamicaldegreeestimate}
is our main tool to compute dynamical degrees:

\begin{mainproposition}\label{Prop:MonomialEig}
Let $f=(f_1,\ldots,f_n)\in \End(\A^n)$ be a dominant endomorphism with maximal eigenvalue $\theta$. 
Then the following holds:
\begin{enumerate}[leftmargin=*]
\item
\label{Prop:MonomialEig1}
There exists a maximal eigenvector 
$\mu=(\mu_1,\ldots,\mu_n)\in (\mathbb{R}_{\ge  0})^n \setminus \{0\}$ of $f$.
\item
\label{Prop:MonomialEig2}
We have $1\le \lambda(f)\le \theta \le \deg(f)$.
\item
\label{Prop:MonomialEig3}
For each maximal eigenvector $\mu$ of $f$, we have $\theta=\deg_\mu(f)$, and the following hold:
\begin{enumerate}
\item\label{MonomialEig3i}
If $f$ is $\mu$-algebraically stable, then $\lambda(f)=\theta$.
\item\label{MonomialEig3ii}
If $\lambda(f)=\theta$,  $\theta>1$ and $\mu\in (\r_{>0})^n$, then $f$ is $\mu$-algebraically stable.
\item\label{MonomialEig3iii}
Let $g \in \End(\A^n)$ be the $\mu$-leading part of $f$. If $\theta > 1$, then
$f$ is $\mu$-algebraically stable if and only if 
for each $r \geq 1$ there is $i \in \{1, \ldots, n\}$ with $\mu_i > 0$ and such that the $i$-th component of 
$g^r$ is non-zero.
\end{enumerate}
\end{enumerate}
\end{mainproposition}

\begin{remark}
In Proposition~\ref{Prop:MonomialEig}\ref{Prop:MonomialEig1}, there are examples with no possibility for $\mu$ to be in $(\mathbb{R}_{>0})^n$, as the examples $f = (x_1, x_2^2)\in \End(\A^2)$ or $f=(x_1,x_3,x_2+x_3^2)\in \Aut(\A^3)$ show. Hence, Proposition~\ref{Prop:Dynamicaldegreeestimate} cannot be directly applied in order to prove Proposition~\ref{Prop:MonomialEig}. However, if some coordinates of $\mu$ are zero, then a linear projection is preserved (this follows from Lemma~\ref{Lem:HomDecEnd}, see also Corollary~\ref{Coro:Computelambda}). To prove Proposition~\ref{Prop:MonomialEig}, we will use Lemma~\ref{Lemma:Dynamicaldegreeestimate}, that is a version of Proposition~\ref{Prop:Dynamicaldegreeestimate} that also works for $\mu\in (\mathbb{R}_{\ge0})^n\setminus \{0\}$.
\end{remark}
\begin{remark}
 The implication of Proposition~\ref{Prop:MonomialEig}\ref{Prop:MonomialEig3}\ref{MonomialEig3i} 
 is not an equivalence, as we show in Example~\ref{ExamplePropBImpNotequi} below.
\end{remark}

The proof of Proposition~\ref{Prop:MonomialEig} is given in Section~\ref{Sec:MatrixEndoProofProp:MonomialEig}. 
For each dominant endomorphism
$f\in \End(\A^n)$, Proposition~\ref{Prop:MonomialEig}\ref{Prop:MonomialEig1} gives the existence of a maximal eigenvector~$\mu$. Moreover, Proposition~\ref{Prop:MonomialEig}\ref{Prop:MonomialEig3} shows that if $f$ is $\mu$-algebraically stable then $\lambda(f)$ is equal to the maximal eigenvalue $\theta$ of $f$. We will use this to compute the dynamical degree of many endomorphisms of $\A^n$.

The following result allows to compute all dynamical degrees of permutation-elementary endomorphism of $\A^n$, and generalises in particular Proposition~\ref{Prop:Jonsson}. 
Its proof is given in $\S\ref{Sec:PermEl}$:
\begin{mainproposition}\label{Prop:PermElementary}
Let  $f\in \Aut(\A^n)$  be a permutation-elementary automorphism. If the maximal eigenvalue $\theta$ of $f$ is bigger than $1$, there exists a maximal eigenvector $\mu$ of $f$ such that $f$ is $\mu$-algebraically stable. In particular, the dynamical degree $\lambda(f)$ is equal to the maximal eigenvalue $\theta$ of $f$, which is a Handelman number.
\end{mainproposition} 

Proposition~\ref{Prop:PermElementary} is false if we replace ``permutation-elementary'' by ``permutation-triangular''
 (see Example~\ref{Exa:PermTriangularNotAlgStab} for examples in dimension~$3$). We can however obtain the following result, which is proven in $\S\ref{Sec:AffineTriangularA3}$:

\begin{mainproposition}\label{Prop:AffineTriangularA3}
	Every affine-triangular automorphism  $f\in \Aut(\A^3)$  is conjugate to a permutation-triangular
	automorphism $f'\in \Aut(\A^3)$ such that $\deg(f')\le \deg(f)$ and such that $f'$ has
	the following property: either the maximal eigenvalue $\theta$ of $f'$ is equal to $1$, or $f'$ is
	$\mu$-algebraically stable for each maximal eigenvector $\mu$. In particular, the dynamical degrees
	$\lambda(f)$ and $\lambda(f')$ are equal to the maximal eigenvalue $\theta$ of $f'$, 
	which is a Handelman number.
\end{mainproposition}

The proof of Theorem~\ref{Thm:AffTriang3} is given at the end of $\S\ref{Sec:AffineTriangularA3}$, directly after proving Proposition~\ref{Prop:AffineTriangularA3}, as it follows almost directly from this result.
We use these results in $\S\ref{SubSec:WeakPerron}$, to prove Theorem~\ref{Theorem:PerronAuto}.
 
\addsubsection*{Acknowledgements}
The authors thank the referee for his careful reading and helpful suggestions. We thank 
Jean-Philippe Furter and Pierre-Marie Poloni for helpful discussions on dynamical degrees of automorphisms of $\A^3$ and Christian Urech for indicating us the result of Mattias Jonsson (Proposition~\ref{Prop:Jonsson}) that inspired our generalisation.

\section{Inequalities associated to degree functions and the proof of Proposition~\ref{Prop:Dynamicaldegreeestimate}}
\label{Sec:IneqAndProofOfProp:Dynamicaldegreeestimat}

\subsection{Definitions of elementary, affine and triangular automorphisms}
	 Let us recall the following classical definitions (even if our definition of elementary is slightly more restrictive than what is used in the literature):
	\begin{definition}\label{definition:triangularaffine}
	An endomorphism $f=(f_1,\ldots,f_n)\in \End(\A^n)$ is said to be 
	\begin{itemize}[leftmargin=*]
	\item\emph{triangular} if $f_i\in \k[x_1,\ldots,x_i]$ for each $i\in \{1,\ldots,n\}$,
	\item\emph{elementary} if $f_i=x_i$ for  for each $i\in \{1,\ldots,n-1\}$.
	\item\emph{an affine automorphism} if $f\in \Aut(\A^n)$ and if $\deg(f)=1$,
	\item\emph{a permutation of the coordinates} if $\{f_1,\ldots,f_n\}=\{x_1,\ldots,x_n\}$,
	\item \emph{affine-triangular} if $f=\alpha\circ \tau$ where $\alpha$ is an affine automorphism and $\tau$ is a triangular endomorphism,
	\item \emph{affine-elementary} if $f=\alpha\circ e$ where $\alpha$ is an affine automorphism and $e$ is an elementary endomorphism,
	\item \emph{permutation-triangular} if $f=\alpha\circ \tau$ where $\alpha$ is a permutation of the coordinates and $\tau$ is a triangular endomorphism.
	\item \emph{permutation-elementary} if $f=\alpha \circ e$ where $\alpha$ is a permutation of
	the coordinates and $e$ is an elementary endomorphism.
	\end{itemize}
\end{definition}

For each $n\le 4$, if $\mathrm{char}(\k)\not=2$, every automorphism of $\A^n$ of degree $2$ is conjugate, by an affine automorphism, to an affine-triangular automorphism, 
see \cite{MGO}. This result is false in dimension $n=5$ \cite{Sun}, as for example
\[
	f = (x_1 + x_2 x_4, x_2 + x_1 x_5 + x_3 x_4, x_3 - x_2 x_5, x_4, x_5) \in \Aut(\AA^5)
\]
shows: the Jacobian of the homogeneous part of degree $2$ of 
an affine-triangular automorphism of degree $\leq 2$ contains a zero-column, but the Jacobian of
the homogeneous part of degree $2$ of $f$ contains linearly independent columns 
(see also \cite[Theorem~3.2]{Sun}).

There are quite a few automorphisms of $\A^3$ of degree $3$ that are not conjugate, by an affine automorphism, to affine-triangular automorphisms. More precisely, when $\k$ is algebraically closed, then
each automorphism of $\AA^3 = \Spec(\k[x, y, z])$ 
of degree $3$ is conjugate, by an affine automorphism, either to an affine-triangular automorphism or
to an automorphism of the form
\[
		\label{eq.non_triangular}
		\tag{$\ast$}
		\alpha(x + yz + za(x, z), y + a(x, z) + r(z), z) \in \Aut(\AA^3)
\]
where $a \in \k[x, z] \setminus \k[z]$ is homogeneous of degree $2$,
$r \in \k[z]$ is of degree $\leq 3$ and $\alpha$ is an affine automorphism, see 
\cite[Theorem 3]{BlSa2019Automorphisms-of-t}. In fact, non of the automorphisms in~\eqref{eq.non_triangular}
is conjugated, by an affine automorphism, to an affine-triangular automorphism, see~\cite[Proposition~3.9.4]{BlSa2019Automorphisms-of-t}.

For $\k = \CC$ various (dynamical) properties of 
the affine-elementary automorphisms
$(x_0+x_1+x_0^q x_2^d, x_0, \alpha x_2) \in \Aut(\AA^3)$ 
with $\alpha \in \CC$, $0 < |\alpha| \leq 1$, $q \geq 2$, $d \geq 1$ 
are studied in~\cite{DeLe2018Dynamics-of-a-fami} and in particular their dynamical 
degree is computed, which is equal to the integer $q$.

\subsection{Existence of dynamical degrees}

We recall the following folklore result, which implies that the dynamical degree is well-defined.
\begin{lemma}
	\label{lem.dynamical_degree_exists}
	Let $(a_r)_{r\ge 1}$ be a sequence of real numbers in $\r_{\ge 1}$ such that $a_{r+s}\le a_{r}\cdot a_s$ for each $r,s\ge 1$. Then, $((a_r)^{1/r})_{r\ge 1}$ is a  sequence that converges towards 
	$\inf_{r\geq 1} ((a_r)^{1/r})\in \r_{\ge 1}$.
\end{lemma}
\begin{proof}	
As $(\log(a_r))_{r\ge 1}$ is subbadditive,  $(\frac{\log(a_r)}{r})_{r\ge 1}$ converges to $\inf_{r \geq 1}(\frac{\log(a_r)}{r}) \geq 0$ by Fekete's subadditivity 
	Lemma~(see \cite[Satz II]{Fe1923Uber-die-Verteilun} or \cite[Lemma~1.2.1]{St1997Probability-theory}).
\end{proof}

In case $\mu = (\mu_1, \ldots, \mu_n) \in (\RR_{\geq 0})^n \setminus \{0\}$ is of the from
$\mu_1 = \ldots = \mu_m = 0$ and $\mu_{m+1} = \ldots = \mu_n = 1$ for some $0 \leq m < n$
we denote for any polynomial $p \in \k[x_1, \ldots, x_n]$ its $\mu$-degree $\deg_{\mu}(p)$ 
by $\deg_{x_{m+1}, \ldots, x_{n}}(p)$. Moreover, we denote for an endomorphism
$f = (f_1, \ldots, f_n) \in \End(\AA^n)$
\[
	\deg_{x_{m+1}, \ldots, x_{n}}(f) = \max_{j\in \{1, \ldots, n\}} \deg_{x_{m+1}, \ldots, x_{n}}(f_j) \, .
\]
If $m = 0$, then $\deg_{\mu}(f)$ is simply the classical degree that we denote by $\deg(f)$.
If $m > 0$, then $\deg_{\mu}(f)$ is in general not equal to $\deg_{x_{m+1}, \ldots, x_{n}}(f)$.
In fact, $\deg_{\mu}(f)$ is equal to $\deg_{x_{m+1}, \ldots, x_{n}}(f)$ in case
$\deg_{x_1, \ldots, x_m}(f_i) = 0$ for all $i \in \{1, \ldots, m\}$ and otherwise it is equal to $\infty$.

\begin{corollary}
	\label{cor.existence_ddeg}
	Let $f\in\End(\A^n)$ be an endomorphism. For each integer $m\in \{0,\ldots,n-1\}$, the sequence
	\[\deg_{x_{m+1},\ldots,x_n}(f^r)^{1/r}\]
	converges to a real number $\mu_m\ge 1$. This gives in particular the dynamical degree $\lambda(f)=\mu_0$, which satisfies $\lambda(f^d)=\lambda(f)^d$ for each $d\ge 1$.
\end{corollary}
\begin{proof}
	This follows from Lemma~\ref{lem.dynamical_degree_exists}, as 
	\[\deg_{x_{m+1},\ldots,x_n}(f^{r+s})\le \deg_{x_{m+1},\ldots,x_n}(f^r)\cdot \deg_{x_{m+1},\ldots,x_n}(f^s),\] for all $r,s\ge 1$.
\end{proof}

\subsection{Basic properties of degree functions}\label{SubSec:DegreeFunctions}

Below we list several properties of degree functions (see Definition~\ref{Defi:degreeMu}). 
Apart from the easy observations  
$\deg_\mu|_{\k^*}=0$, $\deg_\mu(f\cdot g)=\deg_\mu(f) + \deg_\mu(g)$ and $\deg_\mu(f+g)\le \max(\deg_\mu(f),\deg_\mu(g))$, which correspond to say that $-\deg_\mu$ is a valuation (see Remark~\ref{RemVal}), we have:
\begin{remark}\label{Rem:DegreeMu} We fix $\mu=(\mu_1,\ldots,\mu_n)\in (\mathbb{R}_{\ge 0})^n\setminus \{0\}$ and get:
	\begin{enumerate}[leftmargin=*]
		\item\label{DecHom}
		As explained in Definition~\ref{Defi:degreeMu}, each polynomial $p\in \k[x_1,\ldots,x_n]\setminus \{0\}$ can be written uniquely as a finite sum 
		\[
			p=\sum\limits_{\theta\in \r_{\ge 0}} p_\theta
		\]
		where each $p_\theta\in \k[x_1,\ldots,x_n]$ is $\mu$-homogeneous of degree $\theta$. We then obtain $\deg_\mu(p)=\max \{\theta\mid p_\theta\not=0\}$. 
		\item\label{degmudeg}
		Let  $m \in \{0, \ldots, n-1 \}$ and assume that $\mu_i = 0$ for $i \leq m$, but $\mu_i > 0$ for $i > m$.
		Then we have for each polynomial $p\in \k[x_1,\ldots,x_n]\setminus \{0\}$
		\[
		\mu_{\textrm{min}} \cdot \deg_{x_{m+1}, \ldots, x_n}(p) 
		\leq \deg_{\mu}(p) \leq 
		\mu_{\textrm{max}} \cdot \deg_{x_{m+1}, \ldots, x_n}(p)
		\] 
		where $\mu_{\textrm{min}} = \min_{m+1 \leq i \leq n} \mu_i$ and 
		$\mu_{\textrm{max}} = \max_{m+1 \leq i \leq n} \mu_i$.
		In particular, for each dominant endomorphism $f \in \End(\AA^n)$ we have 
		\[
		\lim_{r \to \infty} \deg_{x_{m+1}, \ldots, x_n}(f^r)^{\frac{1}{r}} = 
		\lim_{r \to \infty} \max_{i \in \{1, \ldots, n\}}\deg_\mu((f^r)_i)^{\frac{1}{r}}
		\]
		where $(f^r)_i$ denotes the $i$-th coordinate function of $f^r$.
		Note that the left hand side is the dynamical degree $\lambda(f)$ in case $m = 0$, i.e.~when $\mu \in (\r_{>0})^n$.
	\end{enumerate}
\end{remark}

\subsection{Endomorphisms that preserve a linear projection}
The following is an algebraic analogue of the application of \cite[Theorem 1.1]{DiNg2011Comparison-of-dyna} to endomorphisms of $\A^n$ that preserve a linear projection:
\begin{lemma}
\label{lem.DinhNGuyen}
Let $f =(f_1, \ldots, f_n)\in \End(\A^n)$ be a dominant endomorphism. For each $r\ge 1$, we write
\[
	f^r=((f^r)_1,\ldots,(f^r)_n) \, .
\]
Let $m\in \{0,\ldots,n-1\}$ be such that 
$f_1,\ldots,f_m\in \k[x_1,\ldots,x_m]$. Then, the dynamical degree of $f$ is given by $\lambda(f)=\max\{\lambda_1,\lambda_2\}$, where
\[
	\begin{array}{rcl}
		\lambda_1&=&\lim\limits_{r\to \infty} \max\{\deg((f^r)_1),\ldots, \deg((f^r)_m)\}^{1/r}
		= \lambda((f_1, \ldots, f_m)) \\
		\lambda_2&=&\lim\limits_{r\to \infty} \max\{\deg_{x_{m+1},\ldots,x_n}((f^r)_{m+1}),\ldots, \deg_{x_{m+1},\ldots,x_n}((f^r)_n)\}^{1/r} \\
		&=& \lim\limits_{r\to \infty}\deg_{x_{m+1},\ldots,x_n}(f^r)^{1/r}
	\end{array}
\]
are two limits which exist. $($If $m = 0$, by convention we set $\lambda_1 = 1.)$
\end{lemma}
\begin{proof}
For each $r\ge 1$, we write
\[
\begin{array}{rcl}
a_r&=& \max\{\deg((f^r)_1),\ldots, \deg((f^r)_m)\}\\
b_r&=& \max\{\deg((f^r)_{m+1}),\ldots, \deg((f^r)_{n})\}\\
c_r&=& \max\{\deg_{x_{m+1},\ldots,x_n}((f^r)_{m+1}),\ldots, \deg_{x_{m+1},\ldots,x_n}((f^r)_{n})\} \\
	&=& \deg_{x_{m+1}, \ldots, x_n}(f^r).
\end{array}
\]
As $b_r\ge c_r$, we obtain for each $r \geq 1$
\[\deg(f^r)=\max\{a_r,b_r\}\ge \max\{a_r,c_r\}.\]
It follows from Corollary~$\ref{cor.existence_ddeg}$ that the limits
\[
\lambda_1=\lim\limits_{r\to \infty} a_r^{1/r} \, , \ \lambda_2=\lim\limits_{r\to \infty} c_r^{1/r}
\quad \textrm{and} \quad 
\lambda(f)=\lim\limits_{r\to \infty} \deg(f^r)^{1/r}
\]
exist (and all belong to $\r_{\ge 1}$). 
We obtain 
\[
	\lambda(f)=\lim\limits_{r\to \infty} \max\{ a_r^{1/r}, b_r^{1/r}\} \ge \lim\limits_{r\to \infty} \max\{ a_r^{1/r}, c_r^{1/r}\}= \max\left\{\lambda_1,\lambda_2\right\} \, .
\]
We may thus assume that $\lambda(f)>\lambda_1$, which implies that $\lim_{r\to \infty} b_r^{1/r}$ exists, and is equal to $\lambda(f)$. It remains to see that in this case $\lambda(f)\le \max \{\lambda_1,\lambda_2\}$.

For all $r,s \ge 1$ and each $i\in \{m+1,\ldots,n\}$, the polynomial $(f^{r+s})_i$ is obtained by replacing $x_1,\ldots,x_n$ with $(f^{r})_1,\ldots,(f^r)_n$ in $(f^s)_i$, so the degree of $(f^{r+s})_i$ is at most
\begin{align*}
	&\deg_{x_1,\ldots,x_m}((f^s)_i)\cdot \deg((f^r)_1,\ldots,(f^r)_m) \\
	&+ \deg_{x_{m+1},\ldots,x_n}((f^s)_i)\cdot \deg((f^r)_{m+1},\ldots,(f^r)_n) \, .
\end{align*}
This gives $b_{r+s}\le b_s \cdot a_{r}+c_s\cdot b_{r}$. When we choose then $s=r$, we obtain
\[
	b_{2r}\le b_r \cdot (a_{r}+c_r) \, .
\]
As  $\lambda(f)=\lim\limits_{r\to \infty} b_{2r}^{1/2r}$, we have 
$\lambda(f)^2=\lim\limits_{r\to \infty} b_{2r}^{1/r}$. 
The above inequality gives

\[
	\begin{array}{rcl}
 		\lambda(f)^2&=&\lim\limits_{r\to \infty} b_{2r}^{1/r}\\
		&\le&\lim\limits_{r\to \infty} b_r^{1/r}\cdot \limsup\limits_{r\to \infty}(a_{r}+c_r)^{1/r}\\
		&\le&\lambda(f)\cdot \limsup\limits_{r\to \infty} (2 \max \{a_r, c_r\})^{1/r} \\
 		&=&\lambda(f)\cdot \max\{\lambda_1,\lambda_2\},
 	\end{array}
 \]
 so $\lambda(f)\le \max \{\lambda_1,\lambda_2\}$.
\end{proof}

\begin{corollary}\label{Cor:Jung}
Let $n\ge 2$ and let $f\in \Aut(\A^n)$ be an automorphism such that $f_1,\ldots,f_{n-2}\in \k[x_1,\ldots,x_{n-2}]$ and such that the dynamical degree of $g=(f_1,\ldots,f_{n-2})\in \Aut(\A^{n-2})$ is an integer. Then, the dynamical degree of $f$ is an integer.
\end{corollary}

\begin{proof}
By Lemma~\ref{lem.DinhNGuyen}, one has $\lambda(f)=\max\{\lambda(g),\lambda_2\}$, where 
\[ 
 	\lambda_2=\lim\limits_{r\to \infty} \max\{\deg_{x_{n-1},x_n}((f^r)_{n-1}), \deg_{x_{n-1},x_n}((f^r)_n)\}^{1/r}
 	\, .
 \]
 It remains to see that $\lambda_2$ is an integer.
 As $\k[x_1,\ldots,x_{n-2},f_{n-1},f_{n}]=\k[x_1,\ldots,x_{n}]$, one has $K[f_{n-1},f_{n}]=K[x_{n-1},x_n]$, where $K=\k(x_1,\ldots,x_{n-2})$. Hence, one can see the automorphism
 $(x_1,\ldots,x_n)\mapsto (x_1,\ldots,x_{n-1},f_{n-1},f_{n})$ of $\AA^n$ as an automorphism $F\in \Aut_K(\A^2)$ of $\A^2$ defined over $K$. 
For each $i \geq 0$, the automorphism
 $g^{-i} \circ (x_1,\ldots,x_{n-1},f_{n-1},f_{n}) \circ g^i$ of $\AA^n$ 
 can be seen as an element of $\Aut_K(\AA^2)$ that we denote by $F^{g^i}$
 where we identify $g$ with the automorphism $(f_1, \ldots, f_{n-2}, x_{n-1}, x_n) \in \Aut(\AA^n)$.
 This gives
 \[
 	\max\{\deg_{x_{n-1},x_n}((f^r)_{n-1}), \deg_{x_{n-1},x_n}((f^r)_n)\}=\deg (G_r)
 \]
 where $G_r=F^{g^{r-1}} \circ \cdots \circ  F^{g^2}  \circ F^{g} \circ F\in \Aut_K(\A^2)$, 
 since $G_r = g^{-r} \circ f^r$
 when we consider $G_r$, $g$ and $f$ as automorphisms of $\AA^n$.
 
  According to the Jung-van der Kulk Theorem \cite{Ju1942Uber-ganze-biratio,Ku1953On-polynomial-ring}, one can write $F=F_1\circ \cdots \circ F_s$ where each $F_i\in \Aut_K(\A^2)$ is either triangular or affine. One can moreover assume that two consecutive $F_i$ are not both affine or both triangular (as otherwise one may reduce the description), and get then $\deg(F)=\prod_{i=1}^s \deg(F_i)$ (follows by looking at what happens at infinity or by \cite[Lemma 5.1.2]{Es2000Polynomial-automor}). 
  We prove that $\lambda_2$ is an integer by induction on $s$. If $s=1$, then $F$ is either affine or triangular; this implies that the set $\{\deg(G_r)\mid {r\ge 1}\}$ is bounded, so $\lambda_2=1$. If $s>1$ 
  and $F_1$, $F_s$ are both affine or both triangular, we replace $F$ with $(F_1)^g \circ F\circ F_1^{-1}$. This replaces $G_r=F^{g^{r-1}} \circ \cdots \circ  F^{g} \circ F$ with
  $\tilde{G}_r=(F_1)^{g^r}\circ G_r\circ  F_1^{-1}$. 
  As $\deg((F_1)^{g^r})=\deg(F_1)$ for each $r\ge 1$, one has  
  \[
  	\frac{1}{\deg(F_1)^2}\deg(G_r)\le \deg(\tilde{G}_r)\le \deg(G_r)\cdot \deg(F_1)^2 \, , 
  \]
  so this replacement does not change the value of $\lambda_2$. As this decreases the value of $s$, we may assume that $F_1$ and $F_s$ are not both triangular or affine. Hence, for each $r\ge 1$, $G_r$ is a product of $rs$ elements that are affine or triangular, 
  with no two consecutive in the same group. This gives $\deg(G_r)=\prod_{i=0}^{r-1} \prod_{j=1}^s \deg(F_j^{g^i})=\prod_{i=0}^{r-1} \prod_{j=1}^s \deg(F_j)=\deg(F)^r$. 
  Hence, $\lambda_2=\deg(F)$ is an integer.
\end{proof}

\begin{corollary}\label{Cor:Jung2}
The dynamical degree of any element of $\Aut(\A^2)$ is an integer. Similarly, the dynamical degree of any element of $\Aut(\A^3)$ $($respectively $\Aut(\A^4))$ which preserves the set of fibres of a linear projection $\A^3\to \A^1$ or $\A^3\to \A^2$ $($respectively $\A^4\to \A^2)$ is an integer.
\end{corollary}

\begin{proof}
The fact that the dynamical degree of any element of $\Aut(\A^2)$ is an integer follows from Corollary~\ref{Cor:Jung} applied to $n=2$. If $f\in \Aut(\A^3)$ is an automorphism that preserves the 
set of fibres of a linear projection $\A^3\to \A^1$ or $\A^3\to \A^2$, then one may conjugate by an element of $\GL_3$ and obtain $f=(f_1,f_2,f_3)$ with either $f_1\in \k[x_1]$ or $f_1,f_2\in \k[x_1,x_2]$. The fact that $\lambda(f)$ is an integer follows then from Corollary~\ref{Cor:Jung} and Lemma~\ref{lem.DinhNGuyen}, respectively (in the second case, one uses the fact that the dynamical degree of $(f_1, f_2) \in \Aut(\AA^2)$ is an integer). Similarly, in the case of an automorphism of $\A^4$ preserving a linear projection $\A^4\to \A^2$, one restricts to the case $f=(f_1,\ldots,f_4)\in \Aut(\A^4)$ with $f_1,f_2\in \k[x_1,x_2]$, and applies Corollary~\ref{Cor:Jung}.
\end{proof}

\subsection{Homogeneous endomorphisms}\label{SubSec:HomEndo}

\begin{lemma}
\label{Lemm:DegreeEndhom}
Let $h = (h_1, \ldots, h_n) \in \End(\A^n)$, 
let $\mu=(\mu_1,\ldots,\mu_n)\in (\mathbb{R}_{\ge 0})^n\setminus \{0\}$ and let $\theta\in \r_{\ge 0}$. The following conditions are equivalent:
\begin{enumerate}[leftmargin=*]
\item\label{Endmuhom1}
The polynomial $h_i$ is $\mu$-homogeneous of degree $\theta\mu_i$ for each $i\in \{1,\dots,n\}$.
\item\label{Endmuhom2}
For each $\mu$-homogeneous polynomial $p$ of degree $\xi$ and each integer $r\ge 1$, the polynomial $p \circ h^r$ is $\mu$-homogeneous of degree $\theta^r \xi$.
\end{enumerate}
If additionally $h_i \neq 0$ for each $i \in \{1, \ldots, n\}$, then~\ref{Endmuhom1} and~\ref{Endmuhom2}
are equivalent to
\begin{enumerate}[leftmargin=*]
	\setcounter{enumi}{2}
	\item\label{Endmuhom3}  
	For each Matrix $M$ contained in $h$, $\mu$ is an eigenvector to the eigenvalue $\theta$.
\end{enumerate}
\end{lemma}

\begin{proof}The implication $\ref{Endmuhom2}\Rightarrow \ref{Endmuhom1}$ is given by choosing $p=x_i$ 
for $i = 1, \ldots, n$, so we may assume $ \ref{Endmuhom1}$ and prove \ref{Endmuhom2}. It suffices to prove \ref{Endmuhom2} for $r=1$, as the general result follows by induction. 

If $p=0$, then $h(p)=0$ is $\mu$-homogeneous of any degree. 
It then suffices to do the case where $p$ is a monomial: we write 
$p=\zeta x_1^{a_1}x_2^{a_2}\cdots x_n^{a_n}$ with $\zeta \in \k^*$, $a_1,\ldots,a_n\ge 0$, which is 
$\mu$-homogeneous of degree $\deg_\mu(p)=\sum_{i=1}^n a_i\mu_i$.
As $h_i$
is $\mu$-homogeneous of degree $\theta\mu_i$, 
the polynomial $p\circ h=\zeta h_1^{a_1}h_2^{a_2}\cdots h_n^{a_n}$
is $\mu$-homogeneous of 
degree $\sum_{i=1}^n a_i \theta \mu_i=\theta\deg_\mu(p)$.

Now, we assume additionally that $h_i \neq 0$ for each $i \in \{1, \ldots, n\}$.
The equivalence between~\ref{Endmuhom1} and~\ref{Endmuhom3} follows immediately 
from the definition of the $\mu$-degree.
\end{proof}

\begin{definition}
Let $\mu=(\mu_1,\ldots,\mu_n)\in (\mathbb{R}_{\ge 0})^n\setminus \{0\}$ and let $\theta\in \r_{\ge 0}$. We say that $h \in \End(\AA^n)$
 is \emph{$\mu$-homogeneous of degree $\theta$} if the conditions of Lemma~\ref{Lemm:DegreeEndhom} are satisfied.
\end{definition}

\begin{lemma}\label{Lem:HomDecEndtheta}
 Let $\mu=(\mu_1,\ldots,\mu_n)\in (\mathbb{R}_{\ge 0})^n\setminus \{0\}$. For each $f=(f_1,\ldots,f_n)\in \End(\A^n)$ and each $\theta\in \r_{\ge 0}$, the following are equivalent:
 \begin{enumerate}[leftmargin=*]
 \item\label{fissumhomtheta}
 We can write $f$ as a finite sum $f=\sum_{0\le \xi\le \theta} g_\xi$, where each $g_\xi\in  \End(\A^n)$ is $\mu$-homogeneous of degree~$\xi$.
\item\label{fboundeddegtheta2}
	$\deg_\mu(f)\le \theta$.
 \end{enumerate}
 \end{lemma}
 
\begin{proof}
$\ref{fissumhomtheta}\Rightarrow\ref{fboundeddegtheta2}$: For each $i\in \{1,\ldots,n\}$, the polynomial $f_i$ is the sum of the $i$-th components of the endomorphisms $g_{\xi}$. As each of these polynomials has degree $\xi \mu_i\le \theta \mu_i$, the polynomial $f_i$ is of $\mu$-degree $\deg_\mu(f_i)\le \theta\mu_i$.

$\ref{fboundeddegtheta2}\Rightarrow\ref{fissumhomtheta}$: 
 As in Remark~\ref{Rem:DegreeMu}\ref{DecHom}, we write each $f_i$, $i\in \{1,\ldots,n\}$ as 
 $f_i=\sum_{0\le \kappa \le \theta\mu_i} p_{i,\kappa}$ where each $p_{i,\kappa}$ is  $\mu$-homogeneous of degree $\kappa$.

We define $g_0=(p_{1,0},\ldots,p_{n,0})\in \End(\A^n)$, which is $\mu$-homogeneous of degree $0$. 

For each $\xi\in \r$ with $0\le \xi \le \theta$, we define the $i$-th component $(g_\xi)_i$ of $g_\xi$ as follows: if $\mu_i=0$ and $\xi > 0$, then $(g_\xi)_i=0$ and otherwise, we choose $(g_\xi)_i=p_{i,\xi\mu_i}$. By construction, $g_\xi$ is $\mu$-homogeneous of degree $\xi$.

 Moreover, $f_i=\sum_{0\le \kappa \le \theta\mu_i} p_{i,\kappa}=\sum_{0\le \xi \le \theta} (g_{\xi})_i$ for each $i\in \{1,\ldots,n\}$ with $\mu_i>0$. If $\mu_i=0$, then $f_i=\sum_{0\le \kappa \le \theta\mu_i} p_{i,\kappa}=p_{i,0}=\sum_{0\le \xi \le \theta} (g_{\xi})_i$. This yields  
 $f=\sum_{0\le \xi\le \theta} g_\xi$.\qedhere
 \end{proof}

\begin{remark}
In the decomposition of Lemma~\ref{Lem:HomDecEndtheta}\ref{fissumhomtheta}, 
the $i$-th component of each $g_{\xi}$ is unique, if $\mu_i>0$, but is not unique if $\mu_i=0$.
 \end{remark}
\begin{example}
 We have  $\deg_{(1,\ldots,1)}(f)=\deg(f)$ and $\deg_\mu(\id_{\A^n})=1$ for each  $\mu\in  (\mathbb{R}_{\ge 0})^n\setminus \{0\}$. However, $\deg_{(2,3,0)}(x_1,x_2+x_1^2x_3,x_3)=\frac{4}{3}$.
 \end{example}
\begin{lemma}\label{Lem:HomDecEnd}
 Let $\mu=(\mu_1,\ldots,\mu_n)\in (\mathbb{R}_{\ge 0})^n\setminus \{0\}$. For each $f=(f_1,\ldots,f_n)\in \End(\A^n)$, the following are equivalent:
 \begin{enumerate}[leftmargin=*]
 \item\label{degffinite}
 $\deg_\mu(f)<\infty$.
 \item\label{fzerovariables}
 For each $i\in \{1,\ldots,n\}$ such that $\mu_i=0$, the element $f_i$ is a polynomial in the variables $\{x_j\mid j\in \{1,\ldots,n\},\mu_j=0\}$.
 \end{enumerate}
 In particular, if $\mu \in  (\mathbb{R}_{>0})^n$ then the above conditions hold.
 \end{lemma}
 
 \begin{proof}
$ \ref{degffinite}\Rightarrow \ref{fzerovariables}$: Suppose that $\theta=\deg_\mu(f)<\infty$. For each $i\in \{1,\ldots,n\}$, we get $\deg_\mu(f_i)\le \theta \mu_i$ (Definition~\ref{Defi:degmuendo}). If $\mu_i=0$, then $\deg_\mu(f_i)=0$, which means that  $f_i$ is a polynomial in the variables $\{x_j\mid j\in \{1,\ldots,n\},\mu_j=0\}$.

$ \ref{fzerovariables}\Rightarrow \ref{degffinite}$:   it follows from  \ref{fzerovariables} that $\deg_\mu(f_i) \leq 0$ for each $i\in \{1,\ldots,n\}$ such that $\mu_i=0$. This gives 
$\deg_\mu(f)=\max\left.\left\{ \deg_\mu(f_i) / \mu_i \, \right| \, \mu_i>0 \right\}<\infty$.
 \end{proof}

 \begin{lemma}\label{Lem:MaximalEigenVectorLeadingpartNotZero}
	Let $f=(f_1,\ldots,f_n)\in \End(\A^n)$ be a dominant endomorphism. For each maximal eigenvector $\mu$ of $f$, the $\mu$-leading part $g=(g_1, \ldots, g_n) \in \End(\AA^n)$ of $f$
	has the following properties:
\begin{enumerate}[leftmargin=*]
\item \label{Lem:MaximalEigenVectorLeadingpartNotZero_1}
The maximal eigenvalue $\theta$ of $f$ is such that $\deg_\mu(g)=\deg_\mu(f)=\theta < \infty$;
\item \label{Lem:MaximalEigenVectorLeadingpartNotZero_2}
For each $i\in \{1,\ldots,n\}$, the polynomial $g_i$ is non-constant.
\end{enumerate}
 \end{lemma}
 \begin{proof}
 	As $\mu=(\mu_1,\ldots,\mu_n)\in (\mathbb{R}_{\ge  0})^n$ is a maximal eigenvector of $f$, we have $\deg_\mu(f_i)= \theta \mu_i$ for each $i\in \{1,\ldots,n\}$, where $\theta$ is the maximal eigenvalue of $f$. This gives $\deg_\mu(f)=\theta < \infty$ and therefore
 	$\deg_\mu(g_i)=\theta \mu_i=\deg_\mu(f_i)$ for each $i\in \{1,\ldots,n\}$. Hence, we get~\ref{Lem:MaximalEigenVectorLeadingpartNotZero_1}.
 	In case $\mu_i > 0$, we have $\deg_\mu(g_i) = \theta \mu_i > 0$ 
 	and thus $g_i$ is non-constant.
 	In case $\mu_i = 0$, we have $\deg_\mu(f_i) = \theta \mu_i = 0$
 	and thus $g_i = f_i$. As $f$ is dominant, the latter polynomial is non-constant.
 	This shows~\ref{Lem:MaximalEigenVectorLeadingpartNotZero_2}.
 \end{proof}
 
\subsection{Inequalities obtained by iterations}\label{SubSec:IneqDegrees}

\begin{lemma}\label{Lemma:Dynamicaldegreeestimate}
Let $f=(f_1,\ldots,f_n)\in \End(\A^n)$ be a dominant endomorphism. Suppose that $\mu=(\mu_1,\ldots,\mu_n)\in (\mathbb{R}_{\ge 0})^n$ and that $\theta = \deg_\mu(f) \in \RR_{\geq 0}$.
Let $g=(g_1,\ldots,g_n)\in \End(\A^n)$ be the $\mu$-leading part of $f$. 
Then the following hold:
\begin{enumerate}[leftmargin=*]
\item\label{fsumhomgsmall}
We can write $f$ as a finite sum $f=g+\sum_{0\le \xi< \theta} g_\xi$, where each $g_\xi\in  \End(\A^n)$ is $\mu$-homogeneous of degree~$\xi$.
\item\label{grfrh}
The $i$-the coordinate function $(g^r)_i$ of $g^r$ is the $\mu$-homogeneous part of 
degree $\theta^r \mu_i$ of $(f^r)_i$  for each $i\in \{1,\ldots,n\}$ and each $r\ge1$.
\item\label{deggrfr}
$\deg_\mu(f^r)\le \theta^r$ for each $r\ge 1$.
\item\label{degmudyn}
We have
\[
	1\le \lim\limits_{r\to\infty} \max\limits_{i\in \{1,\ldots,n\}} \deg_\mu ((f^r)_i)^{1/r} =\lim\limits_{r\to\infty}(\deg_\mu(f^r))^{1/r}\le \theta \, .
\]
\item\label{degmudyn2}
If $\theta > 1$, the following are equivalent:
\begin{enumerate}[leftmargin=*]
\item\label{degmudyn2a}
$\lim_{r\to\infty}(\deg_\mu(f^r))^{1/r} = \theta$.
\item\label{degmudyn2b}
$f$ is $\mu$-algebraically stable.
\item\label{degmudyn2c}
For each $r\ge 1$ there is $i\in \{1,\ldots,n\}$ with $\mu_i>0$ and $(g^r)_i\not=0$.
\end{enumerate}
\end{enumerate}
\end{lemma}

\begin{proof}
As $\deg_\mu(f) = \theta$, we have $\deg_\mu(f_i)\le \theta \mu_i$ for each $i\in \{1,\ldots,n\}$. 
Moreover, as $f$ is dominant and $\mu\not=0$, there are $i,j\in \{1,\ldots,n\}$ such that $\mu_i>0$ and $\deg_{x_i}(f_j)\ge 1$. This implies that $\deg_\mu(f_j)\ge \mu_i>0$ and thus 
\[
	0<\deg_\mu(f) = \theta \, .
\] 

We now observe that $\deg_\mu(f-g)<\theta$. Indeed, for each $j\in \{1,\ldots,n\}$, the $j$-th component $g_j$ of $g$ is the $\mu$-homogeneous part of $f_j$ of degree $\theta\mu_j\ge \deg_\mu(f_j)$. If $\mu_j=0$, then $f_j=g_j$, and if $\mu_j>0$, then $\deg_\mu(f_j-g_j)<\theta\mu_j$.

By Lemma~\ref{Lem:HomDecEndtheta}, we  can write $f-g$ as a finite sum 
$f-g=\sum_{0\le \xi<\theta} g_\xi$, where each $g_\xi\in  \End(\A^n)$ is $\mu$-homogeneous of degree~$\xi$. This gives \ref{fsumhomgsmall}.

We now prove \ref{grfrh}-\ref{deggrfr} by induction on $r\ge 1$. For $r=1$, \ref{grfrh} follows from the definition of $g$.
Moreover, \ref{deggrfr} is given by hypothesis.

We now assume  \ref{grfrh}-\ref{deggrfr} for some integer $r\ge 1$ and prove them for $r+1$. For each $i\in \{1,\ldots,r\}$, we write $(f^{r})_i=(g^{r})_i+s_i$, where $(g^{r})_i$ is $\mu$-homogeneous of degree $\theta^r\mu_i$ and $\deg_\mu(s_i)<\theta^r\mu_i$. This gives
\[
\begin{array}{rcl}
(f^{r+1})_i&=&((g^{r})_i+s_i) \circ f \\
&\stackrel{\ref{fsumhomgsmall}}{=}&
(g^{r+1})_i +s_i \circ g +\sum\limits_{0\le \xi< \theta} ((g^{r})_i+s_i) \circ g_\xi
\end{array}
\]
As $g$ is $\mu$-homogeneous of degree $\theta$, the polynomial $(g^{r+1})_i$
is $\mu$-homogeneous of degree $\theta^{r+1}\mu_i$ (Lemma~\ref{Lemm:DegreeEndhom}). As $s_i$ is a sum of $\mu$-homogeneous polynomials of degree $<\theta^r\mu_i$ and $g_\xi$ is 
$\mu$-homogeneous of degree $\xi<\theta$, we have 
\[
	\deg_\mu( s_i \circ g +\sum\limits_{0\le \xi< \theta} ((g^{r})_i+s_i) \circ g_\xi )<
	\theta^{r+1}\mu_i
\] 
(by using Lemma~\ref{Lemm:DegreeEndhom} again). This yields
\ref{grfrh}-\ref{deggrfr} for $r+1$.

We now prove \ref{degmudyn}.  We choose $i\in \{1,\ldots,n\}$ such that $\mu_i=\max\{\mu_1,\ldots,\mu_n\}$, and observe that for each $r\ge 1$, there is $j\in \{1,\ldots,n\}$ such that $\deg_{x_i}((f^r)_j)>0$ (as $f$ is dominant), so $\deg_\mu((f^r)_j) \ge \mu_i = \max\{\mu_1,\ldots,\mu_n\} >0$. This implies that 
\[
	1\le \lim\limits_{r\to\infty} \max\limits_{i\in \{1,\ldots,n\}} \deg_\mu ((f^r)_i)^{1/r}
\]
(the limit exists by Remark~\ref{Rem:DegreeMu}\ref{degmudeg}).
Let us write $I_0=\{i\in \{1,\ldots,n\}\mid \mu_i=0\}$. For each $i\in I_0$, we have  $\deg_\mu(f_i)\le \theta\mu_i=0$, so $f_i$ is a polynomial in the variables $\{x_j\mid j\in I_0\}$. This implies that the same holds for $(f^r)_i$, for each integer $r\ge 1$. Hence, $\deg_\mu((f^r)_i)=0$ for each $i\in I_0$. Writing
$I_{>0}=\{i\in \{1,\ldots,n\}\mid \mu_i>0\}$, we get for each $r\ge 1$, 
\[
	\deg_\mu(f^r)= \max\left\{\frac{\deg_\mu((f^r)_i)}{\mu_i} \big| i\in I_{>0}\right\} \, .
\]
As $\deg_\mu(f^r)\le \theta^r$ (Assertion~\ref{deggrfr}), we obtain
\[
	\lim_{r\to\infty} \left(\max_{i\in \{1,\ldots,n\}} \deg_\mu (f^r)_i\right)^{1/r}=\lim\limits_{r\to\infty}\left(\deg_\mu(f^r)\right)^{1/r}\le \theta \, .
\]

It remains to prove~\ref{degmudyn2}; for this, we assume that $\theta > 1$. 
For each $r\ge 1$, Assertion~\ref{deggrfr} gives $\deg_\mu(f^r)\le \theta^r$, or equivalently $\deg_\mu((f^r)_i)\le \theta^r \mu_i$ for each $i\in \{1,\ldots,n\}$. The equality $\deg_\mu(f^r) = \theta^r$ holds 
if and only if there exists $i\in \{1,\ldots,n\}$ such that $\mu_i>0$ and 
$\deg_\mu((f^r)_i)= \theta^r \mu_i$. Since $(g^r)_i$ is the $\mu$-homogeneous part of  $(f^r)_i$ 
of degree $\theta^r \mu_i$ (follows from \ref{grfrh}), this gives the equivalence between \ref{degmudyn2b} and \ref{degmudyn2c}. It remains then to prove $\ref{degmudyn2a}\Leftrightarrow \ref{degmudyn2c}$.

``$\ref{degmudyn2c}\Rightarrow \ref{degmudyn2a}$'': Suppose that for each $r\ge 1$ there is $i\in \{1,\ldots,n\}$ such that $\mu_i>0$ and $(g^r)_i\not=0$. There is then $j\in \{1,\ldots,n\}$ and an infinite set $I\subset \mathbb{N}$ such that $\mu_j>0$ and $(g^r)_j\not=0$ for each $r\in I$. Assertion \ref{grfrh} implies that $\deg_\mu((f^r)_j)\ge \theta^r\mu_j$, for each $r\in I$, which implies that
\[
	\lim\limits_{r\to\infty}\left(\max\limits_{i\in \{1,\ldots,n\}} \deg_\mu (f^r)_i \right)^{1/r}\ge\theta \, .
\] 
This, together with \ref{degmudyn}, gives $\lim_{r\to\infty}(\deg_\mu(f^r))^{1/r} = \theta$.

``$\ref{degmudyn2a}\Rightarrow \ref{degmudyn2c}$'': Conversely, suppose that there exists $s\ge 1$ such that $(g^s)_i=0$ for each $i\in \{1,\ldots,n\}$ with $\mu_i>0$. For all such $i$ we obtain $\deg_\mu((f^s)_i)< \theta^s \mu_i$ (by \ref{grfrh} and \ref{deggrfr}). As $\theta > 1$, there exists then $\theta'\in \r$ with $1<\theta'< \theta$ such that 
\[
	\deg_\mu((f^s)_i)\le\theta'^s\mu_i
\]
for each $i\in \{1,\ldots,n\}$. Applying the  inequality of \ref{degmudyn} for $f^s$, we obtain
\[
	\lim\limits_{r\to\infty}\left(\max\limits_{i\in \{1,\ldots,n\}} (\deg_\mu (f^{sr})_i)^{1/r} \right)\le \theta'^s
\]
which gives, by taking the $s$-th root,
\[
	\lim\limits_{r\to\infty}\left(\max\limits_{i\in \{1,\ldots,n\}} (\deg_\mu (f^r)_i)^{1/r} \right)\le \theta'<\theta.
\]
\end{proof}

Now we can give a short proof of Proposition~\ref{Prop:Dynamicaldegreeestimate}.
\begin{proof}[Proof of Proposition~$\ref{Prop:Dynamicaldegreeestimate}$]
	\ref{thetafinite}: As $\mu\in (\mathbb{R}_{> 0})^n$,  we have $\theta: = \deg_\mu(f) < \infty$ (Lemma~\ref{Lem:HomDecEnd}).
	
	Using Remark~\ref{Rem:DegreeMu}\ref{degmudeg} we get
		\[
			\lambda(f) = \lim_{r \to \infty} \max\limits_{i\in \{1,\ldots,n\}} (\deg_\mu (f^r)_i)^{1/r} \, . 
		\]
	By definition, $g$ is the $\mu$-leading part of $f$. Now, Lemma~\ref{Lemma:Dynamicaldegreeestimate}\ref{degmudyn} implies that $1 \leq \lambda(f) \leq \theta$. If $\theta > 1$, we moreover obtain 
	\[
		\lambda(f)=\theta\Leftrightarrow\deg_\mu(f^r)=\theta^r  \text{ for each }r\ge 1 \Leftrightarrow g^r\not=0 \text{ for each }r\ge 1
	\] 
	(by Lemma~\ref{Lemma:Dynamicaldegreeestimate}\ref{degmudyn} and Lemma~\ref{Lemma:Dynamicaldegreeestimate}\ref{degmudyn2}).
\end{proof}

Another consequence of Lemma~\ref{Lemma:Dynamicaldegreeestimate} is the following result, that generalises Proposition~\ref{Prop:Dynamicaldegreeestimate} to the case where some coordinates of $\mu$ are zero.
\begin{corollary}\label{Coro:Computelambda}
Let $f=(f_1,\ldots,f_n)\in \End(\A^n)$ be a dominant endomorphism and let $\mu=(\mu_1,\ldots,\mu_n)\in (\mathbb{R}_{\ge 0})^n$ be such that $\theta = \deg_\mu(f) <\infty$, and assume that $m\in \{0,\ldots,n\}$ exists, such that $\mu_i=0$ for $i\in \{1,\ldots,m\}$ and $\mu_i>0$ for $i\in \{m+1,\ldots,n\}$ $($which can always be obtained by conjugating with a permutation$)$. Then, the following hold:
\begin{enumerate}[leftmargin=*]
\item\label{fmEndAm}
For each $i\in \{1,\ldots,m\}$, we have $f_i\in \k[x_1,\ldots,x_m]$. Hence, the element 
$\hat{f}= (f_1,\ldots,f_m)$ belongs to $\End(\A^m)$.
\item\label{lambdhatfequal}
If $\lambda(\hat{f})=\theta$, then $\lambda(f)=\theta$.
\item\label{lambdhatfsmall}
If $\lambda(\hat{f})<\theta$, then $\lambda(f)=\theta \Leftrightarrow f$ is $\mu$-algebraically stable.
\end{enumerate}
\end{corollary}

\begin{proof}
Assertion~\ref{fmEndAm} follows from the fact that $\deg_\mu(f)<\infty$ and the choice of $m$ (Lemma~\ref{Lem:HomDecEnd}\ref{fzerovariables}).

Lemma~\ref{lem.DinhNGuyen} then gives $\lambda(f)=\max\{\lambda(\hat{f}), \lim_{r\to \infty}\deg_{x_{m+1},\ldots,x_n}(f^r)^{1/r}\}$. By using the equality $\lim_{r\to \infty}\deg_{x_{m+1},\ldots,x_n}(f^r)^{1/r}=\lim_{r\to \infty}\deg_{\mu}(f^r)^{1/r}$ (see Remark~\ref{Rem:DegreeMu}\ref{degmudeg} and Lemma~\ref{Lemma:Dynamicaldegreeestimate}\ref{degmudyn}), we obtain
\[\lambda(f)=\max\{\lambda(\hat{f}), \lim\limits_{r\to \infty}\deg_{\mu}(f^r)^{1/r}\}.\]
Moreover, Lemma~\ref{Lemma:Dynamicaldegreeestimate}\ref{degmudyn} implies 
that $\lim_{r\to \infty}\deg_{\mu}(f^r)^{1/r}\le \deg_{\mu}(f)=\theta.$
This provides~\ref{lambdhatfequal}. To show \ref{lambdhatfsmall}, we assume that 
$\lambda(\hat{f})<\theta$ and obtain $\lambda(f)=\theta \Leftrightarrow \lim_{r\to \infty}\deg_{\mu}(f^r)^{1/r}=\theta$. 
This is equivalent to ask that $ f$ is $\mu$-algebraically stable, by Lemma~\ref{Lemma:Dynamicaldegreeestimate}\ref{degmudyn2}
(note that $1 \leq \lambda(\hat{f})$, since $f$ and thus $\hat{f}$ is dominant).
\end{proof}

We finish this section by the following simple observation:
\begin{lemma}
Let $f\in \End(\A^n)$ be a dominant endomorphism. For each $\mu\in (\r_{>0})^n$ such that $\theta=\deg_\mu(f)\in \r_{>1}$ and each translation $\tau=(x_1+c_1,\ldots,x_n+c_n)\in \Aut(\A^n)$ where $c_1,\ldots,c_n\in \k$, the following hold:
\[
	f \text{ is }\mu\text{-algebraically stable}
	\Leftrightarrow \tau\circ f \text{ is }\mu\text{-algebraically stable} \, .
\]
\end{lemma}
\begin{proof}
Denote by $g$ the $\mu$-leading part of $f$. As $\mu\in (\r_{>0})^n$, no component of $g$ contains any constant. Hence, $g$ is also the $\mu$-leading part of $\tau\circ f$.
By Lemma~\ref{Lemma:Dynamicaldegreeestimate}\ref{degmudyn2}, $f$ (respectively $\tau\circ f$) is $\mu$-algebraically stable if and only if for each $r\ge 1$ there is $i\in \{1,\ldots,n\}$ such that $(g^r)_i\not=0$.
\end{proof}

\section{Matrices associated to endomorphisms and the proof of Proposition~\ref{Prop:MonomialEig}}
\label{Sec:MatrixEndoProofProp:MonomialEig}
\subsection{Spectral radii of $N$-uples of matrices}
In the sequel, we fix the usual Euclidean norm on $\r^n$, and on $n\times n$-matrices:
\begin{definition}\label{Defi:NormRn}
Let $n\ge 1$. 
\begin{enumerate}[leftmargin=*]
\item
We endow $\r^n$ will the usual norm: \[\lVert x\rVert=\sqrt{\sum_{i=1}^n x_i^2},\text{ for each } 
x=(x_1, \ldots, x_n) \in \r^n.\]

\item
This endows  the ring $ \Mat_n(\r)$ of $n\times n$-real matrices with the norm

\[\lVert M\rVert=\sup  \left.\left\{\frac{\lVert Mv\rVert}{\lVert v\rVert} \right| v\in \r^n\setminus \{0\}\right\}, \text{ for each }M\in  \Mat_n(\r).\]

\item The \emph{spectrum} of $M \in \Mat_n(\r)$ is the finite subset $\sigma(M) \subset \CC$ of eigenvalues of $M$.

\item \label{Defi:NormRn:Spectralradius}
The \emph{spectral radius} of $M\in  \Mat_n(\r)$ is defined by 
\[
	\rho(M)=\max_{\lambda \in \sigma(M)} \lvert \lambda \rvert
\]
and satisfies
\[
	\rho(M)=\lim_{n\to \infty} \lVert  M^n \rVert^{1/n} \, .
\]
If $M = (m_{i, j})_{i, j = 1}^n$ and $N = (n_{i,j})_{i, j = 1}^n$ are matrices in $\Mat_n(\r)$ such 
that for each $(i, j)$ we have $0 \leq m_{i, j} \leq n_{i, j}$, then $\rho(M) \leq \rho(N)$.

\item We have a partial order on $\RR^n$ given by 
\[
x \leq y \quad \textrm{iff} \quad 
\textrm{$x_i \leq y_i$ for all $i = 1, \ldots, n$}
\]
where $x = (x_1, \ldots, x_n)$ and $y = (y_1, \ldots, y_n)$. Note that for $0 \leq x \leq y$
we have $\lVert x \rVert \leq \lVert y \rVert$.

\item For $M \in \Mat_n(\r)$ we denote by $\chi_M$ the characteristic polynomial of $M$.
\end{enumerate}
\end{definition}

\subsection{The Perron-Frobenius Theorem and its applications}
The Perron-Frobe\-nius theory was first established for matrices with positive coefficients, then generalised to irreducible matrices with non-negative coefficients and then to any matrices with non-negative coefficients.
There are three equivalent definitions of reducible matrices (see \cite[Vol.~2, Chap. XIII, \S1, Definitions~2,2',2'']{Ga1959The-theory-of-matr}). Let us recall one of them:
\begin{definition} \cite[Vol.~2, Chap. XIII, \S1, Definition~2']{Ga1959The-theory-of-matr}
For each $n\ge 1$, a matrix  $M\in \Mat_n(\r_{\ge 0})$ is called \emph{reducible} if there is a permutation matrix $S\in \GL_n(\Z)$ such that the matrix 
$SMS^{-1}\in \Mat_n(\r_{\ge 0})$ is block-triangular, i.e.
\[
	SMS^{-1} =
	\begin{pmatrix} 
		A & 0 \\ 
		C & D
	\end{pmatrix} 
\]
where $A,D$ are square matrices, and where the zero matrix has positive dimensions. 

A matrix  $M\in \Mat_n(\r_{\ge 0})$ is called \emph{irreducible} if it is not reducible.
\end{definition}
\begin{lemma}\label{Lemm:IrredRedMatrix}\cite[Vol.~2, Chap. XIII, \S4]{Ga1959The-theory-of-matr}
For each reducible matrix $M\in \Mat_n(\r_{\ge 0})$, there is a permutation matrix $S\in \GL_n(\Z)$ such that $SMS^{-1}$ is a lower triangular block-matrix \[\begin{pmatrix}  A_{1,1} & 0  & \cdots & 0 \\
	A_{2,1} & A_{2,2} & \ddots & 0  \\
	\vdots & \ddots & \ddots & 0  \\
	A_{m,1} & \cdots &   A_{m,m-1}& A_{m,m}\end{pmatrix}\]
	where $A_{1,1},\ldots,A_{m,m}$ are irreducible matrices.
\end{lemma}
\begin{theorem}[Perron-Frobenius Theorem]\label{Thm:PerronFrobenius} \cite[Vol.~2, Chap. XIII, \S2 and \S3, 
Theorems~2 and~3]{Ga1959The-theory-of-matr}
For each $M\in \Mat_n(\r_{\ge 0})$, there exists an eigenvector $v\in (\r_{\ge 0})^n\setminus\{0\}$ to the eigenvalue $\rho(M)$. If $M$ is moreover irreducible, we can choose $v$ in $(\r_{> 0})^n$.
\end{theorem}

\begin{theorem}[Theorem of Lind on weak-Perron numbers]\label{LindTheorem}For each $\lambda\in \r$, the following conditions are equivalent:
\begin{enumerate}
\item\label{weakPerron1}
$\lambda$ is a weak Perron number $($see Definition~$\ref{def.Perron_number})$;
\item\label{weakPerron2}
$\lambda$ is the spectral radius of a non-zero square matrix with non-negative integral coefficients;
\item\label{weakPerron3}
$\lambda$ is the spectral radius of an irreducible square matrix with non-negative integral coefficients;
\item\label{weakPerron4}
$\lambda > 0$ and $\lambda^m$ is a Perron number for some $m\ge 1$.
\end{enumerate}
\end{theorem}

\begin{proof}
The equivalence between \ref{weakPerron1} and \ref{weakPerron3} follows from \cite[Theorem 3, page 291]{Lind}, and the equivalence between \ref{weakPerron2} and \ref{weakPerron3} follows from Lemma~\ref{Lemm:IrredRedMatrix}. 
The equivalence between \ref{weakPerron1} and \ref{weakPerron4} can be found for 
instance in \cite[Lemma 4]{Schinzel} or \cite[Theorem 2]{Brunotte}.
\end{proof}

As a consequence of Corollary~\ref{Coro:Computelambda} and of the Perron-Frobenius theorem, we obtain the following result (which is classical, see for instance \cite{FavreWulcan,Lin}):
\begin{corollary}\label{Coro:MonomialEndo}
For each matrix  $M=(m_{i,j})_{i,j=1}^n\in \Mat_n(\NN)$ and for each
$(\alpha_1,\ldots,\alpha_n)\in (\k^*)^n$, the monomial endomorphism
\[
	f_M=(\alpha_1 x_1^{m_{1,1}}\cdots x_n^{m_{1,n}},\cdots, \alpha_n x_1^{m_{n,1}}\cdots x_n^{m_{n,n}})
	\in \End(\AA^n)
\]
is dominant if and only if $\det(M)\not=0$. In this case,  the dynamical degree of $f_M$ is equal to the spectral radius of $M$:
\[\lambda(f_M)=\rho(M)\in \r_{\ge 1}.\]
\end{corollary}
\begin{proof}
Note that the endomorphism $f_M \in \End(\AA^n)$ 
restricts to an endomorphism $h_M \in \End((\AA^1 \setminus \{ 0 \})^n)$.

If $\det(M)=0$, any non-zero element of the kernel of the transpose of $M$ 
gives rise to a non-constant element $p$ in the Laurent polynomial ring $\k[x_1^{\pm}, \ldots, x_n^{\pm}]$ 
such that  $p \circ h_M$ is constant, so $h_M$ and thus $f_M$ is not dominant.
We then assume that $\det(M)\not=0$. 
This implies that $h_M \in \End((\AA^1 \setminus \{0\})^n)$
is surjective on $\overline{\k}$-points and thus $f_M$ is dominant. In particular, $\lambda(f_M) \geq 1$.
Thus we only have to show that
$\lambda(f_M)=\rho(M)$.
By the Perron-Frobenius-Theorem (Theorem~\ref{Thm:PerronFrobenius}), there exists an eigenvector 
$\mu\in (\mu_1,\ldots,\mu_n)\in (\r_{\ge 0})^n$ of $M$ to the eigenvalue $\rho(M)$.
Since the spectral radius of $M$ and the dynamical degree of $f_M$ do not change if we conjugate $M$
with a permutation matrix, we may assume that there is $m < n$ 
such that $\mu_1 = \ldots = \mu_m = 0$
and $\mu_i > 0$ for each $i \geq m+1$. Since $(f_M)^r = f_{M^r}$ we get for each $r \geq 1$ and each
$i \in \{1, \ldots, n\}$ that $\deg_\mu(((f_M)^r)_i) = (M^r \mu)_i = \rho(M)^r \mu_i$.
This implies that $\deg_\mu((f_M)^r)=\rho(M)^r$ for each $r \geq 1$. Thus 
$f_M$ is $\mu$-algebraically stable and $\deg_\mu(f_M) = \rho(M) < \infty$. By Corollary~\ref{Coro:Computelambda}\ref{fmEndAm}, we may write
\[
	M = 
	\left(
		\begin{array}{c|c}
			\hat{M}	& 0 \\
			\hline
			\ast & \ast
		\end{array}
	\right)
\]
where $\hat{M} \in \Mat_m(\NN)$ with $\det(\hat{M}) \neq 0$. By induction,
the endomorphism $f_{\hat{M}} \in \End(\AA^m)$ satisfies 
$\lambda(f_{\hat{M}}) =  \rho(\hat{M}) \leq \rho(M)$. By Corollary~\ref{Coro:Computelambda}\ref{lambdhatfequal},\ref{lambdhatfsmall} we get then
$\lambda(f_M) = \deg_\mu(f_M) = \rho(M)$.
\end{proof}

\begin{corollary}\label{Cor:rhosmallerthandeg}
For each endomorphism $f\in \End(\A^n)$  and each matrix $M\in \Mat_n(\NN)$ that is contained in $f$, we have $ \rho(M)\le \deg(f)$.
\end{corollary}
\begin{proof}
By the Perron-Frobenius-Theorem (Theorem~\ref{Thm:PerronFrobenius}), there exists an eigenvector $\mu=(\mu_1,\ldots,\mu_n)\in (\r_{\ge 0})^n$ of $M$ to the eigenvalue $\rho(M)$. Hence, $\sum_{j=1}^n m_{i,j}\mu_j=\rho(M)\mu_i$ for each $i\in\{1,\ldots,n\}$. By choosing
an integer $r\in \{1,\ldots,n\}$ such that $\mu_r=\max\{\mu_1,\ldots,\mu_n\}$, we obtain 
\[\rho(M)\mu_r=\sum_{j=1}^n m_{r,j}\mu_j\le \mu_r\sum_{j=1}^n m_{r,j}.\]
The coefficient of the monomial $\prod_{j=1}^n x_j^{m_{r,j}}$ in $f_r$ is nonzero (as $M$ is contained in $f$, see Definition~\ref{Defi:Matrixcontained}). This monomial has degree $\sum_{j=1}^n m_{r,j}$, so $\deg(f)\ge \sum_{j=1}^n m_{r,j}$. As $\mu_r>0$, this gives $\rho(M)\le \deg(f)$.
\end{proof}

In the following we will use the next basic property of Handelman numbers. It 
is a straightforward application of Descarte's Rule of Signs, see e.g.~\cite[p.91]{St1986A-source-book-in-m}:

\begin{lemma}[Basic property of Handelman numbers]\label{Lemm:Handelman}
	Let $n \geq 1$. For each $(a_0,\ldots,a_{n-1})\in (\r_{\ge 0})^n\setminus \{0\}$, the polynomial
	$x^n-\sum_{i=0}^{n-1} a_i x^i\in \r[x]$
	has a unique positive real root. In particular, a Handelman number has no other 
	positive real Galois conjugate.
\end{lemma}
\begin{corollary}
	\label{Cor:Handelman_implies_weak_Perron}
	Each Handelman number is a weak Perron number.
\end{corollary}
\begin{proof}
	Let $\lambda \in \r_{> 0}$ be a Handelman number. There exists
	$(a_0, \ldots, a_{n-1}) \in \ZZ^n \setminus \{ 0 \}$ such that $\lambda$ is a root
	of $P(x) = x^n - \sum_{i=0}^{n-1} a_i x^i \in \ZZ[x]$. By Lemma~\ref{Lemm:Handelman}, all
	roots of $P$, except $\lambda$, are either non-real or real and non-positive. Since $P$ is the characteristic 
	polynomial of the matrix
	\[
	A = \begin{pmatrix}  
	a_{n-1} & \cdots & a_1 & a_0 \\
	1 & \cdots & 0 & 0  \\
	\vdots & \ddots & \vdots & \vdots  \\
	0 & \cdots & 1 & 0
	\end{pmatrix} \in \Mat_n(\r_{\geq 0}) \,  ,
	\]
	it follows by the Perron-Frobenius-Theorem (Theorem~\ref{Thm:PerronFrobenius}) that the spectral
	radius of $A$ is equal to $\lambda$. This implies that $\lambda$ is a weak Perron number (Theorem~\ref{LindTheorem}).
\end{proof}

\subsection{Sequences of matrices}
To study endomorphisms of $\A^n$, we will need to consider finite sets of elements of $\Mat_n(\r)$ 
that have the property
that we can exchange rows. In order to take the norm on such sets, we will have to see them ordered, and thus see these in $\Mat_n(\r)^N$ for some $N\ge 1$.

\begin{notation}\label{Not:MnN}
	Let $n,N\ge 1$. We denote by $\widehat{\mathcal{M}}_{n,N}\subset \Mat_n(\r)^N$ 
	the $\r$-vector subspace of $N$-tuples $(M_1,\ldots,M_N)$ that have the following property:
	\begin{center}
	{\it For each $i,j\in \{1,\ldots,N\}$ and each
	$l \in \{1,\ldots,n\}$, the replacement of 
	the $l$-th row of $M_i$ with the $l$-th row of $M_j$ gives a matrix which lies in 
	$\{M_1,\ldots,M_N\}$.}
	\end{center}
	
	We then denote by $\mathcal{M}_{n,N}\subset \widehat{\mathcal{M}}_{n,N}$ the subset that consists
	of the 
	$N$-tuples $(M_1,\ldots,M_N)$ where $M_1,\ldots,M_N$ are $N$ distinct matrices with non-negative coefficients.
\end{notation}
	
\begin{remark}
	\label{rem.matrices_contained_in_M}
	If $f\in \End(\A^n)$ is an endomorphism, then there exists 
	some integer $N\ge 1$ and some $N$-tuple
	$(S_1,\ldots,S_N)\in \mathcal{M}_{n,N}$ such that
	$\{S_1,\ldots,S_N\}$ is the set of matrices that are 
	contained in $f$ (as in Definition~\ref{Defi:Matrixcontained}). 
\end{remark}
	
The following two lemmas build the key ingredients for proving the existence
of maximal eigenvectors of endomorphisms of $\AA^n$ in the next subsection (see Proposition~\ref{prop.StepOne}). This
eventually leads then to a proof of Proposition~\ref{Prop:MonomialEig}.

\begin{lemma}
	\label{lem.D_is_dense}
	Let $n,N\ge 1$. For each $M=(M_1,\ldots,M_N)\in  \mathcal{M}_{n,N}$, there exists a sequence ${(D_t)}_{t\in \mathbb{N}}$ of elements $D_t=(D_{t,1},\ldots,D_{t,N})\in \mathcal{M}_{n,N}$ that converges towards $M$ $($with respect to the topology of $\Mat_n(\r)^N$ that is given by the norm as in Definition~$\ref{Defi:NormRn})$ and such that for each $t\in \mathbb{N}$, there is no complex number which is an eigenvalue of two elements of $D_{t,1},\ldots,D_{t,N}$.
\end{lemma}

\begin{proof}
	The result being trivially true for $N=1$, we will assume $N\ge 2$.
For each $i\in \{1,\ldots,n\}$, we denote by 
 $\Gamma_i \subset \r^n$ the finite set of $i$-th rows of the matrices $M_1,\ldots,M_N$:
\[ \Gamma_i=\{r\in \r^n \mid r\text{ is the $i$-th row of one of the matrices }M_1,\ldots,M_N\}.\]
We then write  $\Gamma_i=\{r_{i,1},\ldots,r_{i,s_i}\}$, where $s_i\ge 1$ is the cardinality of  $\Gamma_i$.

As all matrices $M_1,\ldots,M_N$ are pairwise distinct and as one can ``exchange rows'' (see Notation~\ref{Not:MnN}), we have $N=s_1\cdot \cdots s_n$, and obtain a unique $\r$-linear map
\[
	\varphi\colon \prod_{i=1}^n (\r^n)^{s_i}\to \widehat{\mathcal{M}}_{n,N}
\]
with the following properties:
\begin{enumerate}
\item[$(1)$]
For each $k \in \{1, \ldots, N \}$, the composition of $\varphi$ with 
the projection map $\pi_k \colon \Mat_n(\r)^N\to \Mat_n(\r)$ 
onto the $k$-th factor is of the form
\[\begin{array}{rccc}
\pi_k\circ \varphi\colon& \prod_{i=1}^n (\r^n)^{s_i}&\to& \Mat_n(\r)\\
&(v_{i,j})_{1\le i\le n, 1\le j\le s_i} & \mapsto &
\begin{pmatrix}
			v_{1, j_1} \\
			\vdots \\
			v_{n, j_n}
		\end{pmatrix}
		\end{array}\]
		where $j_i\in \{1,\ldots,s_i\}$ for each $i\in \{1,\ldots,n\}$.
\item[$(2)$]
$(M_1,\ldots,M_N)=\varphi((r_{i,j})_{1\le i\le n, 1\le j\le s_i})$.
\end{enumerate}
Indeed, the possibilities for maps $\pi_k \circ \varphi$ as in $(1)$ are parametrised by the $N$ possible choices of $j_i\in \{1,\ldots,s_i\}$ for each $i\in \{1,\ldots,n\}$, and by $(2)$ the image of $(r_{i,j})_{1\le i\le n, 1\le j\le s_i}$ by the maps $\pi_1 \circ \varphi, \ldots, \pi_N \circ \varphi$
give the matrices $M_1,\ldots,M_N$; this gives the existence and the unicity of $\varphi$.

We now identify $\prod_{i=1}^n (\r^n)^{s_i}$  with the real locus $X(\r)$ of the affine space $X=\A^{n\sum s_i}$.

For any two matrices $A,B\in \Mat_n(\r)$, the resultant of the characteristic polynomials $\chi_A$ and $\chi_B$ is denoted by $r(A,B)$. Recall that $r(A, B) = 0$ if and only if $A$ and $B$ have a common eigenvalue. Hence, for any distinct $a,b \in \{1,\ldots,N\}$, the set 
\[
	Z_{a,b}=\left.\left\{x\in \prod_{i=1}^n (\r^n)^{s_i} \right| \begin{array}{l}\text{the matrices }\pi_{a}(\varphi(x))\text{ and }\pi_{b}(\varphi(x))\\
	\text{have a common eigenvalue}\end{array}\right\}
\]
corresponds to the  elements of $X(\RR)$ that satisfy 
one polynomial equation $P_{a,b}\in \r[X]$. 

We now prove that $P_{a,b}\not=0$, or equivalently that $Z_{a,b}\not=
X(\r)=\prod_{i=1}^n (\r^n)^{s_i}$, by showing that 
$\pi_{a}(\varphi(x))$ and $\pi_{b}(\varphi(x))$ have no common eigenvalue for at least one $x\in X(\r)$. 
We consider $j_1,\ldots,j_n$ and $j_1',\ldots,j_n'$ so that $\pi_a\circ \varphi$ and $\pi_b\circ \varphi$ are respectively given by
\[\begin{array}{ccc}
 \prod_{i=1}^n (\r^n)^{s_i}&\to& \Mat_n(\r)\\
(v_{i,j})_{\scriptsize\begin{array}{l} 1\le i\le n\\ 1\le j\le s_i\end{array}} & \mapsto &
\begin{pmatrix}
			v_{1, j_1} \\
			\vdots \\
			v_{n, j_n}
		\end{pmatrix}
		\end{array}
	\qquad \textrm{and} \qquad	
		\begin{array}{ccc}
 \prod_{i=1}^n (\r^n)^{s_i}&\to& \Mat_n(\r)\\
(v_{i,j})_{\scriptsize\begin{array}{l} 1\le i\le n\\ 1\le j\le s_i\end{array}} & \mapsto &
\begin{pmatrix}
			v_{1, j_1'} \\
			\vdots \\
			v_{n, j_n'}
		\end{pmatrix}
		\end{array} \, .
\]
Since the matrices $M_a$ and $M_b$ are distinct, 
the linear maps $\pi_a \circ \varphi$ and $\pi_b\circ \varphi$ are also distinct. There is thus  $l\in \{1,\ldots,n\}$ such that $j_l\not=j_l'$.
		Suppose first that $l=1$, i.e.~$j_1\not=j_1'$. 
		We may choose $x\in X(\r)$ such that 
\[
\pi_{a}(\varphi(x))=\left(\begin{array}{c|c}
0 & 1 \\
\hline
I_{n-1} & 0
\end{array}\right) \quad \textrm{and} \quad \pi_{b}(\varphi(x))=\left(\begin{array}{c|c}
0 & 0 \\
\hline
I_{n-1} & 0
\end{array}\right) \, .
\]
 These matrices have characteristic polynomials $t^n-1$ and  $t^n$, respectively. 
If $l>1$, we simply consider conjugation of the above matrices by permutations. In all cases, we find an
$x\in X(\r)$ such that  $\pi_{{a}}(\varphi(x))$ and $\pi_{{b}}(\varphi(x))$ are matrices without common eigenvalue in $\mathbb{C}$. This shows that $Z_{a,b} \neq X(\RR)$, i.e.~$P_{a,b} \neq 0$.

The product of all polynomials $P_{a,b}$ with distinct $a,b \in \{1,\ldots,n\}$ gives a non-zero 
polynomial $P\in \r[X]$.
We can thus  take a real affine linear map $\ell\colon \A^1 \to X =\A^{n\sum s_i}$ such that $\ell(0)=(r_{i,j})_{1\le i\le n, 1\le j\le s_i}$, such that the coordinates of 
$\ell(\r_{\ge 0})$ are non-negative and such that the restriction of 
$P$ to $\ell(\r)$ is non-zero. 
We obtain that $P(\ell(\frac{1}{n})) \neq 0$ for any sufficiently large 
positive integer $n$. It suffices then to fix a sufficiently large $c\ge 1$ and 
to define $D_t=\varphi(\ell(\frac{1}{t+c}))$ for each integer $t\ge 0$.
\end{proof}

\begin{lemma}
	\label{lem.distincteg.thenok}
	Let $S = (S_1, \ldots, S_N) \in \mathcal{M}_{n,N}$ and let $v\geq 0$ be an 
	eigenvector of
	$S_1$ to the eigenvalue $\lambda \geq 0$.
	Suppose moreover that $\lambda>\rho(S_i)$ for each $i\in \{2,\ldots,N\}$.
	Then $S_i v \leq \lambda v$ for each $i\in \{1,\ldots,N\}$.
\end{lemma}

\begin{proof}
	Assume for contradiction that there is $i\in \{2,\ldots,N\}$ such that  $S_i v \not\leq \lambda v$. 
	Denote by $v_j$ the $j$-th component of $v$ for each 
	$j \in \{1, \ldots, n\}$.
	Since we may replace each row $R_j$ in $S_i$ such that $R_j v < \lambda v_j$
	with the $j$-th row from
	$S_1$ and still get an element in $\{S_1, \ldots, S_N\}$, 
	we may assume that $S_i v \geq \lambda v \geq 0$. As the coefficients of $v$ and $S_i$ are non-negative, we obtain by induction that $(S_i)^r v \geq \lambda^r v \geq 0$ for each $r\ge 1$. In particular,
	\[
		\lVert (S_i)^r\rVert \geq \frac{\lVert (S_i)^r v \rVert}{\lVert v  \rVert} \geq \lambda^r
	\]
	and we obtain  $\rho(S_i)=\lim_{r\to \infty} \lVert (S_i)^r \rVert^{1/r} \geq \lambda$. This
	contradicts the assumption that $\lambda>\rho(S_i)$.
\end{proof}
\subsection{Existence of maximal eigenvectors of  endomorphisms of $\AA^n$}
\begin{proposition}
	\label{prop.StepOne}
	For each $n,N\ge 1$ and each $S=(S_1,\ldots,S_N) \in \mathcal{M}_{n,N}$, there exists $j \in \{1,\ldots,N\}$ and
	an eigenvector $v \in (\r_{\geq 0})^n \setminus \{ 0 \}$ of $S_{j}$ to the eigenvalue 
	$\lambda=\max \{\rho(S_1),\ldots,\rho(S_N)\}$ such that for each $i \in \{1,\ldots,N\}$ we have
	\[
		S_i v \leq S_{j} v= \lambda v  \, .
	\]
\end{proposition}

\begin{proof}
	Let $S = (S_1, \ldots, S_N) \in \mathcal{M}_{n,N}$. 
	By Lemma~\ref{lem.D_is_dense}, 
	there exists a sequence ${(D_t)}_{t\in \mathbb{N}}$ of elements $D_t=(D_{t,1},\ldots,D_{t,N})\in \mathcal{M}_{n,N}$ that converges towards $S$ and such that for each $t\in \mathbb{N}$, there is no complex number which is an eigenvalue of two elements of $D_{t,1},\ldots,D_{t,N}$.
	In particular, $\rho(D_{t, i}) \neq \rho(D_{t, j})$ for distinct $i, j$ by the Perron-Frobenius-Theorem (Theorem~\ref{Thm:PerronFrobenius}).
	
	By possibly replacing this sequence with a subsequence, we may assume that there is a
	$j\in \{1,\ldots,N\}$ such that $\rho(D_{t, j}) > \rho(D_{t, i})$ for all $i \in \{1,\ldots,N\}\setminus \{j\}$ and each $t\in \mathbb{N}$. 
	After exchanging the ordering of $S_1,\ldots,S_N$, we may assume that $j = 1$.
	For each $i\in \{1,\ldots,N\}$, the sequence $(D_{t,i})_{t\in \mathbb{N}}$ converges towards $S_i$, so $(\rho(D_{t,i}))_{t\in \mathbb{N}}$ converges towards $\rho(S_i)$  \cite[Theorem in Appendix~A]{Os1973Solution-of-equati}.
	In particular, $\rho(S_1)=\lambda=\max\{\rho(S_1),\ldots,\rho(S_n)\}$. By the Perron-Frobenius-Theorem 
	(Theorem~\ref{Thm:PerronFrobenius}), 
	there is for each $t\in \mathbb{N}$ 
	an eigenvector $v_t \geq 0$ of $D_{t,1}$ to the eigenvalue $\rho(D_{t,1})$.
	Lemma~\ref{lem.distincteg.thenok} then gives 
	for each $i\in \{1,\ldots,N\}$ and each $t\in \mathbb{N}$
	\[
		D_{t,i} v_t \leq \rho(D_{t,1}) v_t \, .
	\]
	Now, we may assume that $\lVert v_t \rVert = 1$ for all $t$ (after normalizing $v_t$). 
	Let
	\[
		\mathbb{S}^{n-1} = \set{w \in \RR^n}{\lVert w \rVert =1} \, .
	\]
	Since $\mathbb{S}^{n-1}$ is compact (with respect to the Euclidean topology), we may take a subsequence and
	assume that $(v_t)_{t\in \mathbb{N}}$ converges to a $v \geq 0$ in $\mathbb{S}^{n-1}$. Thus we get
	\[
		\lambda v=\rho(S_1) v = \lim_{t\to\infty} \rho(D_{t, 1}) v_t = \lim_{t \to \infty} D_{t, 1} v_t = S_1 v
	\]
	and for each $i\in \{1,\ldots,N\}$
	\[
		S_i v  = \lim_{t\to\infty} D_{t, i} v_t \leq \lim_{t\to\infty} \rho(D_{t, 1}) v_t = 
		\rho(S_1) v =\lambda v
		\quad \, .
	\]
	This finishes the proof of the proposition.
\end{proof}

\begin{proof}[Proof of Proposition~$\ref{Prop:MonomialEig}$]
	By Remark~\ref{rem.matrices_contained_in_M}, there exists $(S_1, \ldots, S_N) \in  \mathcal{M}_{n, N}$
	such that $\{ S_1, \ldots, S_N \}$ is the set of matrices contained in $f$. By Proposition~\ref{prop.StepOne}
	there exists $j \in \{1, \ldots, N\}$  and
	an eigenvector $\mu=(\mu_1,\ldots,\mu_n) \in (\r_{\geq 0})^n \setminus \{ 0 \}$ of $S_{j}$ to the
	eigenvalue $\theta=\max \{\rho(S_1),\ldots,\rho(S_N)\}$ such that
	$S_i \mu \leq S_{j} \mu= \theta \mu$ for each $i \in \{1,\ldots,N\}$. 
	We now prove that this implies that $\deg_\mu(f_l) = \theta \mu_l$ for each 
	$l \in \{1, \ldots, n\}$, 
	which shows that $\mu=(\mu_1,\ldots,\mu_n)$ is a maximal eigenvector of $f$, and thus proves~\ref{Prop:MonomialEig1}. 
	For each monomial $m=\chi x_1^{r_1}\cdots x_n^{r_n}$ of $f_l$ with $\chi\in \k^*$ there is a matrix $S_i$ 
	with its $l$-th line equal to $(r_1\  r_2\ \cdots\ r_n)$. The $l$-th component of $S_i \mu$ is equal to $r_1\mu_1+\cdots+r_n\mu_n=\deg_\mu(m)$. The inequality $S_i\mu \le \theta \mu$ then yields $\deg_\mu(m)\le \theta\mu_l$. As this holds for each monomial of $f_l$, we obtain $\deg_\mu(f_l) \leq \theta \mu_l$. The  equality follows from $S_j\mu =\theta \mu$, since the monomial $m$ 
	that corresponds
	to the $l$-th row of $S_j$ has  $\mu$-degree equal to $\theta\mu_l$.

	We now prove~\ref{Prop:MonomialEig2}. 
	The dominance of $f$ implies that $1\le \deg(f^r)$ for each $r$ and this in turn gives
	$1 \leq \lambda(f)$. The inequality
	$\theta\le \deg(f)$ follows from Corollary~\ref{Cor:rhosmallerthandeg}, so we only need to prove $\lambda(f)\le \theta$. This is done by  induction on $n$. 
	If $n = 1$, then $\mu \in (\r_{>0})^1$
	and the statement follows from Proposition~\ref{Prop:Dynamicaldegreeestimate}\ref{lambdasmaller}. Now, 
	let $n > 1$.
	We may assume (after a permutation of the coordinates) that 
	$\mu_1 \leq \mu_2 \leq \ldots \leq \mu_n$. Now, let $m \in \{0, \ldots, n-1\}$ with $\mu_i = 0$ for $i \leq m$
	and $\mu_i > 0$ for $i > m$. From Remark~\ref{Rem:DegreeMu}\ref{degmudeg} we get
	\[
		\lambda_2 := 
		\lim_{r \to \infty} \deg_{x_{m+1}, \ldots, x_n}(f^r)^{\frac{1}{r}} =
		\lim_{r \to \infty} \max_{i \in \{1, \ldots, n\}} \deg_{\mu}((f^r)_{i})^{\frac{1}{r}} \, .
	\]
	From Lemma~\ref{Lem:HomDecEnd} we get that for each $i \in \{1, \ldots, m \}$, the element
	$f_i$ is a polynomial in the variables $\{x_1, \ldots, x_m \}$. Thus we get from
	Lemma~\ref{lem.DinhNGuyen} that $\lambda(f) = \max \{\lambda_1, \lambda_2 \}$ where
	\[
		\lambda_1 = \lambda(\hat{f}) = \lim_{r \to \infty} \deg( \hat{f}^r)^{\frac{1}{r}}
		\quad \textrm{and} \quad \hat{f} := (f_1, \ldots, f_m) \in \End(\AA^m) \, .
	\]
	Since $m \leq n-1$, by induction hypothesis we have
	\[
		\lambda_1 \leq \theta_1 := 
		\max \set{|\xi| \in \r}{ \textrm{$\xi$ is an eigenvalue of a matrix that is
				contained in $\hat{f}$}} \, .
	\]
	Note that each eigenvalue of a matrix that is contained in $\hat{f}$ is an eigenvalue
	of a matrix that is contained in $f$. Thus we get $\theta_1 \leq \theta$. From Lemma~\ref{Lemma:Dynamicaldegreeestimate}\ref{degmudyn}, it follows that 
	$\lambda_2 \leq \theta$.
	In summary we proved that $\lambda(f) = \max\{\lambda_1, \lambda_2\} \leq \theta$, i.e.~\ref{Prop:MonomialEig2} holds for $n$.
	
	We now prove~\ref{Prop:MonomialEig3}. We take a maximal eigenvector $\mu$ of $f$. As $\deg_\mu(f_i)= \theta \mu_i$ for each $i\in \{1,\ldots,n\}$, we have $\deg_\mu(f)=\theta$.  If $\theta=1$, \ref{MonomialEig3i} follows from \ref{Prop:MonomialEig2} and \ref{MonomialEig3ii} is trivially true, so we may assume that $\theta>1$. 
	If $f$ is $\mu$-algebraically stable, then Lemma~\ref{Lemma:Dynamicaldegreeestimate}\ref{degmudyn2} gives $\lambda_2 = \theta$ and
	thus $\lambda(f) = \theta$, so \ref{MonomialEig3i} is proven. Conversely, if $\mu\in (\r_{>0})^n$ and $\lambda(f)=\theta>1$, then $f$ is $\mu$-algebraically stable by Proposition~\ref{Prop:Dynamicaldegreeestimate}\ref{lambdaequal}. This achieves the proof of \ref{MonomialEig3ii}. As $\theta = \deg_\mu(f) \in \RR_{\geq 0}$ (i.e.~is not equal to $+\infty$),
	\ref{MonomialEig3iii} is a direct consequence of Lemma~\ref{Lemma:Dynamicaldegreeestimate}\ref{degmudyn2}.
\end{proof}

We now give an example that shows that the implication of Proposition~\ref{Prop:MonomialEig}\ref{Prop:MonomialEig3}\ref{MonomialEig3i} is not an equivalence.

\begin{example}\label{ExamplePropBImpNotequi}
We consider the automorphism 
\[
	f = (f_1, f_2, f_3, f_4)=((x_1)^2+x_2,x_1,x_3+(x_3+x_4)^2,x_4-(x_3+x_4)^2)\in \Aut(\A^4) \, .
\] 
As $\deg(f)=2$, the maximal eigenvalue $\theta$ of $f$ (see Definition~\ref{Defi:Matrixcontained}) satisfies $\theta\le 2$ (Corollary~\ref{Cor:rhosmallerthandeg}). Moreover, $\theta=2$, as the matrix 
\[
	\begin{pmatrix} 
		2 & 0 & 0&0\\ 
		1 & 0 & 0&0\\ 
		0 & 0 &1&0\\ 
		0 & 0 & 0&1
	\end{pmatrix}
\] 
is contained in $f$. When we choose $\mu=(0,0,1,1)$, we get $\deg_\mu(f)=2$, and we
see that $f$ is not $\mu$-algebraically stable, as $\deg_\mu(f^2)=2<4$. 
Moreover,  
$\deg_\mu(f_i) = 0$ for $i \in \{1, 2\}$ and $\deg_\mu(f_i) = 2$ for $i \in \{3, 4\}$. Thus $\mu$
is a maximal eigenvector of $f$ (see Definition~\ref{Defi:Matrixcontained}).
However, $\lambda(f)=\theta$. Indeed, $\lambda(f)\le \deg(f)=2$, and $((x_1)^2+x_2,x_1)$ is algebraically stable for the standard degree, as its homogeneous part of degree $2$ is $((x_1)^2,0)$, which satisfies $((x_1)^2,0)^r=((x_1)^{2^r},0)$ for each $r\ge 1$ (see Proposition~\ref{Prop:Dynamicaldegreeestimate}).
\end{example}

\section{Explicit calculation of dynamical degrees of affine-triangular automorphisms}
\label{Sec:Explicit}
In this section, we apply Proposition~\ref{Prop:MonomialEig} to compute the dynamical degrees of affine-triangular dominant endomorphisms of $\A^n$. We prove Proposition~\ref{Prop:PermutationElementary}, 
which implies Propositions~\ref{Prop:Jonsson} and~\ref{Prop:PermElementary}.
\begin{notation} 
	We denote by $\TEnd(\A^n)$ and $\TAut(\A^n)$ (respectively  $\EEnd(\A^n)$ and $\EAut(\A^n)$) the monoid and group of triangular (respectively elementary) endomorphisms and automorphisms of $\A^n$. We denote by $\Aff(\A^n)$ the group of affine automorphisms of $\A^n$ and by $\Sym(\A^n)\subset \Aff(\A^n)$ the group of permutations of the coordinates.
\end{notation}

\subsection{From affine-triangular to permutation-triangular endomorphisms}

We can restrict ourselves to permutation-triangular endomorphisms, as the next simple result shows.
		
\begin{proposition}\label{Affine-Triangular-Permutation}
	Each affine-triangular endomorphism of $\A^n$ is conjugate by an element of $\Aff(\A^n)$ to a permutation-triangular endomorphism.
	\end{proposition}
\begin{proof}
We take $\alpha\in \Aff(\A^n)$ and $\tau\in \TEnd(\A^n)$ and show that we can conjugate $f=\alpha\circ \tau$ to a permutation-triangular endomorphism by an element of $\Aff(\A^n)$.

Let $p=(p_1,\ldots,p_n)\in \A^n$ be the point such that $\alpha(p)=0$ and consider the translation $\tau_p=(x_1+p_1,\ldots,x_n+p_n) \in \Aff(\A^n)\cap \TAut(\A^n)$.
Then $\alpha'=\alpha\circ \tau_p \in \Aff(\A^n)$ fixes the origin $(0,\ldots,0)\in \A^n$. We then replace $\alpha$ with $\alpha'$ and $\tau$ with $\tau_p^{-1} \circ \tau$, and may assume that $\alpha$ belongs to the subgroup $\GL_n=\GL_n(\k)\subset \Aff(\A^n)$ of elements that fix the origin.

	The group $B = \TAut(\AA^n) \cap \GL_n$ is a Borel subgroup of $\GL_n$.
	It consists of all lower triangular matrices. The so-called Bruhat decomposition of $\GL_n$:
	\[
		\GL_n =  B \Sym(\A^n) B
	\]
	yields  $\beta, \gamma \in B$ and $\sigma \in \Sym(\A^n)$ such that
	$\alpha = \beta \circ \sigma\circ \gamma$. This gives
	\[
		\beta^{-1} \circ f \circ \beta =\beta^{-1} \circ \alpha \circ \tau \circ \beta  
		= \sigma \circ \gamma \circ \tau \circ \beta
	\]
	where $\gamma \circ \tau \circ \beta \in \TEnd(\AA^n)$. This achieves the proof.
\end{proof}

\subsection{Permutation-elementary automorphisms}\label{Sec:PermEl}
Up to conjugation, each per\-mu\-ta\-tion-elementary automorphism has a particular form.
This shows the following easy observation.

\begin{lemma}
\label{Lem:Conjugation_Permutation_Elementary}
Let $n\ge 1$ and let $h\in \End(\A^{n+1})$ be  a permutation-elementary automorphism.
There is a permutation of the coordinates $\alpha\in\Sym(\A^{n+1})$ such that 
\[
	f=\alpha \circ h\circ \alpha^{-1}=(f_{1},\ldots,f_{m}, \xi x_{n+1}+p(x_1,\ldots,x_{n}),x_{m+1},\ldots,x_{n}),
\] 
where $0\le m\le n$,  $\{x_1,\ldots,x_m\}=\{f_1,\ldots,f_m\}$, $\xi\in \k^*$ and $p\in \k[x_1,\ldots,x_n]$.
\end{lemma}

\begin{proof}
	We write $h=\sigma\circ \tau$ where $\sigma \in \Sym(\A^{n+1})$ and $\tau\in \EAut(\A^{n+1})$. 
	We may choose $\alpha = (\alpha_1, \ldots, \alpha_{n+1})\in \Sym(\A^{n+1})$ such that $\alpha_{n+1} = x_{n+1}$ and $\alpha \circ \sigma \circ \alpha^{-1}$
	induces the following cyclic permutation on the last coordinates
	\[
		(\alpha \circ \sigma \circ \alpha^{-1})_{m+1} = x_{n+1} \, , \ 
		(\alpha  \circ \sigma \circ \alpha^{-1})_{m+2} = x_{m+1} \, , \
		\ldots \, , \ (\alpha \circ \sigma \circ \alpha^{-1})_{n+1} = x_{n} \, ,
	\]
	for some integer $m$ with $0\le m\le n$. This gives \[
	\alpha \circ \sigma \circ \alpha^{-1} 
	= (f_1, \ldots, f_m, x_{n+1}, x_{m+1}, \ldots, x_n)
	\]
	where
	$\{ x_1, \ldots, x_m \} = \{f_1, \ldots, f_m \}$.
	As $\alpha_{n+1}=x_{n+1}$, we obtain
	\[
	\alpha \circ \tau \circ \alpha^{-1} = (x_1, \ldots, x_{n}, \xi x_{n+1} + p(x_1, \ldots, x_n))
	\]
	for some $\xi \in \k^\ast$ and $p \in \k[x_1, \ldots, x_n]$. This implies that $\alpha \circ h \circ \alpha^{-1}$
	is equal to
	\[
	(\alpha \circ \sigma \circ \alpha^{-1}) \circ
	(\alpha \circ \tau \circ \alpha^{-1}) \\ 
	= (f_{1},\ldots,f_{m}, \xi x_{n+1}+p(x_1,\ldots,x_{n}),x_{m+1},\ldots,x_{n}) \, .
	\]
\end{proof}

We will need the following result to obtain 
Proposition~\ref{Prop:PermutationElementary} below.
Proposition~\ref{Prop:PermutationElementary} will be the key ingredient to show Proposition~$\ref{Prop:Jonsson}$ and Proposition~$\ref{Prop:PermElementary}$.

\begin{lemma}\label{Lem:IterationNotzero}
	Let $0\le m\le n$, let $\hat{f}=(f_{1},\ldots,f_m)\in \Aut(\A^m)$ 
	and let $q\in \k[x_1,\ldots,x_{n+1}]\setminus\{0\}$. For each $r\ge 1$, every component of $g^r$ is non-zero
	where  
	\[
		g=(f_1,\ldots,f_m,q,x_{m+1},\ldots,x_n)\in \End(\A^{n+1}) \, .
	\]
\end{lemma}

\begin{proof}
	For each $r\ge 1$, we write $g^r=((g^r)_1,\ldots,(g^r)_{n+1})$. The 
	result is true by assumption when $r=1$. For each $r\ge 1$ and $1 \leq i \leq m$ we have 
	$(g^r)_i = (\hat{f}^r)_i \not=0$. 
	
	As $(f_1,\ldots,f_m)\in \Aut(\A^m)$, we also have $(f_1,\ldots,f_m,x_{m+1},\ldots,x_{n})\in \Aut(\A^{n})$. In particular, $g$ is dominant if $q\not\in \k[x_1,\ldots,x_n]$, i.e.~if $\deg_{x_{n+1}}(q)\ge 1$.
	Thus we assume that 
	$q \in \k[x_1, \ldots, x_n] \setminus \{0\}$.
	
	Suppose first that $m=n$, in which case $g=(f_1,\ldots,f_m,q)$. 
	For each $r\ge 2$, we get
	$g^r=((g^r)_1,\ldots,(g^r)_{m},q((g^{r-1})_1,\ldots,(g^{r-1})_{m}))$. As $\hat{f}$ is dominant and $q$ is not the zero polynomial, every component of $g^r$ is not zero. 
	
	We then assume that $n>m$ and prove the result by induction on $n-m$. 
	As $(f_1,\ldots,f_m,x_{m+1},\ldots,x_{n})\in \Aut(\A^{n})$, there is a polynomial $h\in \k[x_1,\ldots,x_n]$ such that $h(f_1,\ldots,f_m,x_{m+1},\ldots,x_n)=q$, since
	$q \in \k[x_1,\ldots,x_n]$.
	 We denote by $\phi\colon \A^n\hookrightarrow \A^{n+1}$ the closed embedding that is given by
	 \[
	 	(x_1,\ldots,x_n)\mapsto (x_1,\ldots,x_m,h(x_1,\ldots,x_n),x_{m+1},\ldots,x_n)
	 \]
	 and we write 
	 $\tau=(f_1,\ldots,f_m,h,x_{m+1},\ldots,x_{n-1})\in \End(\A^n)$. We now prove that $g\circ \phi=\phi\circ \tau$: 
	 \begin{align*}
	 			&g\circ \phi(x_1,\ldots,x_n) \\
	 			&=(f_1,\ldots,f_m,q(x_1,\ldots,x_m,h,x_{m+1},\ldots,x_{n-1}),h,x_{m+1},\ldots,x_{n-1})\\
	 			&=(f_1,\ldots,f_m,h(f_1,\ldots,f_m,h,x_{m+1},\ldots,x_{n-1}),h,x_{m+1},\ldots,x_{n-1})\\
	 			&=\phi\circ \tau(x_1,\ldots,x_n).
	\end{align*}
	Hence,  $g^r\circ \phi=\phi\circ \tau^r$ for each $r\ge 1$. By induction, every component of $\tau^r$ is non-zero, so every component of $g^r$ is non-zero, except maybe the $(m+1)$-th one. But if the $(m+1)$-th component of $g^r$ were zero, then the $(m+2)$-th of $g^{r+1}$ would be zero, impossible as the $(m+2)$-th component of $g^{r+1}\circ \phi=\phi\circ \tau^{r+1}$ is not equal to zero.
\end{proof}

\begin{proposition}\label{Prop:PermutationElementary}
	Let $0\le m < n$, let $\hat{f}=(f_{1},\ldots,f_m)\in \Aut(\A^m)$, $\xi\in \k^*$ and $p\in \k[x_1,\ldots,x_n]$. 
Denote by $I\subset \NN^n$ the  finite subset of indices of the monomials of $p$, and define \[
	\theta =\max\left\{\lambda\in \r\left|\lambda^{n-m}=  \sum_{j=m+1}^n i_j \lambda^{n-j}
		\text{ for some } (i_1,\ldots,i_n)\in I\right\}\right..\]
	Then,
	\[f=(f_{1},\ldots,f_{m}, \xi x_{n+1}+p(x_1,\ldots,x_{n}),x_{m+1},\ldots,x_{n})\in \Aut(\A^{n+1})\]
	has the following properties:
	\begin{enumerate}[leftmargin=*]
	\item \label{Prop:PermutationElementary_special}
	If $\deg_{x_{m+1},\ldots,x_n}(p)\le1$, then $\lambda(f)=\lambda(\hat{f})$. 
	\item \label{Prop:PermutationElementary_general}
	If $\deg_{x_{m+1},\ldots,x_n}(p)\ge 2$, define
	\begin{align*}
		\mu &= (\mu_1, \ldots, \mu_{n+1}) = (0,\ldots,0,\theta^{n-m},\theta^{n-m-1},\ldots,\theta,1) \, ,
	\end{align*}
	i.e.~$\mu_{j}=0$ for $j\le m$ and $\mu_{j}=\theta^{n+1-j}$ for $j\ge m+1$.
	Then we have $\theta>1$, $\deg_\mu(f_{j}) = \theta \mu_{j}$ for each $j$ 
	$($in particular $\deg_\mu(f)=\theta)$ and $f$ is $\mu$-algebraically stable.
	If moreover $\lambda(\hat{f}) \leq \theta$ $($in particular, if $m = 0)$, 
	then $\lambda(f) = \theta$.
	\item \label{Prop:PermutationElementary_special_permutation}
	Assume $\{f_1, \ldots, f_m\} = \{x_1, \ldots, x_m\}$. If 
	$\deg_{x_{m+1},\ldots,x_n}(p)\le1$, then
	the maximal eigenvalue of $f$ is equal to $1$ and otherwise it is equal to $\theta$.
	\end{enumerate}
\end{proposition}
\begin{remark}\label{Remarkmntrivial}
The case $m=n$, not treated in  Proposition~$\ref{Prop:PermutationElementary}$, is rather trivial. We have $f=(f_1,\ldots,f_n,\xi x_{n+1}+p(x_1,\ldots,x_n))$ where $\{f_1,\ldots,f_n\}=\{x_1,\ldots,x_n\}$. Every matrix contained in $f$ is then a block-matrix with a $(n\times n)$-permutation-matrix and a $(1\times 1)$-matrix with a $0$ or a $1$ on the diagonal, so every eigenvalue is either $0$ or a root of unity. This implies that $\theta=1$ is the only possible maximal eigenvalue.
\end{remark}

	\begin{proof}[Proof of Proposition~$\ref{Prop:PermutationElementary}$]
	\ref{Prop:PermutationElementary_special}
	Since $\deg_{x_{m+1},\ldots,x_n}(p)\le1$,	
	Lemma~\ref{lem.DinhNGuyen} implies that 
	\[
		\lambda(f)=\max\{\lambda(\hat{f}),\lim_{r\to \infty}\deg_{x_{m+1},\ldots,x_{n+1}}(f^r)^{1/r}\}
		= \lambda(\hat{f})
	\] 	
	where by convention $\lambda(\hat{f}) = 1$ in case $m = 0$.

	\ref{Prop:PermutationElementary_general}:
	For each $i=(i_1,\ldots,i_n)\in I$, we set 
	\[
		p_i= \sum_{j=m+1}^n i_j x^{n-j} \in \Z[x]
	\] 
	and $q_i=x^{n-m}-p_i\in \Z[x]$. Then $\theta$ is the biggest real root of one of the polynomials in $\{q_i\mid i\in I\}$.	
	Note that $q_i$ is monic and of degree 
	$n-m > 0$. As $\deg_{x_{m+1},\ldots,x_n}(p)\ge 2$, there is $i=(i_1,\ldots,i_n)\in I$ such that $p_i(1)\ge 2$. This implies that $q_i(1)=1-p_i(1)<0$, so $q_i$ has a real root that is bigger than $1$. This proves that $\theta>1$. For each $i\in I$, we moreover have $q_i(\theta)\ge 0$, since $q_i$ has no real root bigger than $\theta$. This gives $\theta^{n-m}\ge p_i(\theta)$, with equality for at least one $i \in I$.
	
	We now prove that $\deg_\mu(f_{j})=\theta \mu_{j}$ for each $j\in \{1,\ldots,n+1\}$
	where $f_j$ denotes the $j$-th component of $f$:
	For each $j \in \{1,\ldots,m\}$ we have $\deg_\mu(f_{j})=0=\theta \mu_{j}$ 
	and for each $j \in \{m+2,\ldots,n+1\}$, we have 
	$\deg_\mu(f_{j})=\deg_\mu(x_{j-1})=\mu_{j-1}=\theta \mu_{j}$.
	We moreover have
	\begin{align*}
		\deg_{\mu}(f_{m+1})&= \max \left(\{\deg_\mu(x_{n+1})\}\cup 
		\Bigset{\sum_{j=m+1}^{n} {i_j\mu_j}}{(i_1,\ldots,i_n)\in \NN^n} \right) \\
		&=\max\left(\{1\}\cup \set{\theta\cdot p_i(\theta)}{i=(i_1,\ldots,i_n)\in \NN^n} \right)
		=\theta^{n-m+1}=\theta \mu_{m+1}.
	\end{align*}
	This gives in particular $\theta = \deg_\mu(f)$.
	
	It remains to prove that $f$ is $\mu$-algebraically stable, i.e.~that $\deg_\mu(f^r)=\theta^r$ for each $r\ge 1$; this will then give the result by Corollary~\ref{Coro:Computelambda}.
	
	 By Lemma~\ref{Lemma:Dynamicaldegreeestimate}\ref{degmudyn2}, this corresponds to ask that for each $r\ge 1$, there exists $j \in \{m+1,\ldots,n\}$ such that $(g^r)_{j} \not=0$, where $g=(g_1,\ldots,g_{n+1})\in 
	 \End(\A^{n+1})$ is 
	the $\mu$-leading part of $f$
	 and $(g^r)_{j}$ denotes the $j$-th component of $g^r$.
	 We observe that
	 \[g=(f_1,\ldots,f_m,g_{m+1},x_{m+1},\ldots,x_{n})\]
	 where $g_{m+1}\in \k[x_1,\ldots,x_{n+1}]\setminus \{0\}$.  The result then follows from Lemma~\ref{Lem:IterationNotzero}
	 
	\ref{Prop:PermutationElementary_special_permutation}: 
		The maximal eigenvalue of $f$ is the biggest real number that is an eigenvalue of one of the matrices contained in $f$. Each such matrix is either contained in $(f_{1},\ldots,f_{m}, \xi x_{n+1},x_{m+1},\ldots,x_{n})$, but then has spectral radius equal to $1$, or is contained in $(f_{1},\ldots,f_{m}, \prod_{j=1}^n x_j^{i_j},x_{m+1},\ldots,x_{n})$ for some $(i_1,\ldots,i_n)\in I$. In this latter case, the spectral radius is the one of the matrix 
		\[
		\begin{pmatrix}  
		i_{m+1} & \cdots & i_n & 0 \\
		1 & \cdots & 0 & 0  \\
		\vdots & \ddots & \vdots & \vdots  \\
		0 & \cdots & 1 & 0
		\end{pmatrix}
		\]
		and thus equal to the biggest real root of the polynomial $ x^{n-m}-\sum_{j=m+1}^{n} i_j x^{n-j}$. If $\deg_{x_{m+1},\ldots,x_n}(p)\le1$, the maximal eigenvalue is again equal to $1$,  and if $\deg_{x_{m+1},\ldots,x_n}(p)\ge 2$, we get 
		that $\theta$ is the maximal eigenvalue of $f$.
	\end{proof}
	
As  mentioned
in the introduction, the following result is due to Mattias Jonsson (unpublished).

\begin{proposition}\label{Prop:Jonsson}
	For each $n\ge 1$ and each polynomial $p\in \k[x_1,\ldots,x_{n}]$ of degree $\ge 2$, 
	let $e_p\in \Aut(\A^{n+1})$ be the automorphism
	\[e_p=(x_{n+1}+p(x_1,\ldots,x_{n}),x_1,\ldots,x_{n})\in \Aut(\A^{n+1}).\]
	Let $I\subset \NN^n$ be the  finite subset of indices of the monomials of $p$. We get
	 \[
	 	\lambda(e_p)=\max\left\{\lambda\in \r\left|
	 	\lambda^{n} = \sum_{j=1}^{n} i_j\lambda^{n-j} \ \text{ for some } (i_1,\ldots,i_n)\in I\right\}\right.
	 \] 
	\end{proposition}
\begin{proof}
Apply Proposition~\ref{Prop:PermutationElementary}\ref{Prop:PermutationElementary_general} 
with $m=0$ and $\xi=1$.
\end{proof}

\begin{proof}[Proof of Proposition~$\ref{Prop:PermElementary}$]

Let $h\in \End(\A^{n+1})$ be  a permutation-elementary automorphism.
By Lemma~\ref{Lem:Conjugation_Permutation_Elementary} there is a permutation of the coordinates $\alpha\in\Sym(\A^{n+1})$ such that 
\[
f=\alpha \circ h\circ \alpha^{-1}=(f_{1},\ldots,f_{m}, \xi x_{n+1}+p(x_1,\ldots,x_{n}),x_{m+1},\ldots,x_{n}),
\] 
where $0\le m\le n$,  $\{x_1,\ldots,x_m\}=\{f_1,\ldots,f_m\}$, $\xi\in \k^*$ and $p\in \k[x_1,\ldots,x_n]$.
In particular $\lambda(\hat{f}) = 1$ where $\hat{f} = (f_1, \ldots, f_m) \in \Sym(\AA^m)$.

As the maximal eigenvalue $\theta$ of $f$ is bigger than $1$, 
we have $m<n$ (see Remark~\ref{Remarkmntrivial}).
 Moreover, Proposition~\ref{Prop:PermutationElementary}\ref{Prop:PermutationElementary_special_permutation} 
yields
that $\deg_{x_{m+1,\ldots,x_n}}(p)\ge2$. Then, Proposition~\ref{Prop:PermutationElementary}\ref{Prop:PermutationElementary_general},\ref{Prop:PermutationElementary_special_permutation} 
give the existence of a maximal eigenvector $\mu$ such that $f$ is $\mu$-algebraically stable and prove that the dynamical degree $\lambda(f)$ is equal to the maximal eigenvalue $\theta$ of $f$ 
(this latter fact also follows from Proposition~\ref{Prop:MonomialEig}). Since
$\alpha \in \Sym(\AA^{n+1})$ we get that $\alpha^{-1}(\mu)$ is a maximal eigenvector
of $h = \alpha^{-1} \circ f \circ \alpha$, $h$ is $\alpha^{-1}(\mu)$-algebraically stable
and $\theta$ is the maximal eigenvalue of $h$.
Moreover,  $\lambda(h) = \lambda(f)$. Proposition~\ref{Prop:PermutationElementary}\ref{Prop:PermutationElementary_general} 
shows that $\theta$ is the root of a monic integral polynomial where all coefficients 
(except the first one) are non-positive, so it is a Handelman number by definition.
\end{proof}
\subsection{Affine-triangular automorphisms of $\A^3$}\label{Sec:AffineTriangularA3}
In this section, we apply  Proposition~\ref{Prop:MonomialEig}  to affine-triangular automorphisms $f\in \Aut(\A^3)$ and prove Proposition~\ref{Prop:AffineTriangularA3} and Theorem~\ref{Thm:AffTriang3}. By Proposition~\ref{Affine-Triangular-Permutation}, we can reduce to the case of permutation-triangular automorphisms. If the maximal eigenvalue $\theta$ of $f$ is equal to $1$, then Proposition~\ref{Prop:MonomialEig} gives $\lambda(f)=\theta$. If $\theta>1$, there is a maximal eigenvector 
$\mu=(\mu_1,\ldots,\mu_n)\in (\mathbb{R}_{\ge  0})^n \setminus \{0\}$ of $f$, and if $f$ is $\mu$-algebraically stable, we obtain $\lambda(f)=\theta$ (Proposition~\ref{Prop:MonomialEig}\ref{Prop:MonomialEig3}). We will then study the cases where $f$ is not $\mu$-algebraically stable. This implies that the $\mu$-leading part $g$ of $f$ is such that one component of $g^r$ is equal to zero for some $r\ge 1$. The possibilities for such endomorphisms $g$ are studied in Lemma~\ref{Lem:PossibilitiesForgrzero} below. 
The following result is a simple observation, whose proof is left as an exercise.

\begin{lemma}\label{Lemm:DominantTriangular}
Let $n\ge 1$ and let $f=(f_1,\ldots,f_n)\in \TEnd(\A^n)$ be a triangular endomorphism. Then, 
\begin{enumerate}[leftmargin=*]
	\item \label{Lemm:DominantTriangular_dom} $f$ is dominant if and only if $\deg_{x_i}(f_i)\ge 1$ for each $i\in \{1,\ldots,n\}$;
	\item \label{Lemm:DominantTriangular_auto} 
	$f$ is an automorphism
	if and only if $\deg_{x_i}(f_i) = 1$ for each $i \in \{1, \ldots, n\}$.
\end{enumerate}
\end{lemma}

\begin{lemma}\label{Lem:PossibilitiesForgrzero} Let $g=(g_1,g_2,g_3)=\sigma\circ \tau\in \End(\A^3)$ where $\tau\in \TEnd(\A^3)$ is a triangular endomorphism, $\sigma\in \Sym(\A^3)$ is a permutation of the coordinates, 
where all $g_i$ are non-constant and such that one  of the components of $g^r$ is a constant for some $r\ge 2$. Then, one of the following holds:
\begin{enumerate}[leftmargin=*]
\item\label{Iterzero1}
$g_2,g_3\in \k[x_1]$, $g_1\in \k[x_1,x_2,x_3]\setminus (\k[x_1,x_2]\cup \k[x_1,x_3])$ 
and there exists $\zeta\in \k$ such that  $g_1(t,g_2,g_3)=\zeta$ for each $t\in \k$;
\item\label{Iterzero2}
$g_1,g_3\in \k[x_1]$, $g_2\in \k[x_1,x_3]\setminus \k[x_1]$;
\item\label{Iterzero3}
$g_1,g_2\in \k[x_1]$, $g_3\in \k[x_1,x_2]\setminus \k[x_1]$;
\item\label{Iterzero4}
$g_1,g_2\in \k[x_1,x_2]\setminus \k[x_1]$, $g_3\in \k[x_1]$ and $g_1(g_1,g_2)=\zeta_1$,  $g_2(g_1,g_2)=\zeta_2$ for some $\zeta_1,\zeta_2\in \k$.
\end{enumerate}
\end{lemma}
\begin{proof}
We distinguish some cases, depending on which of the polynomials $g_1,g_2,g_3$ belong to $\k[x_1]$.

We first observe that $g_1,g_2,g_3\in \k[x_1]$ is impossible, as each component of $g^r$, for each 
$r\ge 1$, would then be obtained by composing dominant endomorphisms of $\A^1$ and 
thus would not be constant.

\begin{itemize}
 \item Suppose that $g_1,g_3\in \k[x_1]$. By induction, we obtain $(g^r)_1,(g^r)_3\in \k[x_1]\setminus \k$ for each $r\ge 1$, so $(g^r)_2\in \k$ for some $r\ge 2$. If $g_2\in \k[x_1,x_3]$, we obtain \ref{Iterzero2}. Otherwise, $\deg_{x_2}(g_2)=d\ge 1$ and proceeding by induction we obtain $\deg_{x_2}((g^r)_2)=d^r\ge 1$ for each $r\ge 1$, impossible.

\item If $g_1,g_2\in \k[x_1]$ we do the same argument as before (by exchanging the roles of $x_2$ and $x_3$) and obtain \ref{Iterzero3}.

\item Suppose now that $g_2,g_3\in \k[x_1]$.
As $g_1\in \k[x_1,x_2,x_3]\setminus \k[x_1]$,  the closure of the image of $g\in \End(\A^3)$ is then equal to $\A^1\times \Gamma$, where $\Gamma\subset \A^2$ is the irreducible  curve 
that is the closure of the image of $\A^1\to \A^2$, $x_1\mapsto (g_2(x_1),g_3(x_1))$. The restriction of $g$ gives an endomorphism $h=g|_{\A^1\times \Gamma}\in \End(\A^1\times \Gamma)$. 

We now prove that $h$ is not dominant. 
For each $r\ge 1$ and each $i\in \{1,2,3\}$, the restriction of $(g^r)_{i}$ to $\A^1\times \Gamma$ is equal to $\pi_i\circ h^r$, where $\pi_i\colon \A^1\times \Gamma\to \A^1$ is the $i$-th projection. 
Choosing $i$ and $r$ such that $(g^r)_i$ is constant, we find that $\pi_i\circ h^r$ is constant, so $h^r$ is not dominant, as $\pi_i$ is dominant. This proves that $h$ is not dominant.  

Denote by $\Gamma'\subset \A^1\times \Gamma$ the closure of $h(\A^1\times \Gamma)$, which is an irreducible curve, that contains $\set{(g_1(x,g_2(y),g_3(y)),g_2(x),g_3(x)) }{(x,y)\in \A^2}$. 
This implies that the polynomial $s=g_1(x,g_2(y),g_3(y))\in \k[x,y]$ is contained in $\k[x]$. 
We moreover observe that $s$ is a constant. Indeed, otherwise the restriction of $h$ to $\Gamma'$ would be a dominant map $\Gamma'\to \Gamma'$ and since $\pi_i |_{\Gamma'} \colon \Gamma' \to \AA^1$ 
is non-constant for each $i \in \{1, 2, 3\}$, the restriction of $(g^r)_i$ to $\Gamma'$ 
would be non-constant for each $r\ge 1$ and each $i\in\{1,2,3\}$, contradiction. 
Hence, $\pi_1\circ h=g_1|_{\A^1\times \Gamma}\colon \A^1\times \Gamma\to \A^1$ is equal to a constant $\zeta\in \k$. 
This yields$g_1(t,g_2,g_3)=\zeta$ for each $t\in \k$ and implies that
$g_1\not\in \k[x_1,x_2]\cup \k[x_1,x_3]$, since $g_1, g_2, g_3$ are non-constant,
whence~\ref{Iterzero1}.

\item It remains to assume that at most one of the $g_i$ belongs to $\k[x_1]$. We write $\tau=(\tau_1,\tau_2,\tau_3)$, and observe that $\{g_1,g_2,g_3\}=\{\tau_1,\tau_2,\tau_3\}$.
As $\tau_1\in \k[x_1] \setminus \k$,
we get that exactly one of the $g_i$ belongs to $\k[x_1]$ and that $\tau_2\in \k[x_1,x_2]\setminus \k[x_1]$. As $g$ is not dominant, neither is $\tau$;  Lemma~\ref{Lemm:DominantTriangular} then implies that $\tau_3\in \k[x_1,x_2]$. 
So $g_1,g_2,g_3\in \k[x_1,x_2]$ and exactly one of the three belongs to $\k[x_1]$. Note that the endomorphism $h=(g_1,g_2)\in \End(\A^2)$ is not dominant. Indeed, otherwise no component of $g^r$ is constant for each $r\ge 1$, as  $g_3\in \k[x_1,x_2]$ is non-constant. 
It is thus impossible that $g_1\in \k[x_1]$ or $g_2\in \k[x_1]$, as $(g_1,g_2)$ (respectively $(g_2,g_1)$) would be a dominant triangular endomorphism of $\A^2$ (Lemma~\ref{Lemm:DominantTriangular}). 
Hence, $g_3\in \k[x_1] \setminus \k$ and $g_1,g_2\in \k[x_1,x_2]\setminus \k[x_1]$. 
As $h$ is not dominant, the closure of $h(\A^2)$ is an  irreducible curve $\Gamma\subset \A^2$. 

If $g_j(\Gamma)$ is not a point for $j=1$ or $j=2$, then the  restriction 
$h |_\Gamma \colon \Gamma \to \Gamma$
would be dominant.  As $g_3$ is not constant on $\Gamma$ 
(because $g_3(g_1(x_1,x_2))$ is not constant), we get that $(g^r)_i$ is non-constant for each $r \geq 1$
and each $i \in \{1, 2, 3\}$, contradiction. Thus $g_i(\Gamma) = \{\zeta_i\}$ for $i=1,2$
where $\zeta_i \in \k$. This gives \ref{Iterzero4}. \qedhere
\end{itemize}
\end{proof}

\begin{lemma}\label{Lem:PermTriangularNotAlgStab}
Let $f=\sigma\circ \nu\in \Aut(\A^3)$ be a permutation-triangular  automorphism, where $\sigma\in \Sym(\A^3)$ and $\nu\in\TAut(\A^3)$. Suppose that the maximal eigenvalue $\theta$ of $f$ is bigger than $1$ and let $\mu$ be  a maximal eigenvector of $f$ such that $f$ is not $\mu$-algebraically stable. Then, one of the following cases holds:
\begin{enumerate}[$(i)$, leftmargin=*]
\item\label{PerNotAlgStab1} 
$f=(\xi_3 x_3+p_3(x_1,x_2),p_1(x_1),\xi_2 x_2+p_2(x_1))$ where 
$\xi_2,\xi_3\in \k^*$, $p_1,p_2\in \k[x_1]$, $p_3\in \k[x_1,x_2]$, $\deg(p_1)=1$, 
and $\deg(p_2)=\theta^2>1$. Moreover, there exists $s\in \k[x_2]$ such that the conjugation of 
$f$ by $(x_1,x_2,x_3+s(x_2))$ does not increase the degree of $p_3$ and $($strictly$)$ decreases  
the degree of $p_2$.

\item\label{PerNotAlgStab2}
$f=(\xi_2 x_2+p_2(x_1),\xi_3 x_3+p_3(x_1,x_2),p_1(x_1))$ where $\xi_2,\xi_3\in \k^*$, 
$p_1,p_2\in \k[x_1]$, $p_3\in \k[x_1,x_2]$, $\deg(p_1)=1$, and $\deg(p_2)=\theta>1$. Moreover, 
there exists $s\in \k[x_1]$ such that the conjugation of $f$ by $(x_1,x_2+s(x_1),x_3)$ 
$($strictly$)$ decreases the degrees of $p_2$ and $p_3$.
\end{enumerate}

\end{lemma}
\begin{proof}
Denote by $g=(g_1,g_2,g_3)$ the $\mu$-leading part of $f$. As 
$\mu=(\mu_1,\mu_2,\mu_3)\in (\r_{\ge 0})^3\setminus \{0\}$ is a maximal eigenvector of $f$,  
$g_i\not\in \k$ for each $i\in \{1,2,3\}$ (Lemma~\ref{Lem:MaximalEigenVectorLeadingpartNotZero}).
Moreover, as $f$ is not $\mu$-algebraically stable, there is some $r\ge 1$ 
such that $(g^r)_i=0$ 
for all $i\in \{1,2,3\}$ with $\mu_i>0$ (Lemma~\ref{Lemma:Dynamicaldegreeestimate}\ref{degmudyn2}). 
We write $g=\sigma \circ \tau$ where $\tau=(\tau_1,\tau_2,\tau_3)\in\TEnd(\A^3)$; Lemma~\ref{Lem:PossibilitiesForgrzero} gives then four possibilities \ref{Iterzero1}-\ref{Iterzero2}-\ref{Iterzero3}-\ref{Iterzero4} for $g$, that we consider separately. 
We will show that \ref{PerNotAlgStab1} and \ref{PerNotAlgStab2} occur
in Cases~\ref{Iterzero1} and \ref{Iterzero4}, respectively 
and that \ref{Iterzero2}-\ref{Iterzero3} do not occur.

\ref{Iterzero2}-\ref{Iterzero3}: Let us first observe that Case \ref{Iterzero2} (respectively \ref{Iterzero3}) of Lemma~\ref{Lem:PossibilitiesForgrzero} does not occur. Indeed, otherwise the
first and the last (respectively the first two) components of $g^r$ belong to $\k[x_1]\setminus \k$ for each $r\ge 1$, so $\mu=(0,\mu_2,0)$ (respectively $\mu=(0,0,\mu_3)$), since $(g^r)_i = 0$
for each $i \in \{1, 2, 3\}$ with $\mu_i > 0$.  This gives $\deg_\mu(g_i)=0$ for $i=1,2,3$, as $g_1,g_2,g_3$ belong to $\k[x_1,x_3]$ (respectively $\k[x_1,x_2]$), impossible as $\deg_\mu(g)=\deg_\mu(f)=\theta>1$ (Lemma~\ref{Lem:MaximalEigenVectorLeadingpartNotZero}).

\ref{Iterzero1}: Suppose now that Case \ref{Iterzero1} of Lemma~\ref{Lem:PossibilitiesForgrzero} occurs:
As $g_1\in \k[x_1,x_2,x_3]\setminus (\k[x_1,x_2]\cup \k[x_1,x_3])$ and since the monomials of $g_1$ are some of those of $f_1$, that is one of the coordinates of the triangular automorphism $\nu\in \TAut(\A^3)$, the polynomial $f_1$ is equal to the third coordinate of $\nu$ 
and $g_1$ is of the form $g_1=\xi_3 x_3+q(x_1,x_2)$ for some $\xi_3\in \k^*$ and $q\in \k[x_1,x_2]\setminus \k[x_1]$.  Since $g_1(t,g_2,g_3)=\zeta \in \k$ for each $t\in \k$, we obtain $\xi_3 g_3+q(t,g_2)=\zeta$
for each $t\in \k$, so $q\in \k[x_2]\setminus \k$ and
\[g=(\xi_3 x_3+q(x_2),g_2,(\zeta-q(g_2))/\xi_3),\]
where $g_2\in \k[x_1]$.
 By definition (Definition~\ref{Def:LeadingPart}), $g_i$ is the $\mu$-homogeneous part of $f_i$ of degree $\theta\mu_i$, for each $i\in \{1,2,3\}$ so each monomial of $g_i$ is of $\mu$-degree $\theta\mu_i$. The explicit form of $g_1,g_2,g_3$ directly gives
\[
 	\theta\mu_1=\mu_3 =\deg(q)\mu_2\, , \ \theta\mu_2=\deg(g_2)\mu_1 
 	\ \text{ and } \ \theta\mu_3=\deg(g_3)\mu_1=\deg(q) \deg(g_2)\mu_1 \, .
\] 
In particular, $\mu_1,\mu_2,\mu_3\in \r_{>0}$ and $\deg(g_3)=\deg(q)\deg(g_2)=
\theta^2>1$. Since two monomials in the same variables have distinct $\mu$-degrees, we moreover find that $q$, $g_2$ and $g_3$ are monomials, so $\zeta=0$.

One component of $f$ (and of $\tau$) belongs to $\k[x_1]$ and is of degree $1$. As $g_1\not\in\k[x_1]$ and $\deg(g_3)>1$, we find that $f_2\in \k[x_1]$ is of degree $1$. This yields $\sigma=(x_3,x_1,x_2)$ and $\deg(f_2)=\deg(g_2)=1$, whence $\deg(q)=\deg(g_3)=\theta^2>1$. We obtain the form given in \ref{PerNotAlgStab1}: the automorphism $f$ is 
equal to \[f=(\xi_3 x_3+p_3(x_1,x_2),p_1(x_1),\xi_2 x_2+p_2(x_1))\] where 
$\xi_2,\xi_3\in \k^*$, $p_1,p_2\in \k[x_1]$, $p_3\in \k[x_1,x_2]$, $\deg(p_1)=1$. Moreover, $g_3=-q(g_2)/\xi_3\in \k[x_1]$ is the $\mu$-leading part of $f_3=\xi_2 x_2+p_2(x_1)$, so $g_3$ is only one monomial, of degree 
$\theta^2 = \deg(g_3)=\deg(p_2)$.

To prove that we are indeed in Case~\ref{PerNotAlgStab1}, it remains to show that the conjugation by  $h=(x_1,x_2,x_3+\xi_3^{-1}q(x_2))$ does not increase the degree $p_3$ and strictly decreases the degree of $p_2$. We calculate \[h\circ f \circ h^{-1}=(\xi_3 x_3+p_3(x_1,x_2)-q(x_2),p_1(x_1),\xi_2 x_2+p_2(x_1)+q(p_1(x_1))/\xi_3).\]
 As every monomial of $g_1=\xi_3 x_3+q(x_2)$ is contained in $f_1=\xi_3 x_3+p_3(x_1,x_2)$, the degree of $p_3(x_1,x_2)-q(x_2)$ is at most the one of $p_3(x_1,x_2)$. It remains to see that $\deg(p_2+q(p_1)/\xi_3)<\deg(p_2)$ which follows from the fact that  $g_3=-q(g_2)/\xi_3\in \k[x_1]$ is the $\mu$-leading part of $f_3=\xi_2 x_2+p_2(x_1)$, and that $g_2$ is the leading monomial of $p_1$ (of degree $1$.

  \ref{Iterzero4}: It remains to consider Case \ref{Iterzero4} of Lemma~\ref{Lem:PossibilitiesForgrzero}. 
  As $g_1,g_2\in \k[x_1,x_2]\setminus \k[x_1]$, the 
  only component of $f$ which belongs to $\k[x_1]$ (and is of
  degree $1$) is $f_3$, so $\sigma=(x_3,x_2,x_1)$ or $\sigma=(x_2,x_3,x_1)$. Let $j\in \{1,2\}$ be such that $f_j= \nu_2$, where $\nu = (\nu_1, \nu_2, \nu_3)$. 
  We then have $f_j=\xi_2 x_2 +p_2(x_1)$ for some $\xi_2\in \k^*$ and some $p_2\in \k[x_1]$.
   As 
  $g_j\in \k[x_1,x_2]\setminus \k[x_1]$, we get $g_j=\xi_2 x_2 +q(x_1)$ for some $q\in \k[x_1]$, 
that consists of some monomials of $p_2$. 
Since $\zeta_j=g_j(g_1,g_2)$, we obtain $\zeta_j=\xi_2 g_2+q(g_1)$.

We now show that $j=2$
  leads to a contradiction. It gives
  \[g_2=\xi_2 x_2 +q(x_1)=\xi_2^{-1}(\zeta_2-q(g_1)).\]
 Since $\xi_2 x_2 +q(x_1)$ is irreducible, the polynomial $\zeta_2-q(g_1)$ is irreducible, and thus $\deg(q)=1$, which in turn implies that  $g_2$ and thus $g_1$ is of degree $1$. Hence, $g_1,g_2,g_3$ are of degree $1$, 
 impossible, as  $\theta>1$ is the eigenvalue of a 
  matrix that is contained in $g$ (Lemma~\ref{Lemm:DegreeEndhom}).  
  
  This contradiction proves that $j=1$, so $\sigma=(x_2,x_3,x_1)$. This yields
  \[f=(\xi_2 x_2+p_2(x_1),\xi_3 x_3+p_3(x_1,x_2),p_1(x_1))\] where $\xi_2,\xi_3\in \k^*$, 
$p_1,p_2\in \k[x_1]$, $p_3\in \k[x_1,x_2]$ and $\deg(p_1)=1$, as in \ref{PerNotAlgStab2}.

We also have $g_1=\xi_2 x_2+q(x_1)$ and $\zeta_1=\xi_2 g_2+q(g_1)$, which yields $g_2=(\zeta_1-q(g_1))/\xi_2=(\zeta_1-q(\xi_2 x_2 +q(x_1)))/\xi_2$.
As $g$ is the  
  $\mu$-leading part of $f$, the polynomial $g_2$ is not constant (Lemma~\ref{Lem:MaximalEigenVectorLeadingpartNotZero}), so $\deg(q)\ge 1$.  
Recall that $g_i$ is the $\mu$-homogeneous part of $f_i$ of degree $\theta\mu_i$, 
  for each $i\in \{1,2,3\}$ (Definition~\ref{Def:LeadingPart}) 
  so each monomial of $g_i$ is of $\mu$-degree $\theta\mu_i$.  We thus obtain
  \[
  		\theta\mu_1=\mu_2 =\deg(q)\mu_1 
  		 \quad \text{ and } \quad \theta\mu_3=\mu_1 \, .
  \]
 This proves that $\mu_1,\mu_2,\mu_3\in \r_{>0}$, that $\deg(q)=\theta>1$ and that $\mu=(\theta\mu_3,\theta^2\mu_3,\mu_3)$. 

Since two monomials in the same variables have distinct $\mu$-degrees, we moreover find that $q$ is a monomial, the leading monomial of $p_2$, so $\deg(p_2)=\deg(q)=\theta>1$, as stated in \ref{PerNotAlgStab2}.

To prove that we are indeed in Case~\ref{PerNotAlgStab2}, it remains to show that the conjugation by  $h=(x_1,x_2+q(x_1)/\xi_2,x_3)$ 
strictly decreases the degree of $p_2$ and $p_3$. 
We calculate \[h\circ f \circ h^{-1}=(\xi_2 x_2+p_2'(x_1),\xi_3 x_3+p_3'(x_1,x_2),p_1(x_1)),\] where
  \[
  	\begin{array}{rcl}
  		p_2'(x_1)&=&p_2(x_1)-q(x_1),\\
  		p_3'(x_1,x_2)&=&p_3(x_1,x_2-q(x_1)/\xi_2)+q(\xi_2 x_2+p_2'(x_1))/\xi_2 \, .
  	\end{array}
  \]
  As $q$ is the leading monomial of $p_2$, this conjugation decreases the degree of $p_2$, i.e.~$\deg (p_2')<\deg (p_2)=\theta$. 
  It remains to see that $\deg (p_3')< \deg (p_3)$. To simplify the calculations, 
  we replace $\mu$ by a multiple of itself (this is still a maximal eigenvector) and 
  may assume that $\mu=(1,\theta,\theta^{-1})$. 
  As $g_2=(\zeta_1-q(\xi_2 x_2 +q(x_1)))/\xi_2$ is the $\mu$-homogeneous part of $f_2=\xi_3 x_3+p_3(x_1,x_2)$ of $\mu$-degree $\theta\mu_2=\theta^2$, the polynomial $\Delta=p_3-g_2\in \k[x_1,x_2]$ is equal to 
  \[
  	\Delta=\sum\limits_{i=0}^{\theta-1} x_2^i\Delta_i
  \]
  where each $\Delta_i\in \k[x_1]$ is such that $\deg(\Delta_i)+i\theta< \theta^2$. 
  As $\theta>1$, this implies that $\deg(x_2^i\Delta_i)=i+\deg(\Delta_i)<\theta^2$ for each $i$, so $\deg(\Delta)<\theta^2$, which implies that
  the degree of $p_3=\Delta+g_2$ is equal to $\theta^2$, 
  since $\deg(g_2)=\theta^2$. 
  We then need to show that 
  $\deg(p_3')< \theta^2$. 
  Since $\deg(p_2')<\deg(q)=\theta$, we have $\deg(q(\xi_2 x_2+p_2'(x_1))/\xi_2)<\theta^2$,
  so we only need to 
  show that $\deg(p_3(x_1,x_2-q(x_1)/\xi_2))<\theta^2$. This is given by
  \[
  	\begin{array}{rcl}
		p_3(x_1,x_2-q(x_1)/\xi_2)
		&=& \Delta(x_1,x_2-q(x_1)/\xi_2)+g_2(x_1,x_2-q(x_1)/\xi_2).\\
  		&=& \sum\limits_{i=0}^{\theta-1} (x_2-q(x_1)/\xi_2)^i\Delta_i+(\zeta_1-q(\xi_2 x_2))/\xi_2
  	\end{array}
  \]
  and by the fact that $\deg(\Delta_i)+i\theta< \theta^2$ for each $i$.
\end{proof}
\begin{example}\label{Exa:PermTriangularNotAlgStab}
	We now give two distinct examples to show that Cases~\ref{PerNotAlgStab1}-\ref{PerNotAlgStab2} of Lemma~\ref{Lem:PermTriangularNotAlgStab} indeed occur.
\begin{enumerate}[leftmargin=*]
\item[\ref{PerNotAlgStab1}]
Let $n\ge 2$, and let $f=(x_3-x_2^n,x_1,x_2+x_1^n)\in \Aut(\A^3)$. Because of the matrix contained in $(x_3,x_1,x_1^n)$, the maximal eigenvalue satisfies $\theta\ge \sqrt{n}>1$ and as $f^2=(x_2,x_3-x_2^n,x_1+(x_3-x_2^n)^n)$ and $f^3=(x_1,x_2,x_3)$, the  map
$f$ is not $\mu$-algebraically stable for any maximal eigenvector $\mu$ of $f$. 
It has then to satisfy Case~\ref{PerNotAlgStab1} of Lemma~\ref{Lem:PermTriangularNotAlgStab}, so $\theta=\sqrt{n}$.
\item[\ref{PerNotAlgStab2}]
Let $n\ge 2$, and let $f=(x_2-x_1^n,x_3+(x_2-x_1^n)^n,x_1)\in \Aut(\A^3)$. 
Because of the matrix contained in $(-x_1^n,x_3,x_1)$, the 
maximal eigenvalue satisfies $\theta\ge n>1$ and as $f^2=(x_3,x_1+x_3^n,x_2-x_1^n)$ and $f^3=(x_1,x_2,x_3)$, the element $f$ is not $\mu$-algebraically stable
for each maximal eigenvector $\mu$ of $f$. 
It has then to  satisfy Case~\ref{PerNotAlgStab2} of Lemma~\ref{Lem:PermTriangularNotAlgStab}, so $\theta=n$.
\end{enumerate}

\end{example}
We now give examples of permutation-triangular automorphisms of $\A^3$ which are $\mu$-algebraically stable. These will be useful in the proof of Theorem~\ref{Thm:AffTriang3}.

\begin{lemma}\label{Lemm:TwoBasicExamplesThatwork}
For all $a,b,c\in \NN$ such that $\lambda=\frac{a+\sqrt{a^2+4bc}}{2}\not=0$, the maximal eigenvalue and the dynamical degree of the automorphisms
\[
	f=(x_1^ax_2^b+x_3,x_2+x_1^c,x_1) \text{ and }f'=(x_3+x_1^ax_2^{bc},x_1,x_2)
\] 
are equal to $\lambda$. Both automorphisms 
are $\mu$-algebraic stable for each maximal eigenvector $\mu$.
\end{lemma}

\begin{proof}
The matrices that are contained in $f$ are 
\[
	\begin{pmatrix}
		a &  b& 0\\
		c & 0& 0\\
		1 & 0 &0
	\end{pmatrix}, 
	\begin{pmatrix}
		a &  b& 0\\
		0 & 1& 0\\
		1 & 0 &0
	\end{pmatrix}, 
	\begin{pmatrix}
		0 & 0& 1\\
		c & 0& 0\\
		1 & 0 &0
	\end{pmatrix}
	\text{ and } 
	\begin{pmatrix}
		0 & 0& 1\\
		0 & 1& 0\\
		1 & 0 &0
	\end{pmatrix}
\]
whose characteristic polynomials are $x(x^2-ax-bc)$, $x(x-a)(x-1)$, $x(x^2-1)$ and $(x+1)(x^2-1)$, respectively. The corresponding spectral radii are respectively $\lambda$, $a$, $1$ and~$1$. Hence, the maximal eigenvalue of $f$ is $\lambda$.

Similarly, the matrices contained in $f'$are 
\[
	\begin{pmatrix}
		a &  bc& 0\\
		1 & 0& 0\\
		0 & 1 &0
	\end{pmatrix}
\text{ and } 
	\begin{pmatrix}
		0 & 0& 1\\
		1 & 0& 0\\
		0 & 1 &0
	\end{pmatrix}
\]
whose characteristic polynomials  are $x(x^2-ax-bc)$ and $x^3-1$. The maximal eigenvalue of $f'$ is then also $\lambda$.

As neither $f$ nor $f'$ satisfies any of the two Cases \ref{PerNotAlgStab1}-\ref{PerNotAlgStab2} of Lemma~\ref{Lem:PermTriangularNotAlgStab}, both $f$ and $f'$ are $\mu$-algebraically stable for each maximal eigenvector $\mu$ (of $f$ and $f'$, respectively). This gives then $\lambda(f)=\lambda(f')=\lambda$ (Proposition~\ref{Prop:MonomialEig}) and achieves the proof.
\end{proof}
\begin{lemma}\label{Lemm:SetOfEigenvaluesAffT3}
The maximal eigenvalue $\theta$ of a permutation-triangular automorphism $f\in \Aut(\A^3)$ of degree $d\ge 1$ is a non-zero number equal to $(a+\sqrt{a^2+4bc})/2$ for some $(a,b,c)\in \NN^3$ where 
$a+b\le d$ and $c\le d$. 
It is thus a positive integer or a quadratic integer and a Handelman number.
\end{lemma}
\begin{proof}
Each real number $\theta=\frac{a+\sqrt{a^2+4bc}}{2}\not=0$, where $(a,b,c) \in \NN^3$ is a root of the polynomial $P(x)=x^2-ax-bc$, with $a,b,c\in \NN^2\setminus \{0\}$ so it is a Handelman number. If $P$ is irreducible, then $\theta$ is a quadratic integer, and otherwise it is a positive integer. 
It remains to see that  the maximal eigenvalue of every $f$ is of the desired form.

We write $f=\sigma\circ \tau$, where $\sigma\in \Sym(\A^3)$ and $\tau\in \TAut(\A^3)$ is a triangular automorphism, that we write as $\tau=(\nu_1x_1+\epsilon,\nu_2 x_2+p(x_1),\nu_3 x_3+q(x_1,x_2))$ where  $\nu_1,\nu_2,\nu_3\in \k^*$, $\epsilon\in \k$, $p\in \k[x_1]$ and $q\in \k[x_1,x_2]$. The matrices contained in $\tau$ are all of the form
\[\begin{pmatrix}
m & 0 &0\\
0 & 1 & 0\\
0 & 0& 1
\end{pmatrix},
\begin{pmatrix}
m & 0 &0\\
k & 0 & 0\\
0 & 0& 1
\end{pmatrix},\begin{pmatrix}
m & 0 &0\\
0 & 1 & 0\\
i & j& 0
\end{pmatrix},
\begin{pmatrix}
m & 0 &0\\
k & 0 & 0\\
i & j& 0
\end{pmatrix}\]
where $m, k, i, j$ are non-negative integers and $0 \leq m \leq 1$, 
$k\le \deg(p)\le d$ and $i+j\le \deg (q)\le d$. 
Since the spectral radius is order-preserving
on real square matrices with non-negative coefficients (see Definition~\ref{Defi:NormRn}\ref{Defi:NormRn:Spectralradius}) and since $\nu_1 \neq 0$, 
the maximal eigenvalue is the spectral radius of a matrix where $m=1$.
The matrices contained in $f$ are obtained from one of the above four types by permuting the rows. Permuting the rows of the identity matrix only gives a spectral radius equal to $1$. In the second case, we conjugate by the permutation of the last two. In any case, we obtain that $\theta$ is either equal to $1$ or is the spectral radius of a matrix $\sigma' M$, where $\sigma'$ is a permutation matrix and M is of the form
\[
\begin{pmatrix}
1 & 0 &0\\
0 & 1& 0\\
k & 0 & 0
\end{pmatrix},\begin{pmatrix}
1 & 0 &0\\
0 & 1 & 0\\
i & j& 0
\end{pmatrix},
\begin{pmatrix}
1 & 0 &0\\
k & 0 & 0\\
i & j& 0
\end{pmatrix}\]
where $k\le d$ and $i+j\le d$. We obtain
\[\sigma'M=\begin{pmatrix}
m_{11} & m_{12} &0\\
m_{21} & m_{22} & 0\\
m_{31} & m_{32}& 0
\end{pmatrix}\]
for some  $m_{ij}\in \NN$, so $\theta$ is the spectral radius of the matrix 
\[
	\begin{pmatrix}
		m_{11} & m_{12}\\
		m_{21} & m_{22}
	\end{pmatrix} \, .
\]
This last matrix is one of the following:
\[
	\begin{pmatrix}
		r & 0\\
		s & 0
	\end{pmatrix}, 
	\begin{pmatrix}
		r & 0\\
		0 & s
	\end{pmatrix},  
	\begin{pmatrix}
		r & 0\\
		i & j
	\end{pmatrix}, 
	\begin{pmatrix}
		i & j\\
		0 & 1
		\end{pmatrix}
	\begin{pmatrix}
		i & j\\
		r & 0
	\end{pmatrix},
	\begin{pmatrix}
	0 & 1 \\
	i & j
	\end{pmatrix},
	\text{ or }
	\begin{pmatrix}
		0 & r\\
		s & 0
	\end{pmatrix},
\]
where $r,s\in \{1,i,j,k\}$. In the first four cases, $\theta$ is an integer in $\{1,\ldots,d\}$, so has the desired form, with $a=\theta$, and $b=c=0$. In the fifth case,  the characteristic polynomial is  $x^2-ix-jr$. 
Choosing $a=i$, $b=j$ and $c=r$ we get $\theta=(a+\sqrt{a^2+4bc})/2$. 
In the sixth case, the characteristic polynomial is
$x^2 - jx - i$. When we choose $a = j$, $b = i$ and $c = 1$, we get again 
$\theta = (a+\sqrt{a^2+4bc})/2$.
In the last case, the characteristic polynomial is $x^2-rs$. We then choose $a=0$, $b=r$ and $c=s$.
\end{proof}

We can now give the proof of Proposition~\ref{Prop:AffineTriangularA3}.
\begin{proof}[Proof of Proposition~\ref{Prop:AffineTriangularA3}]
We take an affine-triangular automorphism  $f\in \Aut(\A^3)$. By Proposition~\ref{Affine-Triangular-Permutation}, there exists $\alpha\in \Aff(\A^3)$ such that $f'=\alpha f \alpha^{-1}$ is a permuta\-tion-triangular automorphism. We then have $\deg(f')=\deg(f)$. Moreover, Proposition~\ref{Prop:MonomialEig} shows that there exists a maximal eigenvector of $f$. We denote by $\theta$ the maximal eigenvalue of $f'$. If $\theta=1$ or if $f'$ is $\mu$-algebraically stable for each maximal eigenvector $\mu$, the dynamical degrees $\lambda(f)$ and $\lambda(f')$ are equal to the maximal eigenvalue $\theta$ of $f'$ (Proposition~\ref{Prop:MonomialEig}), which is a Handelman number (Lemma~\ref{Lemm:SetOfEigenvaluesAffT3}) so the result holds.

Suppose now that $\theta>1$ and that $f'$ is not $\mu$-algebraically stable for some maximal eigenvector $\mu$. Lemma~\ref{Lem:PermTriangularNotAlgStab} gives two possibilities for $f'$:
\begin{align*}
	f' &=(\xi_3 x_3+p_3(x_1,x_2),p_1(x_1),\xi_2 x_2+p_2(x_1)) \quad \textrm{or} \\
	f' &=(\xi_2 x_2+p_2(x_1),\xi_3 x_3+p_3(x_1,x_2),p_1(x_1))
\end{align*}
where $p_1, p_2 \in \k[x_1]$, $p_3 \in \k[x_1, x_2]$, 
$\xi_2, \xi_3 \in \k^\ast$, $\deg(p_1)=1$ and $\deg(p_2)>1$. In both cases, Lemma~\ref{Lem:PermTriangularNotAlgStab} shows that one can replace $f'$  by a conjugate, decrease the degree of $p_2$ and do not increase the degree of $f'$. After finitely many steps, we obtain the desired case where $\theta=1$ or $f'$ is $\mu$-algebraically stable for each maximal eigenvector $\mu$. Moreover, we still have $\deg(f')\le \deg(f)$.
\end{proof}

\begin{proof}[Proof of Theorem~$\ref{Thm:AffTriang3}$]
Let $f\in \Aut(\A^3)$ is an affine-triangular automorphism of $\A^3$ of degree $d$. Proposition~\ref{Prop:AffineTriangularA3} gives the existence of a permutation-triangular automorphism $f'\in \Aut(\A^3)$ such that $\deg(f')\le \deg(f)$ and such that either the maximal eigenvalue $\theta$ of $f'$ is equal to $1$, or $\theta>1$ and $f'$ is $\mu$-algebraically stable for each maximal eigenvector 
$\mu$. In the first case, the dynamical degree $\lambda(f)$ is equal to $\lambda(f')=1$, by Proposition~\ref{Prop:MonomialEig}\ref{Prop:MonomialEig2}. 
In the second case, we obtain $\lambda(f)=\lambda(f')=\theta$, by Proposition~\ref{Prop:MonomialEig}\ref{Prop:MonomialEig3}. 
Moreover, Lemma~\ref{Lemm:SetOfEigenvaluesAffT3} proves that $\theta=\frac{a+\sqrt{a^2+4bc}}{2}$ for some $a,b,c\in \NN$ with $a+b\le d, c\le d$ (and that $\theta\not=0$).

Conversely, for all $a,b,c\in \NN$ such that $\theta=\frac{a+\sqrt{a^2+4bc}}{2}\not=0$, the element $\theta$ is the dynamical degree of $(x_1^ax_2^b+x_3,x_2+x_1^c,x_1)$ and $(x_3+x_1^ax_2^{bc},x_1,x_2)$ (Lemma~\ref{Lemm:TwoBasicExamplesThatwork}), and thus of a permutation-triangular automorphism of $\A^3$. This achieves the proof.
\end{proof}

\begin{corollary}
	\label{Cor:ToThm1}
	For each $d \ge 3$ the set of all dynamical degrees of shift-like automorphisms of $\A^3$ of degree $d$
	is strictly contained in the set of all dynamical degrees of affine-triangular 
	automorphisms of degree $d$.
\end{corollary}

\begin{proof}
	As each shift-like automorphism is also an affine-triangular automorphism, we have an inclusion, that we need to prove to be strict.
	From~ Proposition~\ref{Prop:Jonsson} it follows that the set of
	dynamical degrees of all shift-like automorphisms of $\A^3$ of degree $d$ is equal
	to 
	\[
		\Bigset{({a + \sqrt{a^2 + 4d-4a}})/{2}}{0 \leq a \leq d} \, .
	\]
	From Theorem~\ref{Thm:AffTriang3} it follows that $\lambda_d=(1 + \sqrt{1 + 4d})/2$ 
	is the dynamical degree of the affine-triangular automorphism $(x_3 + x_1 x_2,x_2+x_1^d,x_1)$.
	In order to show that $\lambda_d$ is not the dynamical degree of any shift-like automorphism
	of $\AA^3$ of degree $d$, for each $d \geq 3$,
	we only have to show that there exists no $d \geq 3$ and no $a \in \{0, \ldots, d \}$ such that
	\[
		\sqrt{1 + 4d} = (a-1) + \sqrt{a^2 + 4d-4a} \, .
	\]
	Indeed, if this would be the case, then $1 + 4d = (a-1)^2 + 2(a-1)\sqrt{a^2 + 4d-4a} + a^2 + 4d-4a$,
	which yields
	\[
		a(3-a) = (a-1)\sqrt{a^2 + 4d-4a} \, .
	\]
	This implies that $a \leq 3$ and $a \not\in \{0, 1\}$, i.e.~$a = 2$. However, in this case $d = 2$.
	
\end{proof}

\subsection{Automorphisms of affine spaces associated to weak-Perron numbers}\label{SubSec:WeakPerron}
In this section, we construct some affine-triangular automorphisms associated to weak-Perron numbers and prove Theorem~\ref{Theorem:PerronAuto}.

\begin{lemma}\label{Lemm:Monomialauto}
Let $n\ge 1$ and let $A=(a_{i,j})_{i,j=1}^n\in \Mat_n(\NN)$ be an irreducible matrix with 
spectral radius $\rho(A)>1$. The automorphism  $f \in \Aut(\A^{2n})$ given by 
\[
	\label{Lemm:Monomialauto_1}
	\tag{$\ast$}
	\left(x_{n+1}+ \prod_{i=1}^n x_i^{a_{1,i}},
		x_{n+2}+ \prod_{i=1}^n x_i^{a_{2,i}},\ldots,x_{2n}+ \prod_{i=1}^n x_i^{a_{n,i}},
		x_1,\ldots,x_n\right)
\]
has dynamical degree $\lambda(f)=\rho(A)$.
\end{lemma}

\begin{proof}
Let us write $\theta=\rho(A)$ and choose an eigenvector $v=(v_1,\ldots,v_n)\in (\r_{>0})^n$ of $A$ to the eigenvalue $\theta$ (which exists by Theorem~\ref{Thm:PerronFrobenius}). We then choose $\mu=(\theta v_1,\ldots,\theta v_n,v_1,\ldots,v_n)\in (\r_{>0})^{2n}$. The matrix 
\[
	M=
	\begin{pmatrix} 
		A & 0 \\ 
		I_n & 0 
	\end{pmatrix}\in \Mat_{2n}(\NN)
\]
is contained in $f$, its spectral radius is $\theta$ and $\mu$ is an eigenvector of $M$ to the eigenvalue $\theta$. Writing $f=(f_1, \ldots, f_{2n})$, we now prove that $\deg_\mu(f_j)= \theta \mu_j$ for each $j\in \{1,\ldots,2n\}$, and compute the $\mu$-homogeneous part $g_j$ of $f_j$ of degree $\theta\mu_j$:
\begin{enumerate}[leftmargin=*]
\item
 For each $j\in \{1,\ldots,n\}$, we have $\deg_\mu(x_{n+j})=v_j$ and $\deg_\mu(\prod_{i=1}^n x_i^{a_{j,i}})=\sum_{i=1}^n \theta a_{j,i}v_i =\theta^2 v_j$, so $\deg_\mu(f_j)=\theta^2 v_j=\theta \mu_j$ and $g_j=\prod_{i=1}^n x_i^{a_{j,i}}$.
 \item
 For each $j\in \{n+1,\ldots,2n\}$ we have $\deg_\mu(f_j)=\deg_\mu(x_{j-n})=\theta v_{j-n}=\theta \mu_j$ and $g_j=f_j$.
\end{enumerate}
This implies that $\deg_\mu(f)=\theta$. As the endomorphism $g=(g_1,\ldots,g_{2n})\in \End(\A^{2n})$ is monomial, it satisfies $g^r\not=0$ for each $r\ge 1$ (and moreover each component of $g^r$ is not zero). This implies that $f$  is $\mu$-algebraically stable and that $\lambda(f)=\theta$ (see Proposition~\ref{Prop:Dynamicaldegreeestimate}).
\end{proof}

\begin{proposition}\label{Prop:PerronQuad}
Let $\lambda\in \r$ be a weak Perron number that is a quadratic integer, and let $x^2-ax-b$ be its minimal polynomial, with $a,b\in \Z$. We then have $a\ge 0$ and the following hold:
\begin{enumerate}[leftmargin=*]
\item\label{Perronquad1}
If $b\ge 0$, then $\lambda$ is the dynamical degree of the shift-like automorphism  \[(x_3+x_1^ax_2^b,x_1,x_2)\in \Aut(\A^3)\,.\]
\item\label{Perronquad2}
If $b<0$, then $\lambda$ is not the dynamical degree of an affine-triangular automorphism of $\A^3$, but is the dynamical degree of a permutation-triangular 
automorphism of $\A^4$ of the form~\eqref{Lemm:Monomialauto_1} in Lemma~$\ref{Lemm:Monomialauto}$. 
\end{enumerate}
\end{proposition}

\begin{proof}
Let us write $x^2-ax-b=(x-\lambda)(x-\mu)$ for some $\mu\in \r$. Note that $\mu\not=\lambda$, as otherwise $\lambda^2\in \Z$ and $2\lambda\in \Z$ would imply that $\lambda\in \Z$, impossible as $\lambda$ is a quadratic integer. Since $\lambda$ is a weak-Perron number, we have $\lambda \geq 1$ and $-\lambda\le \mu< \lambda$. In particular, $a=\lambda+\mu \ge 0$. 
As $x^2 -ax -b$ is irreducible and has 
a real root by assumption, the discriminant is $a^2+4b \geq 1$.

If $b\ge 0$, Assertion~\ref{Perronquad1} follows from Lemma~\ref{Lemm:TwoBasicExamplesThatwork} (and also from Proposition~\ref{Prop:PermutationElementary}).

Suppose now that $b< 0$. As $\lambda\mu=-b$, this implies that $\mu>0$, so $\lambda$ is not a Handelman number (Lemma~\ref{Lemm:Handelman}) and thus is not the dynamical degree of an affine-triangular automorphism of $\A^3$ (Proposition~\ref{Prop:AffineTriangularA3}). 
 It is now enough to show that
\[
	f=(x_3+x_1^{\alpha}x_2,x_4+x_1^{\alpha(a-\alpha)+b}x_2^{a-\alpha},x_1,x_2)\in \Aut(\A^4)
\]
is a permutation-triangular
automorphism with dynamical degree $\lambda(f)=\lambda$.

Firstly, we prove that $f$ is a permutation-triangular 
automorphism of $\A^4$ by showing 
that the exponents are non-negative. As $a\ge 0$, the numbers 
$\alpha= \lfloor a/2 \rfloor$ and $a-\alpha$ are non-negative integers, 
so we only need to see that $\alpha(a-\alpha)+b\ge 0$. 
Since $a^2+4b \geq 1$ we get in case $a$ is even, that $\alpha(a-\alpha)+b=\alpha^2+b=(a^2+4b)/4>0$ and 
in case $a$ is odd, that $\alpha=(a-1)/2$, so 
$\alpha(a-\alpha)+b=((a-1)/2) \cdot ((a+1)/2)+b=(a^2+4b-1)/4 \ge 0$.

Secondly, the matrix 
\[
	A=\begin{pmatrix} 
			\alpha & 1 \\ 
			\alpha(a-\alpha)+b & a-\alpha
		\end{pmatrix}
		\in \Mat_2(\NN)
\]
has characteristic polynomial $x^2-ax-b$ and thus spectral radius $\rho(A)=\lambda$.
As $x^2-ax-b$ is irreducible by assumption, 
it follows that $A$ is an irreducible matrix.
Moreover, as $b \leq -1$ and as $x^2 -ax -b$ has a real root, we get $a \neq 0$, hence $a \geq 1$.
Since $a^2 + 4b \geq 1$, we get $\lambda = (a + \sqrt{a^2 + 4b})/2 \geq 1$. Now, if $\lambda = 1$, then 
$1 \le a \le 2$ and thus $a^2 + 4b \leq 0$ (as $b \leq -1$), contradiction. 
Thus $\lambda > 1$ and we can apply Lemma~\ref{Lemm:Monomialauto} 
and get that the dynamical degree of $f$ is $\lambda(f)=\rho(A)=\lambda$.
\end{proof}

\begin{proof}[Proof of Theorem~$\ref{Theorem:PerronAuto}$]
Let  $\lambda \ge 1$ be a weak-Perron number. By Theorem~\ref{LindTheorem}, 
$\lambda$ is the spectral radius of an irreducible square matrix with non-negative integral coefficients.  Lemma~\ref{Lemm:Monomialauto} then shows that $\lambda$ is the dynamical degree of an affine-triangular automorphism of $\A^{n}$ for some integer $n$. 
We denote by $n_0$ the least possible such $n$.

If $\lambda=1$, then $n_0=1$, by taking the identity.

If $\lambda>1$ is an integer, then $n_0\ge 2$, as every automorphism of $\A^1$ is affine and thus has dynamical degree $1$. Moreover, $n_0=2$ as  $f=(x_1^\lambda+x_2,x_1)$ has dynamical degree equal to $\lambda$ 
($f$ is $\mu$-algebraic stable for $\mu=(1,0)$ and $\deg_{\mu}(f) = \lambda$).

If $\lambda$ is not an integer, then $n_0\ge 3$, as the dynamical degree of every automorphism of $\A^2$ is an integer (Corollary~\ref{Cor:Jung2}). If $\lambda$ is a quadratic integer, the minimal polynomial of $\lambda$ is equal to $x^2-ax-b$ with $a\ge 0$ and $b\in \Z$ (Proposition~\ref{Prop:PerronQuad}). If the conjugate of  $\lambda$ is negative, we have $b>0$, so $n_0=3$ by Proposition~\ref{Prop:PerronQuad}\ref{Perronquad1}. If the conjugate of  $\lambda$ is positive, we have $b<0$, so $n_0=4$ by Proposition~\ref{Prop:PerronQuad}\ref{Perronquad2}.
\end{proof}

To complement Theorem~$\ref{Theorem:PerronAuto}$, we now give a family of examples of quadratic integers that do not arise as dynamical degrees of affine-triangular automorphisms of $\Aut(\A^3)$ but which arise as dynamical degrees of some other automorphisms of $\A^3$.
\begin{lemma}\label{Lemm:Examplerst}
For all integers $r, s, t\ge 1$, the dynamical degree of the automorphism
\[
	f=(y+x^rz^t,z,x+z^s(y+x^rz^t))\in \Aut(\A^3)
\]
is the biggest root of $x^2-ax+b\in \RR[x]$, with $a=r+s+t$, 
$b=rs$ and satisfies $\lambda(f)>s+1$. In particular, if $\lambda(f)$ is not an integer, it is not the dynamical degree of an affine-triangular automorphism of $\A^3$, so $f$ is not conjugate to an affine-triangular automorphism of $\A^3$.
\end{lemma}

\begin{proof} 
Let $\theta$ be the biggest root of $P(x)=x^2-ax+b =(x-r)(x-s)-tx\in \RR[x]$
As $P(s+1)=(s+1-r)-t(s+1)=(s+1)(1-t)-r<0$, we find that $\theta>s+1$. 
In particular, $\mu=(\theta-s,1,\theta)\in \RR_{\ge 1}$.

We compute $\deg_\mu(x^rz^{s+t})=r(\theta-s)+(s+t)\theta=(r+s+t)\theta -rs=\theta^2$ and $\deg_\mu(x^rz^t)=\theta^2-s\theta=\theta(\theta-s)$. 
This gives $\deg_\mu(f)=\theta$, with $\mu$-leading part $g=(x^rz^t,z,x^rz^{s+t})$. Hence, $\lambda(f)=\theta$ by Proposition~\ref{Prop:Dynamicaldegreeestimate}. 

If $\theta$ is not an integer, the other root of $P(x)$ is positive, so 
$\theta$ is not the dynamical  degree of an affine-triangular automorphism of $\A^3$ (Theorem~\ref{Theorem:PerronAuto}). This implies that $f$ is not conjugate to an affine-triangular automorphism of $\A^3$.
\end{proof}

\begin{example}
We now apply Lemma~\ref{Lemm:Examplerst} to small values of $r,s,t$, and find
some examples of automorphisms $f=(y+x^rz^t,z,x+z^s(y+x^rz^t))\in \Aut(\A^3)$ whose dynamical degree $\lambda(f)$ is not the one of an affine-triangular automorphism of $\A^3$. We give below all examples of $\lambda(f)\le 5$ given by  Lemma~\ref{Lemm:Examplerst}. Exchanging $r$ and $s$ does not change the value of $\lambda(f)$, so we will assume that $r\le s\le 3$.
\[\begin{array}{|ccc|c|c|cc}
\hline
 r & s & t & f & \lambda(f)\\
\hline
 1 &1&1& (y+xz,z,x+z(y+xz)) & (3+\sqrt{5})/2\\
\hline
 1 &1&2& (y+xz^2,z,x+z(y+xz^2)) & 2+\sqrt{3}\\
\hline
 1 &1&3& (y+xz^3,z,x+z(y+xz^3)) & (5+\sqrt{21})/2\\
\hline
 1 &2&1& (y+xz,z,x+z^2(y+xz)) & 2+\sqrt{2}\\
\hline
 1 &2&2& (y+xz^2,z,x+z^2(y+xz^2)) & (5+\sqrt{17})/2\\
\hline
 1 &3&1& (y+xz,z,x+z^3(y+xz)) & (5+\sqrt{13})/2\\
\hline
 2 &3&1& (y+x^2z,z,x+z^3(y+x^2z)) & 3+\sqrt{3}\\
\hline
 \end{array}\]
\end{example}
\begin{remark}\label{Rem:WeakPerronSmall}
Let $\lambda$ be a weak-Perron number that is a quadratic integer.

By Theorem~\ref{Theorem:PerronAuto},  $\lambda$ is the dynamical degree of an affine-triangular automorphism of $\A^4$ but is the dynamical degree of an affine-triangular automorphism of $\A^3$ if and only if its conjugate $\lambda'$ is negative. If $\lambda'>0$, then one can 
ask if $\lambda$ is the dynamical degree of an automorphism of $\A^3$ (which would then necessarily be not conjugate to an affine-triangular automorphism). 
Writing $x^2-ax+b$ the minimal polynomial of $\lambda$, with $a,b$ positive integers, Lemma~\ref{Lemm:Examplerst} shows that this is indeed true if one can write $b=rs$ with $r,s\ge 1$ and $a>r+s$. 
In particular, this holds if $b\le 4$, as $a^2-4b>0$, so $a>2\sqrt{b}$. 
If $b=5$, then $a\ge 5$ (as $a > 2 \sqrt{b}$), 
and  Lemma~\ref{Lemm:Examplerst} applies as soon as $a\ge 6$. The case where $a=b=5$ corresponds to $\lambda=(5+\sqrt{5})/2$, which is then the ``simplest'' weak-Perron quadratic integer that is not covered 
by Theorem~\ref{Theorem:PerronAuto} or Lemma~\ref{Lemm:Examplerst}.
\end{remark}

According to Remark~\ref{Rem:WeakPerronSmall}, it seems natural to ask if every quadric weak-Perron number  is the dynamical degree of an automorphism of $\A^3$.
A first intriguing case concerns the following question, which was in fact already asked to us by Jean-Philippe Furter and Pierre-Marie Poloni:
\begin{question}
Is $(5+\sqrt{5})/2$ the dynamical degree of an automorphism of $\A^3$?
\end{question}

\par\bigskip
\renewcommand{\MR}[1]{}
\bibliographystyle{amsalpha}
\bibliography{BIB}

\providecommand{\bysame}{\leavevmode\hbox to3em{\hrulefill}\thinspace}
\providecommand{\MR}{\relax\ifhmode\unskip\space\fi MR }
\providecommand{\MRhref}[2]{%
  \href{http://www.ams.org/mathscinet-getitem?mr=#1}{#2}
}
\providecommand{\href}[2]{#2}
\begin{thebibliography}{Mae01b}

\bibitem[Bas97]{Bassino}
Fr\'{e}d\'{e}rique Bassino, \emph{Nonnegative companion matrices and
  star-height of {${\bf N}$}-rational series}, Theoret. Comput. Sci.
  \textbf{180} (1997), no.~1-2, 61--80. \MR{1453859}

\bibitem[BFs00]{BonifantFornaess}
Araceli~M. Bonifant and John~Erik Forn\ae~ss, \emph{Growth of degree for
  iterates of rational maps in several variables}, Indiana Univ. Math. J.
  \textbf{49} (2000), no.~2, 751--778. \MR{1793690}

\bibitem[Bis08]{Bisi}
Cinzia Bisi, \emph{On commuting polynomial automorphisms of {$\Bbb C^k,\
  k\geq3$}}, Math. Z. \textbf{258} (2008), no.~4, 875--891. \MR{2369061}

\bibitem[Bla16]{Blanc2016}
J\'{e}r\'{e}my Blanc, \emph{Conjugacy classes of special automorphisms of the
  affine spaces}, Algebra Number Theory \textbf{10} (2016), no.~5, 939--967.
  \MR{3531358}

\bibitem[BP98]{BedfordPambuccian}
Eric Bedford and Victoria Pambuccian, \emph{Dynamics of shift-like polynomial
  diffeomorphisms of {${\bf C}^N$}}, Conform. Geom. Dyn. \textbf{2} (1998),
  45--55. \MR{1624646}

\bibitem[Bru13]{Brunotte}
Horst Brunotte, \emph{Algebraic properties of weak {P}erron numbers}, Tatra Mt.
  Math. Publ. \textbf{56} (2013), 27--33. \MR{3184889}

\bibitem[BV18]{SayaniKaushal}
Sayani Bera and Kaushal Verma, \emph{Some aspects of shift-like automorphisms
  of {$\Bbb C^k$}}, Proc. Indian Acad. Sci. Math. Sci. \textbf{128} (2018),
  no.~1, Art. 10, 48. \MR{3778389}

\bibitem[BvS19]{BlSa2019Automorphisms-of-t}
J{\'e}r{\'e}my Blanc and Immanuel van Santen, \emph{Automorphisms of the affine
  $3$-space of degree $3$}, \url{https://arxiv.org/abs/1912.02144}, 2019.

\bibitem[DF20]{DaFa2020Spectral-interpret}
Nguyen-Bac Dang and Charles Favre, \emph{Spectral interpretations of dynamical
  degrees and applications}, \url{https://arxiv.org/abs/2006.10262}, 2020.

\bibitem[DL18]{DeLe2018Dynamics-of-a-fami}
Julie D\'{e}serti and Martin Leguil, \emph{Dynamics of a family of polynomial
  automorphisms of {$\Bbb C^3$}, a phase transition}, J. Geom. Anal.
  \textbf{28} (2018), no.~1, 190--224. \MR{3745855}

\bibitem[DN11]{DiNg2011Comparison-of-dyna}
Tien-Cuong Dinh and Vi\^{e}t-Anh Nguy\^{e}n, \emph{Comparison of dynamical
  degrees for semi-conjugate meromorphic maps}, Comment. Math. Helv.
  \textbf{86} (2011), no.~4, 817--840. \MR{2851870}

\bibitem[Fek23]{Fe1923Uber-die-Verteilun}
M.~Fekete, \emph{\"{U}ber die {V}erteilung der {W}urzeln bei gewissen
  algebraischen {G}leichungen mit ganzzahligen {K}oeffizienten}, Math. Z.
  \textbf{17} (1923), no.~1, 228--249. \MR{1544613}

\bibitem[FJ07]{FavreJonsson07}
Charles Favre and Mattias Jonsson, \emph{Eigenvaluations}, Ann. Sci. \'{E}cole
  Norm. Sup. (4) \textbf{40} (2007), no.~2, 309--349. \MR{2339287}

\bibitem[FJ11]{FavreJonsson}
\bysame, \emph{Dynamical compactifications of {${\bf C}^2$}}, Ann. of Math. (2)
  \textbf{173} (2011), no.~1, 211--248. \MR{2753603}

\bibitem[FsW98]{FornaessWu}
John~Erik Forn\ae~ss and He~Wu, \emph{Classification of degree {$2$} polynomial
  automorphisms of {${\bf C}^3$}}, Publ. Mat. \textbf{42} (1998), no.~1,
  195--210. \MR{1628170}

\bibitem[Fur99]{Furter99}
Jean-Philippe Furter, \emph{On the degree of iterates of automorphisms of the
  affine plane}, Manuscripta Math. \textbf{98} (1999), no.~2, 183--193.
  \MR{1667603}

\bibitem[FW12]{FavreWulcan}
Charles Favre and Elizabeth Wulcan, \emph{Degree growth of monomial maps and
  {M}c{M}ullen's polytope algebra}, Indiana Univ. Math. J. \textbf{61} (2012),
  no.~2, 493--524. \MR{3043585}

\bibitem[Gan59]{Ga1959The-theory-of-matr}
F.~R. Gantmacher, \emph{The theory of matrices. {V}ols. 1, 2}, Translated by K.
  A. Hirsch, Chelsea Publishing Co., New York, 1959. \MR{0107649}

\bibitem[GS02]{GuedjSibony}
Vincent Guedj and Nessim Sibony, \emph{Dynamics of polynomial automorphisms of
  {$\bold C^k$}}, Ark. Mat. \textbf{40} (2002), no.~2, 207--243. \MR{1948064}

\bibitem[Gue02]{Guedj02}
Vincent Guedj, \emph{Dynamics of polynomial mappings of {$\Bbb C^2$}}, Amer. J.
  Math. \textbf{124} (2002), no.~1, 75--106. \MR{1879000}

\bibitem[Gue04]{Guedj04}
\bysame, \emph{Dynamics of quadratic polynomial mappings of {$\Bbb C^2$}},
  Michigan Math. J. \textbf{52} (2004), no.~3, 627--648. \MR{2097402}

\bibitem[Jun42]{Ju1942Uber-ganze-biratio}
Heinrich W.~E. Jung, \emph{\"{U}ber ganze birationale {T}ransformationen der
  {E}bene}, J. Reine Angew. Math. \textbf{184} (1942), 161--174. \MR{0008915}

\bibitem[JW12]{JonssonWulcan}
Mattias Jonsson and Elizabeth Wulcan, \emph{Canonical heights for plane
  polynomial maps of small topological degree}, Math. Res. Lett. \textbf{19}
  (2012), no.~6, 1207--1217. \MR{3091603}

\bibitem[Lin84]{Lind}
D.~A. Lind, \emph{The entropies of topological {M}arkov shifts and a related
  class of algebraic integers}, Ergodic Theory Dynam. Systems \textbf{4}
  (1984), no.~2, 283--300. \MR{766106}

\bibitem[Lin12]{Lin}
Jan-Li Lin, \emph{Pulling back cohomology classes and dynamical degrees of
  monomial maps}, Bull. Soc. Math. France \textbf{140} (2012), no.~4, 533--549
  (2013). \MR{3059849}

\bibitem[Mae00]{Mae2000}
Kazutoshi Maegawa, \emph{Three dimensional shift-like mappings of dynamical
  degree golden ratio}, Proceedings of the {S}econd {ISAAC} {C}ongress, {V}ol.
  2 ({F}ukuoka, 1999), Int. Soc. Anal. Appl. Comput., vol.~8, Kluwer Acad.
  Publ., Dordrecht, 2000, pp.~1057--1062. \MR{1940893}

\bibitem[Mae01a]{Maequadratic}
\bysame, \emph{Classification of quadratic polynomial automorphisms of {${\Bbb
  C}^3$} from a dynamical point of view}, Indiana Univ. Math. J. \textbf{50}
  (2001), no.~2, 935--951. \MR{1864065}

\bibitem[Mae01b]{Maegoldenratio}
\bysame, \emph{Quadratic polynomial automorphisms of dynamical degree golden
  ratio of {$\Bbb C^3$}}, Ergodic Theory Dynam. Systems \textbf{21} (2001),
  no.~3, 823--832. \MR{1836434}

\bibitem[Mat89]{Ma1989Commutative-ring-t}
Hideyuki Matsumura, \emph{Commutative ring theory}, second ed., Cambridge
  Studies in Advanced Mathematics, vol.~8, Cambridge University Press,
  Cambridge, 1989, Translated from the Japanese by M. Reid. \MR{1011461}

\bibitem[MO91]{MGO}
Gary~H. Meisters and Czes\l~aw Olech, \emph{Strong nilpotence holds in
  dimensions up to five only}, Linear and Multilinear Algebra \textbf{30}
  (1991), no.~4, 231--255. \MR{1129181}

\bibitem[Ost73]{Os1973Solution-of-equati}
A.~M. Ostrowski, \emph{Solution of equations in {E}uclidean and {B}anach
  spaces}, Academic Press [A Subsidiary of Harcourt Brace Jovanovich,
  Publishers], New York-London, 1973, Third edition of {{\i}t Solution of
  equations and systems of equations}, Pure and Applied Mathematics, Vol. 9.
  \MR{0359306}

\bibitem[Sch97]{Schinzel}
Andrzej Schinzel, \emph{A class of algebraic numbers}, Tatra Mt. Math. Publ.
  \textbf{11} (1997), 35--42, Number theory (Liptovsk\'{y} J\'{a}n, 1995).
  \MR{1475503}

\bibitem[Sib99]{Sibony1999}
Nessim Sibony, \emph{Dynamique des applications rationnelles de {$\bold P^k$}},
  Dynamique et g\'{e}om\'{e}trie complexes ({L}yon, 1997), Panor. Synth\`eses,
  vol.~8, Soc. Math. France, Paris, 1999, pp.~ix--x, xi--xii, 97--185.
  \MR{1760844}

\bibitem[Ste97]{St1997Probability-theory}
J.~Michael Steele, \emph{Probability theory and combinatorial optimization},
  CBMS-NSF Regional Conference Series in Applied Mathematics, vol.~69, Society
  for Industrial and Applied Mathematics (SIAM), Philadelphia, PA, 1997.
  \MR{1422018}

\bibitem[{Str}86]{St1986A-source-book-in-m}
D.~J. {Struik}, \emph{{A source book in mathematics, 1200-1800. (Reprint of the
  1969 ed.)}}, {Princeton Paperbacks. Princeton, N. J.: Princeton University
  Press. XVI, 427 p. (1986).}, 1986 (English).

\bibitem[Sun14]{Sun}
Xiaosong Sun, \emph{Classification of quadratic homogeneous automorphisms in
  dimension five}, Comm. Algebra \textbf{42} (2014), no.~7, 2821--2840.
  \MR{3178045}

\bibitem[Ued04]{Ueda}
Tetsuo Ueda, \emph{Fixed points of polynomial automorphisms of {$\bold C^n$}},
  Complex analysis in several variables---{M}emorial {C}onference of {K}iyoshi
  {O}ka's {C}entennial {B}irthday, Adv. Stud. Pure Math., vol.~42, Math. Soc.
  Japan, Tokyo, 2004, pp.~319--324. \MR{2087065}

\bibitem[vdE00]{Es2000Polynomial-automor}
Arno van~den Essen, \emph{Polynomial automorphisms and the {J}acobian
  conjecture}, Progress in Mathematics, vol. 190, Birkh\"{a}user Verlag, Basel,
  2000. \MR{1790619}

\bibitem[vdK53]{Ku1953On-polynomial-ring}
W.~van~der Kulk, \emph{On polynomial rings in two variables}, Nieuw Arch.
  Wiskunde (3) \textbf{1} (1953), 33--41. \MR{0054574}

\bibitem[Xie17]{XieMordellLang}
Junyi Xie, \emph{The dynamical {M}ordell-{L}ang conjecture for polynomial
  endomorphisms of the affine plane}, Ast\'{e}risque (2017), no.~394, vi+110.
  \MR{3758955}

\end{thebibliography}

\end{document}